\documentclass[a4paper]{article}
\usepackage{amsmath}
\usepackage{amssymb}
\usepackage{amsthm}
\usepackage[pdftex]{graphicx}
\usepackage{subcaption}
\usepackage{epsfig}
\usepackage{epstopdf}
\usepackage{mathtools}
\usepackage{relsize}
\usepackage{xr-hyper}
\usepackage{hyperref}
\hypersetup{
    colorlinks=true,     % false: boxed links; true: colored links
    allcolors=blue,     % all colors
}
\usepackage[numbers, sort, comma, square]{natbib}
\usepackage{enumitem}
\usepackage{algorithm,algpseudocode}
\usepackage{fullpage}
\usepackage{setspace}
\usepackage{caption}
\usepackage{xspace}
\usepackage{dsfont}
\usepackage{authblk}
\usepackage{mdframed}
\usepackage{scalerel}

\newcounter{mycounter}

\mdfdefinestyle{MyFrame}{%
    linecolor=black,
    outerlinewidth=2pt,
    innertopmargin=4pt,
    innerbottommargin=4pt,
    innerrightmargin=4pt,
    innerleftmargin=4pt,
        leftmargin = 4pt,
        rightmargin = 4pt
        }

\algnewcommand{\IfThenElse}[3]{% \IfThenElse{<if>}{<then>}{<else>}
  \State \algorithmicif\ #1\ \algorithmicthen\ #2\ ; \algorithmicelse\ #3.}

\newcommand{\Set}[2]{\left\lbrace #1 : #2 \right\rbrace}

\DeclarePairedDelimiter{\ceil}{\lceil}{\rceil}
\DeclarePairedDelimiter\abs{\lvert}{\rvert}%
\DeclarePairedDelimiter\norm{\lVert}{\rVert}%
\DeclarePairedDelimiterX{\inner}[2]{\langle}{\rangle}{#1, #2}

\DeclareMathOperator*{\argmin}{arg\,min}

\newcommand{\exclude}[1]{}
\newcommand{\tr}[1]{\ensuremath{{#1}^\text{T}}}
\newcommand{\uset}[2]{\ensuremath{\underset{#1}{#2}}}
\newcommand{\prob}[1]{\mathbb{P}\left\lbrace{#1}\right\rbrace}
\newcommand{\bexpect}[1]{\mathbb{E}\big[{#1}\big]}
\newcommand{\Bexpect}[1]{\mathbb{E}\Big[{#1}\Big]}

\newcommand{\1}[1]{\mathds{1}\left[#1\right]}

\newcommand{\cl}[1]{\textup{cl}\left(#1\right)}
\newcommand{\conv}[1]{\textup{co}#1}
\newcommand{\median}[1]{\textup{median}\left\lbrace{#1}\right\rbrace}
\newcommand{\diam}[1]{\textup{diam}\left(#1\right)}
\newcommand{\proj}[2]{\textup{proj}\left( #1 , #2 \right)}

\newcommand{\ph}{\hat{p}}

\newcommand{\xv}{X_{\nu}}

\newcommand{\A}{\mathcal{A}}
\newcommand{\I}{\mathcal{I}}
\newcommand{\J}{\mathcal{J}}
\newcommand{\N}{\mathbb{N}}
\renewcommand{\P}{\mathcal{P}}
\newcommand{\Pb}{\mathbb{P}}
\newcommand{\R}{\mathbb{R}}
\newcommand{\infeasible}{\texttt{infeas}}
\newcommand{\tle}{\texttt{tle}}
\newcommand{\halmos}{}

\newtheorem{assumption}{Assumption}
\newtheorem{casestudy}{Case study}

\let\oldassumption\assumption
\renewcommand{\assumption}{\oldassumption\normalfont}

\let\oldcasestudy\casestudy
\renewcommand{\casestudy}{\oldcasestudy\normalfont}

\newtheorem{theorem}{Theorem}
\newtheorem{lemma}{Lemma}
\newtheorem{remark}{Remark}
\newtheorem{definition}{Definition}
\newtheorem{corollary}{Corollary}
\newtheorem{proposition}{Proposition}
\newtheorem{example}{Example}

\let\oldtheorem\theorem
\renewcommand{\theorem}{\oldtheorem\normalfont}

\let\oldlemma\lemma
\renewcommand{\lemma}{\oldlemma\normalfont}

\let\oldremark\remark
\renewcommand{\remark}{\oldremark\normalfont}

\let\olddefinition\definition
\renewcommand{\definition}{\olddefinition\normalfont}

\let\oldcorollary\corollary
\renewcommand{\corollary}{\oldcorollary\normalfont}

\let\oldproposition\proposition
\renewcommand{\proposition}{\oldproposition\normalfont}

\let\oldexample\example
\renewcommand{\example}{\oldexample\normalfont}

\providecommand{\keywords}[1]{{\small \textbf{Key words:} #1}}

\title{A stochastic approximation method for approximating the efficient frontier of chance-constrained nonlinear programs\thanks{This research is supported by the Department of Energy, Office of Science, Office of Advanced Scientific Computing Research, Applied Mathematics program under Contract Number DE-AC02-06CH11347.}}

\author[1]{Rohit Kannan}
\author[2]{James Luedtke}
\affil[1]{Wisconsin Institute for Discovery, University of Wisconsin-Madison, Madison, WI, USA. \protect\\ E-mail: rohit.kannan@wisc.edu}
\affil[2]{Department of Industrial \& Systems Engineering and Wisconsin Institute for Discovery, \protect\\ University of Wisconsin-Madison, Madison, WI, USA. E-mail: jim.luedtke@wisc.edu}

\date{\hspace*{0.87in}Version 3 (this document): May 28, 2020 \\[0.05in] Version 2: November 8, 2019 \\[0.05in] \hspace*{0.07in}Version 1: December 17, 2018}

\begin{document}
%%%%%%%%%%%%%%%%

\maketitle

\begin{abstract}
We propose a stochastic approximation method for approximating the efficient frontier of chance-constrained nonlinear programs.
Our approach is based on a bi-objective viewpoint of chance-constrained programs that seeks solutions on the efficient frontier of optimal objective value versus risk of constraint violation.
To this end, we construct a reformulated problem whose objective is to minimize the probability of constraints violation subject to deterministic convex constraints (which includes a bound on the objective function value).
We adapt existing smoothing-based approaches for chance-constrained problems to derive a convergent sequence of smooth approximations of our reformulated problem, and apply a projected stochastic subgradient algorithm to solve it.
%In order to be able to apply a projected stochastic subgradient algorithm to solve our reformulation with the probabilistic objective, we adapt existing smoothing-based approaches for chance-constrained problems to derive a convergent sequence of smooth approximations of our reformulated problem.
In contrast with exterior sampling-based approaches (such as sample average approximation) that approximate the original chance-constrained program with one having finite support, our proposal converges to stationary solutions of a smooth approximation of the original problem, thereby avoiding poor local solutions that may be an artefact of a fixed sample.
Our proposal also includes a tailored implementation of the smoothing-based approach that chooses key algorithmic parameters based on problem data.
Computational results on four test problems from the literature indicate that our proposed approach can 
efficiently determine good approximations of the efficient frontier. \\[0.05in]
%We also present a bisection approach for solving chance-constrained programs with a prespecified risk level.
\keywords{stochastic approximation, chance constraints, efficient frontier, stochastic subgradient}
\end{abstract}

\section{Introduction}
\label{sec:introduction}

Consider the following chance-constrained nonlinear program (NLP):
\begin{alignat}{2}
\nu^*_{\alpha} \:  := \: &\uset{x \in X}{\min} \:\: && f(x) \tag{CCP} \label{eqn:ccp} \\
&\:\: \text{s.t.} && \prob{g(x,\xi) \leq 0} \geq 1 - \alpha, \nonumber
\end{alignat}
where $X \subset \R^n$ is a nonempty closed convex set, $\xi$ is a random vector with probability distribution $\P$ supported on $\Xi \subset \R^d$, $f:\R^n \to \R$ is a continuous quasiconvex function, $g:\R^n \times \R^d \to \R^{m}$ is continuously differentiable, 
%with a locally Lipschitz continuous gradient, 
and the constant $\alpha \in (0,1)$ is a user-defined acceptable level of constraint violation. 
%and for each $x \in \R^n$, we write $\prob{g(x,\xi) \leq 0}$ to denote the probability with which the random constraint $g(x,\xi) \leq 0$ is satisfied.
%Further assumptions on the set $X$ and the constraint functions $g$ will be made in the ensuing sections.
We explore a stochastic approximation algorithm for approximating the efficient frontier of optimal objective
value~$\nu^*_{\alpha}$ for varying levels of the risk tolerance $\alpha$ in~\eqref{eqn:ccp},
i.e., to find a collection of solutions which have varying probability of constraint violation, and
where for each such solution the objective function value is (approximately) minimal among solutions having that
probability of constraint violation or smaller (see also \citep{rengarajan2009estimating,luedtke2014branch}).
%Apart from making a few mild assumptions on the data defining the above problem (see Section~\ref{sec:smoothcon}), most of our results only assume that it is possible to generate i.i.d.~samples of $\xi$ from $\P$ (our results regarding `convergence of stationary solutions' additionally require $\xi$ to be a continuous random variable satisfying additional distributional assumptions). 
Most of our results only require mild assumptions on the data defining the above problem, such as regularity conditions
on the functions $g$ and the ability to draw independent and identically distributed (i.i.d.) samples of $\xi$ from $\P$, and do not assume Gaussianity of the distribution $\P$ or make structural assumptions on the functions~$g$ (see Section~\ref{sec:smoothapp}). 
%In particular, we will not impose restrictive assumptions such as Gaussianity of the distribution $\P$ or make structural assumptions on the random constraint functions~$g$.
Our result regarding `convergence to stationary solutions' of~\eqref{eqn:ccp} additionally requires $\xi$ to be a
continuous random variable satisfying some mild distributional assumptions.
%\footnote{We believe that this result can be extended to a more general setting (see Section~\ref{sec:smoothapp}), but a proof eludes us.}.
%the set $X$, the objective function $f$, the constraint functions $g$, and the probability distribution $\P$ in the ensuing sections, 
We note that Problem~\eqref{eqn:ccp} can model joint nonconvex chance constraints, nonconvex deterministic constraints (by incorporating them in $g$), and even some recourse structure (by defining $g$ through the solution of auxiliary optimization problems, see Appendix~\ref{app:recourse}). 
%Section~\ref{subsec:limitapp} considers an extension for the case when Problem~\eqref{eqn:ccp} contains multiple sets of joint chance constraints.

Chance-constrained programming was introduced as a modeling framework for optimization under uncertainty by~\citet{charnes1958cost}, and was soon after generalized to the joint chance-constrained case by~\citet{miller1965chance} and to the nonlinear case by~\citet{prekopa1970probabilistic}.
%Pr\'{e}kopa~\citep{prekopa1970probabilistic,prekopa2013stochastic} pioneered the investigation of properties of chance-constrained (convex) nonlinear programs, established sufficient conditions under which they can be reformulated as deterministic convex programs, and developed nonlinear programming techniques for their solution.
Apart from a few known tractable cases, e.g., see~\citet{prekopa2013stochastic} and~\citet{lagoa2005probabilistically},
solving chance-constrained NLPs to optimality is in general hard since the feasible region is not guaranteed to be
convex  and evaluating feasibility of the probabilistic constraint involves multi-dimensional integration.
%Despite continued efforts towards delineating situations in which chance-constrained programs can be solved efficiently~\citep{prekopa2013stochastic,lagoa2005probabilistically,henrion2008convexity,van2017eventual}, it is anticipated that solving chance-constrained NLPs is in general hard since the feasible region is not guaranteed to be convex (a notable case is when the underlying deterministic problem is itself nonconvex) and even merely evaluating the probabilistic constraint involves multi-dimensional integration.
Motivated by a diverse array of applications~\citep{calafiore2011research,gotzes2016quantification,li2008chance,
zhang2011chance}, many numerical approaches for chance-constrained NLPs attempt to determine good-quality feasible
solutions in reasonable computation times (see Section~\ref{sec:existapp} for a brief review).

%The last decade has also seen a surge of interest in another subarea of stochastic optimization that encompasses optimization problems with nonconvex expected value objectives and simple 
%(deterministic) 
%convex constraints, primarily inspired by applications in machine learning.
%This has resulted in a flurry of advances, especially in the development and analysis of first-order algorithms, e.g.,~\citet{bottou2018optimization,davis2018stochastic,
%drusvyatskiy2016efficiency,duchi2018stochastic,
%ghadimi2016accelerated,ghadimi2016mini,davis2019stochastic}, and~\citet{nemirovski2009robust}, that guarantee finding first-order stationary points of a broad class of such problems in a well-defined sense.

Rather than solving \eqref{eqn:ccp} for a fixed $\alpha$, we are interested in approximating the efficient frontier
between risk of constraint violation and objective value since 
this efficient frontier provides decision-makers with a tool for choosing the parameter $\alpha$ in~\eqref{eqn:ccp}.
An added advantage of estimating the efficient frontier of~\eqref{eqn:ccp} is that we automatically obtain an estimate of the efficient frontier of a family of distributionally robust counterparts~\citep{jiang2016data}.
In this work, we investigate the potential of a stochastic subgradient algorithm~\citep{davis2018stochastic,ghadimi2016mini,nurminskii1973quasigradient} for approximating the efficient frontier of~\eqref{eqn:ccp}. 
While our proposal is more naturally applicable for approximating the efficient frontier, we also present a bisection strategy for solving~\eqref{eqn:ccp} for a fixed risk level $\alpha$.

We begin by observing that the efficient frontier can be recovered by solving the following stochastic optimization problem~\citep{rengarajan2009estimating}:
\begin{alignat}{2}
\label{eqn:sp}
&\uset{x \in X}{\min} \:\: && \prob{g(x,\xi) \not\leq 0} \\
&\:\: \text{s.t.} && f(x) \leq \nu, \nonumber
\end{alignat}
where the above formulation determines the minimum probability of constraints violation given an upper bound $\nu$ on the objective function value. 
The reformulation in~\eqref{eqn:sp} crucially enables the use of stochastic subgradient algorithms for its approximate solution.
Assuming that~\eqref{eqn:ccp} is feasible for risk levels of interest, solving~\eqref{eqn:sp} using
varying values of the bound $\nu$ yields the same efficient frontier as solving~\eqref{eqn:ccp} using varying values of $\alpha$~\citep{rengarajan2009estimating}.
% (this can be argued by noting that two different risk levels cannot correspond to the same optimal objective value on the efficient frontier)~\citep{rengarajan2009estimating}.

Let $\mathds{1}:\R \to \{0,1\}$, defined by $\1{z} = 1$ if $z > 0$ and $\1{z} = 0$ if $z \leq 0$, denote the characteristic function of the set $(0,+\infty)$, and $\max\left[\1{g(x,\xi)}\right]$ denote $\uset{j\in \{1,\cdots,m\}}{\max} \1{g_j(x,\xi)}$.
Note that the following problem is equivalent to~\eqref{eqn:sp}:
\[
\uset{x \in \xv}{\min} \:\: \bexpect{\max\left[\1{g(x,\xi)}\right]},
\]
where $\xv$ denotes the closed convex set $\Set{x \in X}{f(x) \leq \nu}$.
The above formulation is almost in the form that stochastic gradient-type algorithms can solve~\citep{davis2018tame},
but poses two challenges that prevent immediate application of such algorithms: the step function $\1{\cdot}$ is discontinuous and the $\max$ function is nondifferentiable.
We propose to solve a partially-smoothened approximation of the above formulation using the projected stochastic subgradient algorithm of~\citet{nurminskii1973quasigradient} and~\citet{ermol1998stochastic}.
In particular, we use the procedure outlined in~\citet{davis2018stochastic} to solve our approximation (also see the closely-related proposals of~\citet{ermoliev1988stochastic},~\citet{ghadimi2016mini}, and the references in these works).
To enable this, we replace the discontinuous step functions in the above formulation using smooth
%\footnote{We will refer to the approximations to the step function as `smooth' throughout this work, although continuously differentiable approximations with Lipschitz continuous gradients will suffice for the purposes of our analysis. Note that we do not require the approximations $\phi_j$ to overestimate the step function (cf.~\citep{geletu2017inner,shan2014smoothing}).} 
approximations $\phi_j: \R \to \R$, $j \in \{1,\cdots,m\}$, to obtain the following approximation to~\eqref{eqn:sp}:
\begin{alignat}{2}
\label{eqn:app}
&\uset{x \in \xv}{\min} \: && \bexpect{\max\left[\phi\left(g(x,\xi)\right)\right]},  \tag{$\text{APP}_{\nu}$}
\end{alignat}
where $\max\left[\phi\left(g(x,\xi)\right)\right]$ is shorthand for the composite function $\max\{\phi_1\left(g_1(x,\xi)\right),\cdots,\phi_m\left(g_m(x,\xi)\right)\}$.
Section~\ref{subsec:exmsmoothapprox} presents some options for the smooth approximations $\phi_j$.

Our first primary contribution is to analyze the convergence of optimal solutions and stationary points
of \eqref{eqn:app} to those of problem \eqref{eqn:sp}.
This analysis closely follows previous work on the use of similar approximations for the
chance-constrained problem \eqref{eqn:ccp} (cf.~\citet{norkin1993analysis,hong2011sequential,shan2014smoothing,shan2016convergence,
hu2013smooth}, and~\citet{geletu2017inner}). Our analysis is required because: i.~our approximations~\eqref{eqn:app} involve a convergent sequence of approximations of the objective function of~\eqref{eqn:sp} rather than a convergent sequence of approximations of the feasible region of~\eqref{eqn:ccp}, ii.~we handle joint chance-constrained NLPs directly without reducing them to the case of individual chance-constrained programs,
%smoothing the $\max$ function, 
and iii.~our computational results rely on a slight generalization of existing frameworks for smooth approximation of
the step function. Related to the first point, \citet{norkin1993analysis} analyzes smooth approximations in the
probability maximization context, but their analysis is restricted to the case when the probability function being
maximized is concave (see
also~\citet[Example~5.6]{ermoliev1995minimization},~\citet[Section~4.4.2]{shapiro2009lectures},~\citet{lepp2009}, and
the references in~\citet{norkin1993analysis} for related work).

Our second primary contribution is to propose a practical implementation of a method based on the projected stochastic subgradient algorithm
of~\citet{nurminskii1973quasigradient} and~\citet{ermol1998stochastic} (we use the specific version proposed in~\citet{davis2018stochastic}) to obtain stationary points for problem
\eqref{eqn:app} for varying levels of $\nu$, and thus obtain an approximation of the efficient frontier.
For each fixed $\nu$ this method yields convergence to a neighborhood of an approximately stationary solution since the objective function of~\eqref{eqn:app} is weakly convex under mild
assumptions~\citep{drusvyatskiy2016efficiency,nurminskii1973quasigradient}. 
%Because we rely on an algorithm that does not guarantee finding globally optimal solutions to~\eqref{eqn:app}, our proposal may not yield the efficient frontier of~\eqref{eqn:ccp}, but may only result in a conservative (`achievable') approximation of it.
%the efficient frontier. 
A similar approach has been proposed in~\citet{norkin1993analysis}. In particular,~\citet{norkin1993analysis} considers the case when the function $g$ is jointly convex with respect to $x$ and $\xi$, uses Steklov-Sobolev averaging for constructing~\eqref{eqn:app} from~\eqref{eqn:sp}, assumes that the negative of the objective of~\eqref{eqn:sp} and the negative of the objective of~\eqref{eqn:app} are `$\alpha$-concave functions' 
%for some $\alpha \in \R$ 
(see Definition~2.3 of~\citet{norkin1993analysis} or Definition~4.7 of~\citet{shapiro2009lectures}), and establishes the rate of convergence of the stochastic quasigradient method of~\citet{ermoliev1988stochastic,nemirovsky1983problem} for solving~\eqref{eqn:app} to global optimality in this setting.
Our contributions differ from those of~\citet{norkin1993analysis} in two main respects: we present theory for the
case when~\eqref{eqn:app} may not be solved to global optimality (see~\citet{lepp2009} for a related approach), and we
design and empirically test a practical implementation of a stochastic approximation algorithm for approximating the
solution of~\eqref{eqn:sp}. 

%Smoothing-based approaches to probability maximization problems have previously been considered by~\citet{norkin1993analysis}, where~\eqref{eqn:sp} was approximated by~\eqref{eqn:app} under certain convexity assumptions on $g$ (also see~\citep[Example~5.6]{ermoliev1995minimization},~\citep[Section~4.4.2]{shapiro2009lectures}, and the references therein).

A major motivation for this study is the fact that sample average approximations~\citep{luedtke2008sample}
of~\eqref{eqn:ccp} may introduce ``spurious local optima'' as a byproduct of sampling as illustrated by the following
example from Section~5.2 of~\citet{curtis2018sequential} (see also the discussion in~\citet{ermoliev1997nonsmooth,ermol1998stochastic}).

\begin{example}
\label{exm:spurious_min}
Consider the instance of~\eqref{eqn:ccp} with $n = 2$, $d = 2$, $X = \R^n$, $\xi_1 \sim U[-12,12]$, $\xi_2 \sim U[-3,3]$, $f(x) = x_2$, and $g(x,\xi) = \frac{x^4_1}{4} - \frac{x^3_1}{3} - x^2_1 + 0.2x_1 - 19.5 + \xi_2 x_1 + \xi_1\xi_2 - x_2$ with $\alpha = 0.05$.
Figure~\ref{fig:spurious_min} plots the (lower) boundary of the feasible region defined by the chance constraint in a neighborhood of the optimal solution $x^* \approx (1.8,-0.1)$ along with the boundaries of the feasible region generated by a ten thousand sample average approximation (SAA) of the chance constraint and the feasible region generated by the smooth approximation of the step function adopted in this work.
The points highlighted by triangles in Figure~\ref{fig:spurious_min} are local minima of the SAA approximation, but are
not near any local minimia of the true problem. We refer to these as spurious local minima.
%This example shows that exterior sampling approaches can induce spurious local minima as a byproduct of sampling, and that smoothing-based approaches in conjunction with interior sampling approaches for solving the smooth approximation can mitigate this issue.
\end{example}

\begin{figure}[t]
\begin{center}
\caption{Comparison of the boundary of the feasible region (in a neighborhood of the optimal solution) of a SAA of Example~\ref{exm:spurious_min} (with sample size $10000$) with that of the smooth approximation of the step function adopted in this work (the function $\phi$ is defined in Example~\ref{exm:oursmoothapprox} and the smoothing parameter $\tau = 1.085$).} \label{fig:spurious_min}
\includegraphics[width=0.7\linewidth]{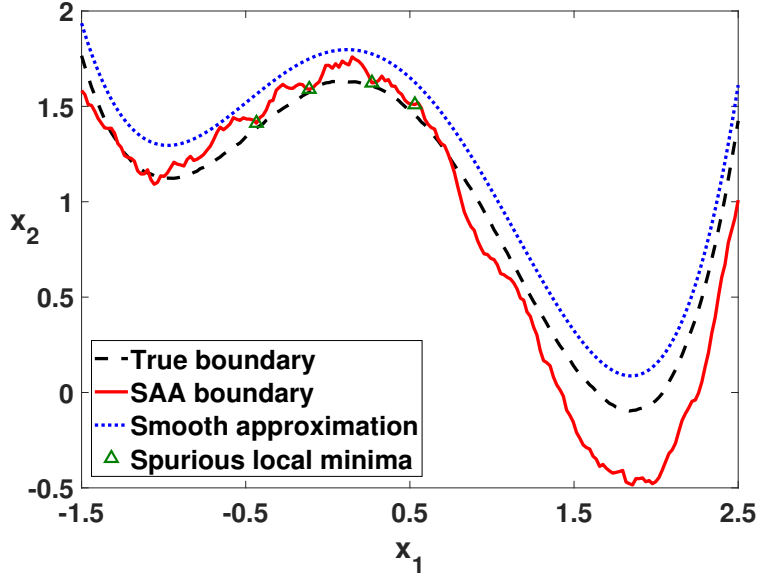}
\end{center}
\end{figure}

%This may unnecessarily complicate the process of obtaining a good approximate solution to~\eqref{eqn:ccp}, especially in cases when $\xi$ is a `well-behaved' continuous random variable that makes the probability function $p(x) := \prob{g(x,\xi) \not\leq 0}$ `well behaved' (even though it might be nonconvex).
Example~\ref{exm:spurious_min} illustrates that obtaining a local optimal solution of an accurate SAA of~\eqref{eqn:ccp} may not necessarily yield a good solution even when $\xi$ is a `well-behaved' continuous random variable that makes the probability function $p(x) := \prob{g(x,\xi) \not\leq 0}$ `well behaved' (even though it might be nonconvex).
%While smoothing the chance constraint~\citep{geletu2017inner} helps mitigate this issue as evidenced by Figure~\ref{fig:spurious_min}, their practical deployment still requires integration within a sample average approximation (SAA) framework. 
While smoothing the chance constraint helps mitigate this issue (cf.\ Figure~\ref{fig:spurious_min}), existing algorithms~\citep{shan2014smoothing,geletu2017inner} still use SAA to solve the resulting smooth approximation. 
By employing a stochastic subgradient algorithm, we do not restrict ourselves to a single fixed batch of samples of the random variables, and can therefore hope to converge to truly locally optimal solutions of~\eqref{eqn:app}. 
%Another motivating factor is that theoretical bounds~\citep{campi2011sampling,luedtke2008sample} on the number of samples required by SAA-based approaches to ensure feasibility for~\eqref{eqn:ccp} are typically very conservative~\citep{henrion2011critical,
%luedtke2008sample,pagnoncelli2009sample,homem2014monte}, and their use results in suboptimal solutions and increased solution times (this issue is alleviated through the use of `tuned SAA' and carefully designed solution approaches, see Section~\ref{sec:computexp}).
Our proposal leverages well-known advantages of stochastic subgradient-type approaches, including a low per-iteration cost, low memory requirement, good scaling with number of random variables and the number of joint chance constraints, and reasonable scaling with the number of decision variables (cf.~the discussion in~\citet{nemirovski2009robust}).

%An alternative to approximating the solution of~\eqref{eqn:sp} using~\eqref{eqn:app} is to solve Problem~\eqref{eqn:sp} directly using zeroth-order smoothing-based stochastic optimization approaches, see~\citet{ermoliev1997nonsmooth},~\citet[Section~5]{ghadimi2016mini}, and~\citet{jin2018minimizing}, for instance. A potential disadvantage of such approaches is that they may require taking a relatively large sample (mini-batch) of $\xi$ at each iteration to make reasonable progress, since they ignore the structure of the objective function of~\eqref{eqn:sp} altogether.
%As an aside, we note that our proposal can also be used for obtaining near-stationary solutions of classes of reliability maximization problems~\citep{prekopa1979optimal,royset2004reliability,jasour2015semidefinite,norkin1993analysis} since such problems can be cast directly in the form~\eqref{eqn:sp}.

This article is organized as follows. 
Section~\ref{sec:existapp} reviews related approaches for solving chance-constrained NLPs.
Section~\ref{sec:smoothapp} establishes consistency of the smoothing approach in the limit of its parameter values, and
discusses some choices for smooth approximation of the step function.
Section~\ref{sec:propalgo} outlines our proposal for approximating the efficient frontier of~\eqref{eqn:ccp} and lists some practical considerations for its implementation.
Section~\ref{sec:computexp} summarizes implementation details and presents computational experiments on four test problems that illustrate the strength of our approach.
We close with a summary of our contributions and some avenues for future work in Section~\ref{sec:conclandfw}.
The appendices present auxiliary algorithms, omitted proofs, and computational results, and provide technical and implementation-related details for problems with recourse structure.

\textbf{Notation}.~We use $e$ to denote a vector of ones of suitable dimension, $\1{\cdot}$ to denote the l.s.c.~step function, $\xv := \Set{x \in X}{f(x) \leq \nu}$, $p(x) := \prob{g(x,\xi) \not\leq 0}$, $B_{\delta}(x) := \Set{z \in \R^n}{\norm{z-x} < \delta}$, where $\norm{\cdot}$ denotes the Euclidean norm, and `a.e.' for the quantifier `almost everywhere' with respect to the probability measure $\Pb$. 
Given functions $h: \R^n \times \R^d \to \R^m$ and $\phi_j:\R \to \R$, $j = 1,\cdots,m$, we write
$\max\left[h(x,\xi)\right]$ to denote $\max\{ h_j(x,\xi) : j\in\{1,\ldots, m\}\}$ and $\phi\left(h(x,\xi)\right)$ to denote the element-wise composition $(\phi_1(h_1(x,\xi)),\cdots,\phi_m(h_m(x,\xi)))$.
Given a set $Z \subset \R^n$, we let $\conv\{Z\}$ denote its convex hull, $I_Z(\cdot)$ denote its indicator function, i.e., $I_Z(z) = +\infty$ if $z \not\in Z$, and zero otherwise, $N_Z(z)$ denote its normal cone at a vector $z \in \R^n$ and $\proj{z}{Z}$ to denote the projection of $z$ onto $Z$ (we abuse notation to write $y = \proj{z}{Z}$ when $Z$ is a nonempty closed convex set), and $\abs{Z}$ denote its cardinality when it is a finite set.
Given a twice differentiable function $f: \R \to \R$, we write $d f(x)$ and $d^2 f(x)$ to denote its first and second derivatives at $x \in \R$.
Given a locally Lipschitz continuous function $f: \R^n \to \R$, we write $\partial f(x)$ to denote its Clarke generalized gradient at a point $x \in \R^n$~\citep{clarke1990optimization} and $\partial_B f(x)$ to denote its $B$-subdifferential (note: $\partial f(x) = \conv{\left\lbrace\partial_B f(x)\right\rbrace}$).
%All expectations are taken with respect to the random variable $\xi$. 
We assume measurability of random functions throughout.

\section{Review of related work}
\label{sec:existapp}

We restrict our attention in this brief review to algorithms 
%for solving chance-constrained programs 
that attempt to generate good-quality \textit{feasible} solutions to chance-constrained \textit{nonlinear} programs.

In the \textit{scenario approximation} approach, a fixed sample $\{\xi^i\}_{i=1}^{N}$ of the random variables from the
distribution $\P$ is taken and the sampled constraints $g(x,\xi^i) \leq 0$, $\forall i \in \{1,\cdots,N\}$, are enforced
in lieu of the chance constraint $\prob{g(x,\xi) \leq 0} \geq 1 - \alpha$. The advantage of this approach is that the scenario problem is a standard NLP that can be solved using off-the-shelf solvers. 
Calafiore, Campi, and Garatti~\citep{calafiore2005uncertain,campi2011sampling} present upper bounds on the sample
size~$N$ required to ensure that the solution of the scenario problem is \textit{feasible} for~\eqref{eqn:ccp} with a
given probability when the functions $g(\cdot,\xi)$ are convex for each fixed $\xi \in \Xi$.
When the functions $g$ do not possess the above convexity property but instead satisfy alternative Lipschitz assumptions,~\citet[Section~2.2.3]{luedtke2008sample} determine an upper bound on~$N$ that provides similar theoretical guarantees for a slightly perturbed version of the scenario problem.
A major limitation of scenario approximation is that it lacks strong guarantees on solution quality~\citep{luedtke2008sample,rengarajan2009estimating}.
A related approach for obtaining feasible solutions to~\eqref{eqn:ccp} constructs convex approximations of the feasible region using conservative convex approximations of the step function~\citep{nemirovski2006convex,rockafellar2000optimization,chen2010cvar}. Such approximations again may result in overly conservative solutions, e.g., see~\citet{hong2011sequential} and~\citet{cao2018sigmoidal}.

There has been a recent surge of interest in iterative smoothing-based NLP approaches in which the chance constraint
$p(x) \leq \alpha$ is replaced by a convergent sequence of smooth approximations.
To illustrate these approaches, consider first the case when we only have an individual chance constraint (i.e., $m =
1$), and suppose $p(x) = \bexpect{\1{g(x,\xi)}} \leq \alpha$ is approximated by $\bexpect{\phi_k\left(g(x,\xi)\right)}
\leq \alpha$, where $\{\phi_k\}$ is a monotonic sequence of smooth approximations converging to the step function.
If each element of $\{\phi_k\}$ is chosen to overestimate the step function, then solving the sequence of smooth
approximations furnishes feasible solutions to~\eqref{eqn:ccp} of improving quality. 
%Due to the general difficulty in
%evaluating the expectation, %(the algorithm of~\citet{lan2016algorithms} can be used to solve problems with convex expectation constraints), 
%the smooth approximation is typically applied to a scenario-based approximation.
Due to the difficulty in
evaluating high-dimensional expectations in general, the smooth approximation is typically applied to a scenario-based approximation.
Joint chance constraints have been accommodated either by replacing the $\max$ function in $\bexpect{\max\left[\phi_k\left(g(x,\xi)\right)\right]}$ with its own convergent sequence of smooth approximations~\citep{hu2013smooth,shan2016convergence}, or by conservatively approximating~\eqref{eqn:ccp} using a Bonferroni-type approximation~\citep{nemirovski2006convex}.

Lepp (see~\citep{lepp2009} and the references therein) proposes a stochastic version of a modified Lagrangian algorithm for solving progressively refined smooth approximations of individual chance-constrained NLPs with a differentiable probability function~$p$, and establishes almost sure convergence to stationary solutions of~\eqref{eqn:ccp}. A key difference between our proposal and Lepp's~\citep{lepp2009} is that the latter requires a diverging sequence of mini-batch sizes to ensure convergence.
\citet{andrieu2007stochastic} propose smoothing and finite difference-based stochastic Arrow-Hurwicz algorithms for
individual chance-constrained NLPs with continuous random variables, and establish local convergence to stationary
solutions under relatively strong assumptions.

\citet{hong2011sequential} propose a sequence of conservative nonsmooth difference-of-convex (DC) approximations of the step function and use it to solve joint chance-constrained convex programs. They establish convergence of any stationary/optimal solution of the DC approximation scheme to the set of stationary/optimal solutions of~\eqref{eqn:ccp} in the limit under certain assumptions.
\citet{hu2013smooth} build on this work by developing a sequence of conservative smooth approximations of joint chance-constrained convex programs (using log-sum-exp approximations of the $\max$ function), and use a sequential convex approximation algorithm to solve their smooth approximation of a SAA of~\eqref{eqn:ccp}.
\citet{shan2014smoothing,shan2016convergence} develop a family of conservative smooth DC approximations of the step function, also employ conservative log-sum-exp approximations of the $\max$ function to construct smooth approximations of~\eqref{eqn:ccp}, and provide similar theoretical guarantees as~\citet{hong2011sequential} under slightly weaker assumptions.
\citet{geletu2017inner} propose a framework for conservative analytic approximations of the step function, provide
similar theoretical guarantees as these works, and illustrate the applicability of their framework for individual chance-constrained NLPs using a class of sigmoid-like smoothing functions. They also propose smooth outer-approximations of~\eqref{eqn:ccp} for generating lower bounds.
%They establish sufficient conditions for the probability function $p$ to be differentiable, and additionally propose smooth lower bounding approximations of~\eqref{eqn:ccp}.
Adam et al.~\citep{adam2016nonlinear,adam2018solving} develop a continuous relaxation of SAAs of~\eqref{eqn:ccp}, and propose to use smooth approximations of the step function that are borrowed from the literature on mathematical programs with complementarity constraints along with Benders' decomposition~\citep{benders1962partitioning} to determine stationary points.
\citet{cao2018sigmoidal} propose a nonsmooth sigmoidal approximation of the step function, and use it to solve a smooth approximation of a SAA of individual chance-constrained NLPs. 
%They propose to initialize their algorithm using the solution of a scenario-based approximation of a conservative convex optimization problem~\citep{rockafellar2000optimization}. 
They propose to initialize their algorithm using the solution of a SAA of a conservative approximation of~\eqref{eqn:ccp} obtained by replacing the step function with a convex overestimate~\citep{rockafellar2000optimization}. 
Finally,~\citet{plw2019} also consider using a smooth
approximation on an SAA of \eqref{eqn:ccp}, analyze the joint convergence with respect to sample size and sequence of
smooth approximations, propose to solve a quantile-based approximation, and devise a trust-region method to solve 
the approximation when applied to a joint chance constraint.

An alternative to smoothing the chance constraint is to solve a SAA problem directly using mixed-integer NLP techniques,
which may be computationally challenging, especially when the constraint functions $g$ are not convex with respect to the decisions $x$.
Several tailored approaches~\citep{luedtke2014branch,van2016inexact}
%,xie2017quantile} 
have been proposed for solving SAAs of chance-constrained convex programs.
\citet{curtis2018sequential} attempt to directly solve a SAA of~\eqref{eqn:ccp} using NLP techniques. They develop an exact penalty function for the SAA, and propose a trust region algorithm that solves quadratic programs with linear cardinality constraints to converge to stationary points of~\eqref{eqn:ccp}.

Another important body of work by~\citet{van2014gradient,van2017sub} establishes (sub)gradient formulae for the
probability function $p$ when the function $g$ possesses special structure and $\P$ is a Gaussian or Gaussian-like
distribution. These approaches employ internal sampling to numerically evaluate (sub)gradients, and can be used within a
NLP framework for computing stationary point solutions to~\eqref{eqn:ccp}.
When Problem~\eqref{eqn:ccp} possesses certain special structures, NLP approaches based on `$p$-efficient points'~\citep{dentcheva2013regularization,van2017probabilistic} can also be employed for its solution.

The recent independent work of~\citet{adam2018machine} also considers a stochastic approximation method for chance-constrained NLPs. They consider the case when the random vector $\xi$ has a discrete distribution with finite (but potentially large) support, rewrite the chance constraint in~\eqref{eqn:ccp} using a quantile-based constraint, 
develop a penalty-based approach to transfer this constraint to the objective, and employ a mini-batch projected stochastic subgradient method to determine an approximately stationary solution to~\eqref{eqn:ccp} for a given risk level.
This work does not include a analysis that shows the method converges to a stationary solution.

%The major distinguishing feature of our proposal from (most of) the above works is that it does not solve a SAA of~\eqref{eqn:ccp} that relies on a fixed sample of the random variables, but instead attempts to solve~\eqref{eqn:sp} using a stochastic approximation algorithm that uses implicit sampling.
%Consequently, our proposed algorithm typically has a low per-iteration computational cost that is well-suited for large-scale instances, and can hope to converge to local optimal solutions of~\eqref{eqn:sp}. 
%This is demonstrated by the numerical experiments in Section~\ref{sec:computexp} which indicate that our approach can find good-quality approximations of the efficient frontier in reasonable solution times.
%Our proposed algorithm is unfortunately not entirely immune to drawbacks; some of them are listed at the end of Section~\ref{sec:propalgo} after introducing it formally.
%providing relevant background information.

\section{The smoothing approach}
\label{sec:smoothapp}

In the first part of this section, we establish conditions under which global/stationary solutions of convergent sequences of
smooth approximations~\eqref{eqn:app} converge to a global/stationary solution of~\eqref{eqn:sp}. We then present some
concrete examples of sequences of smooth approximations that can be accommodated within this framework. 

\subsection{Consistency}
\label{subsec:smoothcon}

We establish sufficient conditions under which solving a sequence of smooth approximations~\eqref{eqn:app} of the stochastic program~\eqref{eqn:sp} yields a point on the efficient frontier of~\eqref{eqn:ccp}.
The techniques in this section follow closely those in previous work that analyzed similar convergence for
\eqref{eqn:ccp} (cf.~\citet{hong2011sequential,shan2014smoothing,shan2016convergence,
hu2013smooth}, and~\citet{geletu2017inner}). We use the following assumptions on Problem~\eqref{eqn:sp}  in our
analysis. Not all results we present require all these assumptions; the required assumptions are included in the statements of our results.

\begin{assumption}
\label{ass:xcompact} 
The convex set $X_{\nu} := \Set{x \in X}{f(x) \leq \nu}$ is nonempty and compact.
\end{assumption}

\begin{assumption}
\label{ass:probzero} 
For each $x \in \xv$, $\prob{\max\left[g(x,\xi)\right] = 0} = 0$.
\end{assumption}

\begin{assumption}
\label{ass:confunc}
\setcounter{mycounter}{\value{assumption}}
The following conditions on the functions $g_j$ hold for each $\xi \in \Xi$ and $j \in \{1,\cdots,m\}$:
\begin{enumerate}[label=\Alph*.,itemsep=0em]
\item The function $g_j(\cdot,\xi)$ is Lipschitz continuous on $\xv$ with a nonnegative measurable Lipschitz constant
$L_{g,j}(\xi)$ satisfying $\bexpect{L^2_{g,j}(\xi)} < +\infty$. %\label{ass:conlipschitz}  
\item The gradient function $\nabla g_j(\cdot,\xi)$ is Lipschitz continuous on $\xv$ with a measurable Lipschitz
constant $L^{'}_{g,j}(\xi) \in \R_+$ satisfying $\bexpect{L^{'}_{g,j}(\xi)} < +\infty$.  %\label{ass:congradlipschitz}  
\item  There exist positive constants $\sigma_{g,j}$ satisfying $\bexpect{\norm{\nabla_x g_j(\bar{x},\xi)}^2} \leq
\sigma^2_{g,j}$, $\forall \bar{x} \in \xv$. %\label{ass:congradvar} 
\end{enumerate}
\end{assumption}

Assumption~\ref{ass:xcompact} is used to ensure that Problem~\eqref{eqn:sp} is well defined (see Lemma~\ref{lem:probconts}).
Assumption~\ref{ass:probzero} guarantees that the family of approximations of the function~$p$ considered converge
pointwise to it on $\xv$ in the limit of their parameter values (see Proposition~\ref{prop:pointwiseconv}). This
assumption, when enforced, essentially restricts $\xi$ to be a continuous random variable (see Lemma~\ref{lem:probconts} for a consequence of this assumption).
Assumption~\ref{ass:confunc} is a standard assumption required for the analysis of stochastic approximation algorithms that serves two purposes: it ensures that the approximations of the function $p$ possess important regularity properties (see Lemma~\ref{lem:problipconst}), and it enables the use of the projected stochastic subgradient algorithm for solving the sequence of approximations~\eqref{eqn:app}.
Note that Assumption~\ref{ass:probzero} is implied by Assumption~\ref{ass:phiandcdfass}B, which we introduce later, and
that Assumption~\ref{ass:confunc}A implies $\bexpect{L_{g,j}(\xi)} < +\infty$.
\begin{lemma}
\label{lem:probconts}
The probability function $p:\R^n \to [0,1]$ is lower semicontinuous. Under Assumption~\ref{ass:probzero}, we additionally have that $p$ is continuous on $\xv$.
\end{lemma}
\begin{proof}
Follows from Theorem~10.1.1 of~\citet{prekopa2013stochastic}.  \halmos
\end{proof}

A natural approach to approximating the solution of Problem~\eqref{eqn:sp}  is to construct a sequence of approximations~\eqref{eqn:app} based on an associated sequence of smooth approximations that converge to the step function (cf.~Section~\ref{sec:existapp}). 
In what follows, we use $\{\phi_k\}$ to denote a sequence of smooth approximations of the step function, where $\phi_k$ corresponds to a vector of approximating functions $(\phi_{k,1},\cdots,\phi_{k,m})$ for the $m$ constraints defined by the function~$g$.
Since we are mainly interested in employing sequences of approximations $\{\phi_k\}$ that depend on associated sequences of smoothing parameters $\{\tau_k\}$ (with $\phi_{k,j}$ `converging' to the step function $\1{\cdot}$ as the smoothing parameter $\tau_{k,j}$ converges to zero), we sometimes write $\phi(\cdot;\tau_{k,j})$, where $\phi:\R \times \R_+ \to \R$, instead of $\phi_{k,j}(\cdot)$ to make this parametric dependence explicit.
For ease of exposition, we make the following (mild) blanket assumptions on each element of the sequence of smoothing functions $\{\phi_k\}$ throughout this work. Section~\ref{subsec:exmsmoothapprox} lists some examples that satisfy these assumptions.
\begin{assumption}
\label{ass:phibasic}
\setcounter{mycounter}{\value{assumption}}
The following conditions hold for each $k \in \N$ and $j \in \{1,\cdots,m\}$:
\begin{enumerate}[label=\Alph*.,itemsep=0em]
\item  The functions $\phi_{k,j}:\R \to \R$ are continuously differentiable. %\label{ass:phismooth} 
\item  Each function $\phi_{k,j}$ is nondecreasing, i.e., $y \geq z \implies \phi_{k,j}(y) \geq \phi_{k,j}(z)$. %\label{ass:phiincreasing} 
\item The functions $\phi_{k,j}$ are equibounded, i.e., there exists a universal constant $M_{\phi} > 0$ (independent of
indices $j$ and $k$) such that $\abs{\phi_{k,j}(y)} \leq M_{\phi}$, $\forall y \in \R$. %\label{ass:phibounded}  
\item Each approximation $\phi_{k,j}(y)$ converges pointwise to the step function $\1{y}$ except possibly at $y = 0$, i.e.,
$\uset{k \to \infty}{\lim} \phi_{k,j}(y) = \1{y}, \: \forall y \in \R \backslash \{0\}$, or, equivalently, $\uset{k \to
\infty}{\lim} \phi(y;\tau_{k,j}) = \1{y}, \: \forall y \in \R \backslash \{0\}$.  %\label{ass:phipointconv}  
\end{enumerate}
\end{assumption}
We say that `the \textit{strong form} of Assumption~\ref{ass:phibasic}' holds if Assumption~\ref{ass:phibasic}D is replaced with the stronger condition of pointwise convergence \textit{everywhere}, i.e.,
$\uset{k \to \infty}{\lim} \phi_{k,j}(y) = \1{y}, \: \forall y \in \R$.
Note that a sequence of conservative smooth approximations of the step function (which overestimate the step function everywhere) cannot satisfy the strong form of Assumption~\ref{ass:phibasic}, whereas sequences of underestimating smooth approximations of the step function may satisfy it (see Section~\ref{subsec:exmsmoothapprox}).
In the rest of this paper, we use $\ph_k(x)$ to denote the approximation $\bexpect{\max\left[\phi_k\left(g(x,\xi)\right)\right]}$ (again, we sometimes write $\ph(x;\tau_k)$ instead of $\ph_k(x)$ to make its dependence on the smoothing parameters $\tau_k$ explicit). Note that $\ph_k:\R^n \to \R$ is continuous under Assumption~\ref{ass:phibasic}, see~\citet[Theorem~7.43]{shapiro2009lectures}.
The following result establishes sufficient conditions under which $\ph_k(x) \to p(x)$ pointwise on $\xv$. 
%which is a minimum requirement for the consistency of the smoothing approach in general.

\begin{proposition}
\label{prop:pointwiseconv}
Suppose Assumptions~\ref{ass:probzero} and~\ref{ass:phibasic} hold, or the strong form of Assumption~\ref{ass:phibasic}
holds. Then
$\uset{k \to \infty}{\lim}\: \ph_k(x) \equiv \uset{k \to \infty}{\lim}\: \ph(x;\tau_k) = p(x)$, $\forall x \in \xv$.
\end{proposition}
\begin{proof}
Define $h_k:\R^n \times \R^d \to \R$ by $h_k(x,\xi) := \max\left[\phi_k\left(g(x,\xi)\right)\right]$ for each $k \in
\N$, and note that $h_k$ is continuous by virtue of Assumption~\ref{ass:phibasic}A. Additionally,
Assumption~\ref{ass:phibasic}C implies that $\abs{h_k(x,\xi)} \leq M_{\phi}$ for each $x \in \R^n, \xi \in \R^d$, and $k \in \N$. By noting that for each $x \in \xv$,
\[
\uset{k \to \infty}{\lim} h_k(x,\xi) = \max\left[\1{g(x,\xi)}\right]
\]
for a.e.~$\xi \in \Xi$ due to Assumptions~\ref{ass:probzero} and~\ref{ass:phibasic} (or just the strong form of
Assumption~\ref{ass:phibasic}), we obtain the stated result by Lebesgue's dominated convergence
theorem~\citep[Theorem~7.31]{shapiro2009lectures}.   \halmos
\end{proof}

The next two results show
that a global solution of~\eqref{eqn:sp} can be obtained by solving a sequence of approximations~\eqref{eqn:app} to
global optimality. To achieve this, we rely on the concept of epi-convergence of sequences of extended real-valued
functions (see~\citet[Chapter~7]{rockafellar2009variational} for a comprehensive introduction).

\begin{proposition}
\label{prop:epiconv1}
Suppose Assumptions~\ref{ass:probzero} and~\ref{ass:phibasic} hold.
Then, the sequence of functions $\left\lbrace \ph_k(\cdot) + I_{\xv}(\cdot) \right\rbrace_k$ epi-converges to $p(\cdot) + I_{\xv}(\cdot)$ for each $\nu \in \R$.
\end{proposition}
\begin{proof}
Define $h_k:\R^n \to \R$ and $h:\R^n \to \R$ by $h_k(x) := \ph_k(x) + I_{\xv}(x)$ and $h(x) := p(x) + I_{\xv}(x)$, respectively. From Proposition~7.2 of~\citet{rockafellar2009variational}, we have that $\{h_k\}$ epi-converges to $h$ if and only if at each $x \in \R^n$
\begin{align*}
\liminf_{k \to \infty} h_k(x_k) \geq h(x), & \:\: \text{for \textit{every} sequence } \{x_k\} \to x, \text{ and} \\
\limsup_{k \to \infty} h_k(x_k) \leq h(x), & \:\: \text{for \textit{some} sequence } \{x_k\} \to x.
\end{align*}
Consider the constant sequence with $x_k = x$ for each $k \in \N$. We have
\[
\limsup_{k \to \infty} h_k(x_k) = I_{\xv}(x) + \limsup_{k \to \infty} \ph_k(x) = I_{\xv}(x) + p(x) = h(x)
\]
as a result of Proposition~\ref{prop:pointwiseconv}, which establishes the latter inequality.

To see the former inequality, consider an arbitrary sequence $\{x_k\}$ in $\R^n$ converging to $x \in \R^n$. By noting that the indicator function $I_{\xv}(\cdot)$ is lower semicontinuous (since $\xv$ is closed) and
\[
\liminf_{k \to \infty} h_k(x_k) \geq \liminf_{k \to \infty} \ph_k(x_k) + \liminf_{k \to \infty} I_{\xv}(x_k),
\]
it suffices to show $\liminf_k \ph_k(x_k) \geq p(x)$ to establish $\{h_k\}$ epi-converges to $h$. In fact, it suffices to show the above inequality holds for any $x \in \xv$ since the former inequality holds trivially for $x \not\in \xv$.

Define $q_k:\R^d \to \R$ by $q_k(\xi) := \max\left[\phi_k\left(g(x_k,\xi)\right)\right]$, and note that $q_k$ is $\Pb$-integrable by virtue of Assumption~\ref{ass:phibasic}. By Fatou's lemma~\citep[Theorem~7.30]{shapiro2009lectures}, we have that \[
\liminf_k \bexpect{q_k(\xi)} = \liminf_k \ph_k(x_k) \geq \bexpect{q(\xi)},
\]
where $q:\R^d \to \R$ is defined as 
\[
q(\xi) := \liminf_k q_k(\xi) = \liminf_k \max\left[\phi_k\left(g(x_k,\xi)\right)\right].
\]
Therefore, it suffices to show that $q(\xi) \geq \max\left[\1{g(x,\xi)}\right]$ a.e.~for each $x \in \xv$ to establish
the former inequality.
(Note that this result holds with Assumption~\ref{ass:probzero} if the strong form of Assumption~\ref{ass:phibasic} is
made.)

Fix $\xi \in \Xi$. Let $l \in \{1,\cdots,m\}$ denote an index at which the maximum in $\max\left[\1{g(x,\xi)}\right]$ is attained, i.e., $\max\left[\1{g(x,\xi)}\right] = \1{g_l(x,\xi)}$. Consider first the case when $g_l(x,\xi) > 0$. By the continuity of the functions $g$, for any $\varepsilon > 0$ there exists $N_{\varepsilon} \in \N$ such that $j \geq N_{\varepsilon} \implies g(x_j,\xi) > g(x,\xi) - \varepsilon e$. Therefore
\[
q(\xi) = \liminf_k \max\left[\phi_k\left(g(x_k,\xi)\right)\right] \geq \liminf_k \max\left[\phi_k\left(g(x,\xi) - \varepsilon e\right)\right] \geq \liminf_k \phi_{k,l}\left(g_l(x,\xi) - \varepsilon\right),
\]
where the first inequality follows from Assumption~\ref{ass:phibasic}B. Choosing $\varepsilon < g_l(x,\xi)$ yields
\[
q(\xi) = \liminf_k \max\left[\phi_k\left(g(x_k,\xi)\right)\right] \geq \liminf_k \phi_{k,l}\left(g_l(x,\xi) - \varepsilon\right) = 1 = \max\left[\1{g(x,\xi)}\right]
\]
by virtue of Assumption~\ref{ass:phibasic}D.
The case when $g_l(x,\xi) < 0$ follows more directly since
\[
q(\xi) = \liminf_k \max\left[\phi_k\left(g(x_k,\xi)\right)\right] \geq \liminf_k \phi_{k,l}\left(g_l(x,\xi) - \varepsilon\right) = \1{g_l(x,\xi) - \varepsilon} = 0 = \max\left[\1{g(x,\xi)}\right],
\]
where the second equality follows from Assumption~\ref{ass:phibasic}. Since $\max\left[g(x,\xi)\right] \neq 0$
a.e.~$\forall x \in \xv$ by Assumption~\ref{ass:probzero}, we have that $q(\xi) \geq \max\left[\1{g(x,\xi)}\right]$
a.e.~for each $x \in \xv$, which concludes the proof.  \halmos
\end{proof}

A consequence of the above proposition is the following key result (cf.~\citet[Theorem~2]{hong2011sequential},~\citet[Theorem~4.1]{shan2014smoothing},~\citet[Corollary~3.7]{geletu2017inner}, and~\citet[Theorem~5]{cao2018sigmoidal}), which establishes convergence of the optimal solutions and objective values of the sequence of approximating problems~\eqref{eqn:app} to those of the true problem~\eqref{eqn:sp}.

\begin{theorem}
\label{thm:convinmin}
Suppose Assumptions~\ref{ass:xcompact},~\ref{ass:probzero}, and~\ref{ass:phibasic} hold. Then
\[
\uset{k \to \infty}{\lim}\: \uset{x \in \xv}{\min}\: \ph(x;\tau_k) = \uset{x \in \xv}{\min}\: p(x)
\quad \text{and} \quad
\uset{k \to \infty}{\limsup} \: \uset{x \in \xv}{\argmin} \: \ph(x;\tau_k) \subset \uset{x \in \xv}{\argmin} \: p(x).
\]
\end{theorem}
\begin{proof}
Follows from Proposition~\ref{prop:epiconv1} and Theorem~7.33 of~\citet{rockafellar2009variational}.  \halmos
\end{proof}

Theorem \ref{thm:convinmin} has a couple of practical limitations. First, it is not applicable to situations where
$\xi$ is a discrete random variable since it relies crucially on Assumption~\ref{ass:probzero}. Second, it only
establishes that any accumulation point of a sequence of \textit{global} minimizers of~\eqref{eqn:app} is a
\textit{global} minimizer of~\eqref{eqn:sp}. Since~\eqref{eqn:app} involves minimizing a nonsmooth nonconvex
expected-value function, obtaining a global minimizer of \eqref{eqn:app} is itself challenging. The next result circumvents the first limitation
when the smoothing function is chosen judiciously (cf.~the outer-approximations of~\citet{geletu2017inner} and
Section~\ref{subsec:exmsmoothapprox}). Proposition~\ref{prop:strictminapprox} shows that strict local minimizers
of~\eqref{eqn:sp} can be approximated using sequences of local minimizers of~\eqref{eqn:app}. While this doesn't fully
address the second limitation, it provides some hope that strict local minimizers of~\eqref{eqn:sp} may be approximated
by solving a sequence of approximations~\eqref{eqn:app} to local optimality.
Theorem~\ref{thm:convofstatpts} establishes that accumulation points of sequences of stationary solutions to the approximations~\eqref{eqn:app} yield stationary solutions to~\eqref{eqn:sp} under additional assumptions.

\begin{proposition}
\label{prop:epiconv2}
Suppose the approximations $\{\phi_k\}$ satisfy the strong form of Assumption~\ref{ass:phibasic}.
Then the sequence of functions $\left\lbrace\ph_k(\cdot) + I_{\xv}(\cdot)\right\rbrace_k$ epi-converges to $p(\cdot) + I_{\xv}(\cdot)$.
Moreover, the conclusions of Theorem~\ref{thm:convinmin} hold if we further make Assumption~\ref{ass:xcompact}.
\end{proposition}
\begin{proof}
Follows by noting that Assumption~\ref{ass:probzero} is no longer required for the proof of
Proposition~\ref{prop:epiconv1} when the strong form of Assumption~\ref{ass:phibasic} is made.  \halmos
\end{proof}

In fact, the proof of Proposition~\ref{prop:epiconv2} becomes straightforward if we impose the additional (mild) condition (assumed by~\citet{geletu2017inner}) that the sequence of smooth approximations $\{\phi_k\}$ is monotone nondecreasing, i.e.,
$\phi_{k+1,j}(y) \geq \phi_{k,j}(y)$, $\forall y \in \R, k \in \N$, and $j \in \{1,\cdots,m\}$.
Under this additional assumption, the conclusions of Proposition~\ref{prop:epiconv2} follow from Proposition~7.4(d) and Theorem~7.33 of~\citet{rockafellar2009variational}.

The next result, in the spirit of Proposition~3.9 of~\citet{geletu2017inner}, shows that strict local minimizers of~\eqref{eqn:sp} can be approximated using local minimizers of~\eqref{eqn:app}.

\begin{proposition}
\label{prop:strictminapprox}
Suppose the assumptions of Theorem~\ref{thm:convinmin} (or those of Proposition~\ref{prop:epiconv2}) hold. If $x^*$ is a strict local minimizer of~\eqref{eqn:sp}, then there exists a sequence of local minimizers $\{\hat{x}_k\}$ of $\uset{x \in \xv}{\min}\: \ph_k(x)$ with $\hat{x}_k \to x^*$.
\end{proposition}
\begin{proof}
Since $x^*$ is assumed to be a strict local minimizer of~\eqref{eqn:sp}, there exists $\delta > 0$ such that $p(x) > p(x^*)$, $\forall
x \in \xv \cap \cl{B_{\delta}(x^*)}$. Since $\xv \cap \cl{B_{\delta}(x^*)}$ is compact, $\uset{x \in \xv \cap
\cl{B_{\delta}(x^*)}}{\min}\: \ph_k(x)$ has a global minimum, say $\hat{x}_k$, due to Assumption~\ref{ass:phibasic}A. Furthermore, $\{\hat{x}_k\}$ has a convergent subsequence in $\xv \cap \cl{B_{\delta}(x^*)}$. 

Assume without loss of generality that $\{\hat{x}_k\}$ itself converges to $\hat{x} \in \xv \cap \cl{B_{\delta}(x^*)}$.
Applying Theorem~\ref{thm:convinmin} to the above sequence of restricted minimization problems yields the conclusion
that $\hat{x}$ is a global minimizer of $p$ on $\xv \cap \cl{B_{\delta}(x^*)}$, i.e., $p(x) \geq p(\hat{x})$, $\forall x
\in \xv \cap \cl{B_{\delta}(x^*)}$. Since $x^*$ is a strict local minimizer of~\eqref{eqn:sp}, this implies $\hat{x} =
x^*$. Since $\{\hat{x}_k\} \to x^*$, this implies that $x_k$ belongs to the open ball $B_{\delta}(x^*)$ for $k$
sufficiently large. Therefore, $\hat{x}_k$ is a local minimizer of $\uset{x \in \xv}{\min}\: \ph_k(x)$ for $k$
sufficiently large since it is a global minimizer of the above problem on $B_{\delta}(x^*)$.  \halmos
\end{proof}

We are unable to establish the more desirable statement that a convergent sequence of local minimizers of approximations~\eqref{eqn:app} converges to a local minimizer of~\eqref{eqn:sp} without additional assumptions (cf.~Theorem~\ref{thm:convofstatpts}).

The next few results work towards establishing conditions under which a convergent sequence of stationary solutions to the sequence of approximating problems~\eqref{eqn:app} converges to a stationary point of~\eqref{eqn:sp}.
We make the following additional assumptions on each element of the sequence of smoothing functions $\{\phi_k\}$ for this purpose.
\begin{assumption}
\label{ass:phiadvanced}
\setcounter{mycounter}{\value{assumption}}
The following conditions hold for each $k \in \N$ and $j \in \{1,\cdots,m\}$:
\begin{enumerate}[label=\Alph*.,itemsep=0em]
\item  The derivative mapping $d\phi_{k,j}(\cdot)$ is bounded by $M^{'}_{\phi,k,j}$ on $\R$, i.e., $\abs{d\phi_{k,j}(y)}
\leq M^{'}_{\phi,k,j}$, $\forall y \in \R$. %\label{ass:phiderivbound} 
\item The derivative mapping $d\phi_{k,j}(\cdot)$ is Lipschitz continuous on $\R$ with Lipschitz constant
$L^{'}_{\phi,k,j} \in \R_+$.  %\label{ass:phiderivLipschitz}  
\end{enumerate}
\end{assumption}

The above assumptions are mild since we let the constants $M^{'}_{\phi,k,j}$ and $L^{'}_{\phi,k,j}$ depend on the sequence index $k$.
Note that in light of Assumption~\ref{ass:phibasic}A, Assumption~\ref{ass:phiadvanced}A  is equivalent to the assumption that $\phi_{k,j}(\cdot)$ is Lipschitz continuous on $\R$ with Lipschitz constant $M^{'}_{\phi,k,j}$.
We make the following assumptions on the cumulative distribution function of a `scaled version' of the constraint functions $g$ and on the sequence $\{\phi_{k,j}\}$ of smoothing functions to establish Proposition~\ref{prop:clarkegradlimit} and Theorem~\ref{thm:convofstatpts}.

\begin{assumption}
\label{ass:phiandcdfass}
\setcounter{mycounter}{\value{assumption}}
The following conditions on the constraint functions $g$, the distribution of $\xi$, and the sequences $\{\phi_{k,j}\}$ of smoothing functions hold for each $k \in \N$ and $\nu \in \R$ of interest:
\begin{enumerate}[label=\Alph*.,itemsep=0em]
\item The sequence of smoothing parameters $\{\tau_k\}$ satisfies  %\label{ass:conrescaling}
\begin{align}
\label{eqn:smoothingapproxfactor}
\tau_{k+1,j} = \tau_c \tau_{k,j} \text{  with  } \tau_{1,j} > 0, \quad \forall j \in \{1,\cdots,m\},
\end{align}
for some constant factor $\tau_c \in (0,1)$, i.e., for each $j \in \{1,\cdots,m\}$, the smoothing parameters are chosen from a monotonically decreasing geometric sequence (with limit zero).
Furthermore, the underlying smoothing function $\phi$ is homogeneous of degree zero, i.e.,
\[
\phi(\lambda y;\lambda z) = \phi(y;z), \quad \forall \lambda > 0, y \in \R, z \in \R_+.
\]
%There exist positive constants $\beta_1, \cdots, \beta_m$ such that the approximations $\ph_k(x) := \expect{\max\left[\phi_k\left(g(x,\xi)\right)\right]}$ can be rewritten as $\ph_k(x) = \expect{\bar{\phi}_k\left(\max\left[\bar{g}(x,\xi)\right]\right)}$, where $\bar{g}:\R^n \times \R^d \to \R^{m}$ is defined by $\bar{g}_j(x,\xi) := \beta_j g_j(x,\xi)$, $\forall (x,\xi) \in \R^n \times \R^d$ and $j \in \{1,\cdots,m\}$, and $\bar{\phi}_k$ is a scalar smoothing function satisfying Assumptions~\ref{ass:phibasic} and~\ref{ass:phiadvanced}.

\item
For each $j \in \{1,\cdots,m\}$, let $\bar{g}_j:\R^n \times \R^d \to \R$ be defined by $\bar{g}_j(x,\xi) := \frac{g_j(x,\xi)}{\tau_{1,j}}$, $\forall (x,\xi) \in \R^n \times \R^d$.
Let $F:\R^n \times R \to [0,1]$ denote the cumulative distribution function of $\max\left[\bar{g}(x,\xi)\right]$, i.e., $F(x,\eta) := \prob{\max\left[\bar{g}(x,\xi)\right] \leq \eta}$.
There exists a constant $\theta > 0$ such that distribution function $F$ is continuously differentiable on $\bar{X} \times (-\theta,\theta)$, where $\bar{X} \supset \xv$ is an open subset of $\R^n$. Furthermore, for each $\eta \in \R$, $F(\cdot,\eta)$ is Lipschitz continuous on $\xv$ with a measurable Lipschitz constant $L_F(\eta) \in \R_+$ 
that is also Lebesgue integrable, i.e., $\int L_F(\eta) d\eta < +\infty$.  %\label{ass:cdflipschitz}  
%such that $L_F(\cdot) d\bar{\phi}_k(\cdot)$ is Lebesgue integrable, i.e., $\int L_F(\eta) d\bar{\phi}_k(\eta) d\eta < +\infty$.

\item  %\label{ass:mollifiertail}  
There exists a sequence of positive constants $\{\varepsilon_k\} \downarrow 0$ such that
\[
\lim_{k \to \infty} \int_{\abs{\eta} \geq \varepsilon_k} L_F(\eta) d\phi(\eta;\tau_c^{k-1}) d\eta = 0 \quad \text{and} \quad \lim_{k \to \infty} \phi(\varepsilon_k;\tau_c^{k-1}) - \phi(-\varepsilon_k;\tau_c^{k-1}) = 1.
\]
\end{enumerate}
\end{assumption}

Assumption~\ref{ass:phiandcdfass}A is a mild assumption on the choice of the smoothing functions $\phi_k$ that is
satisfied by all three examples in Section~\ref{subsec:exmsmoothapprox}. When this
assumption is made in conjunction with Assumption~\ref{ass:phibasic}B, the approximation $\ph_k(x) :=
\bexpect{\max\left[\phi_k\left(g(x,\xi)\right)\right]}$ can be rewritten as $\ph_k(x) = \bexpect{\phi\left(\max\left[\bar{g}(x,\xi)\right];\tau_c^{k-1}\right)}$, which lends the following interpretation: the sequence of approximations $\{\ph_k\}$ can be thought of as having been constructed using the same sequence of smoothing parameter values $\{\tau_{k,j}\} = \{\tau_c^{k-1}\}$ for each constraint~$j$ on a rescaled version $\bar{g}_j$ of the constraint function $g_j$.
The Lipschitz continuity assumption in Assumption~\ref{ass:phiandcdfass}B is mild (and is similar to Assumption~4
of~\citet{shan2014smoothing}). The assumption of local continuous differentiability of the distribution function $F$ in
Assumption~\ref{ass:phiandcdfass}B is quite strong
%\footnote{We believe that this assumption can be weakened significantly, see Lemma~4.2 of~\citet{shan2014smoothing}. We do not adopt the assumptions and proof of Lemma~4.2 of~\citet{shan2014smoothing} since we are unable to follow a key step in the proof.} 
(this assumption is similar to Assumption~4 of~\citet{hong2011sequential}). A consequence of this assumption is that the function $p(x) \equiv 1 - F(x,0)$ is continuously differentiable on $\bar{X}$.
Finally, note that Assumption~\ref{ass:phiandcdfass}C is mild, see Section~\ref{subsec:exmsmoothapprox} for examples that (trivially) satisfy it. When made along with Assumption~\ref{ass:phibasic}, the first part of this assumption ensures that the mapping $d\phi(\cdot;\tau_c^{k-1})$ approaches the `Dirac delta function' sufficiently rapidly (cf.~Remark~3.14 of~\citet{ermoliev1995minimization}).

The following result ensures that the Clarke gradient of the approximation $\ph_k$ is well defined.

\begin{lemma}
\label{lem:problipconst}
Suppose Assumptions~\ref{ass:confunc},~\ref{ass:phibasic}, and~\ref{ass:phiadvanced} hold. Then for each $k \in \N$:
\begin{enumerate}
\item $\phi_{k,j}(g_j(\cdot,\xi))$ is Lipschitz continuous on $\xv$ with Lipschitz constant $M^{'}_{\phi,k,j} L_{g,j}(\xi)$ for each $\xi \in \Xi$ and $j \in \{1,\cdots,m\}$.
\item $\ph_k$ is Lipschitz continuous on $\xv$ with Lipschitz constant $\Bexpect{\uset{j \in \{1,\cdots,m\}}{\max}
\bigl[ M^{'}_{\phi,k,j} L_{g,j}(\xi) \bigr]}$.
\end{enumerate}
\end{lemma}
\begin{proof}
\begin{enumerate}
\item For any $x, y \in \xv$, $j \in \{1,\cdots,m\}$, and $\xi \in \Xi$, we have
\[
\abs{\phi_{k,j}\left(g_j(y,\xi)\right) - \phi_{k,j}\left(g_j(x,\xi)\right)} \leq M^{'}_{\phi,k,j}\abs{g_j(y,\xi) - g_j(x,\xi)} \leq M^{'}_{\phi,k,j} L_{g,j}(\xi) \norm{y-x},
\]
where the first step follows from Assumptions~\ref{ass:phibasic}A and~\ref{ass:phiadvanced}A and the mean-value theorem,
and the second step follows from Assumption~\ref{ass:confunc}A.

\item For any $x, y \in \xv$, we have
\begin{align*}
\abs{\ph_k(y) - \ph_k(x)} &= \abs*{\bexpect{\max\left[\phi_k\left(g(y,\xi)\right)\right] - \max\left[\phi_k\left(g(x,\xi)\right)\right]}} \\
&\leq \bexpect{\max\left[ \abs{\phi_k\left(g(y,\xi)\right) - \phi_k\left(g(x,\xi)\right)} \right]} \\
&\leq \Bexpect{\max_j \left[ M^{'}_{\phi,k,j} L_{g,j}(\xi) \right]} \norm{y-x},
\end{align*}
\end{enumerate}
which completes the proof.  \halmos
\end{proof}

The next result characterizes the Clarke generalized gradient of the approximating objectives~$\ph_k$.

\begin{proposition}
\label{prop:clarkegradapprox}
Suppose Assumptions~\ref{ass:confunc},~\ref{ass:phibasic}, and~\ref{ass:phiadvanced} hold. Then the Clarke generalized gradient of $\ph_k$ can be expressed as
%\[
%\partial \ph_k(x) = \expect{\conv{\left\lbrace\Gamma_k(x,\xi)\right\rbrace}} = \conv{\left\lbrace\expect{\Gamma_k(x,\xi)}\right\rbrace},
%\]
\[
\partial \ph_k(x) = \bexpect{\conv{\left\lbrace\Gamma_k(x,\xi)\right\rbrace}},
\]
where $\Gamma_k : \R^n \times \R^d \rightrightarrows \R^n$ is defined as $\Gamma_k(x,\xi) := \Set{\nabla_x \phi_{k,l}(g_l(x,\xi))}{l \in \A(x,\xi)}$, $\A(x,\xi)$ denotes the set of active indices $l \in \{1,\cdots,m\}$ at which $\max\left[\phi_k\left(g(x,\xi)\right)\right] = \phi_{k,l}\left(g_l(x,\xi)\right)$, and the expectation above is to be interpreted in the sense of Definition~7.39 of~\citet{shapiro2009lectures}.
\end{proposition}
\begin{proof}
From the corollary to Proposition~2.2.1 of~\citet{clarke1990optimization}, we have that $\phi_k\left(g(\cdot,\xi)\right)$ is strictly differentiable. Theorem~2.3.9 of~\citet{clarke1990optimization} then implies that $\max\left[\phi_k\left(g(\cdot,\xi)\right)\right]$ is Clarke regular.
The stronger assumptions of Theorem~2.7.2 of~\citet{clarke1990optimization} are satisfied because of Assumption~\ref{ass:phibasic}, Lemma~\ref{lem:problipconst}, and the regularity of $\max\left[\phi_k\left(g(\cdot,\xi)\right)\right]$, which then yields
\[
\partial \ph_k(x) = \partial \bexpect{\max\left[\phi_k\left(g(x,\xi)\right)\right]} = \bexpect{\partial_x\max\left[\phi_k\left(g(x,\xi)\right)\right]}.
\]
Noting that $\phi_k\left(g(\cdot,\xi)\right)$ is Clarke regular from Proposition~2.3.6
of~\citet{clarke1990optimization}, the stated equality then follows from Proposition~2.3.12
of~\citet{clarke1990optimization}.  \halmos
\end{proof}

The next result, similar to Lemma~4.2 of~\citet{shan2014smoothing}, helps characterize the accumulation points of stationary solutions to the sequence of approximations~\eqref{eqn:app}.

\begin{proposition}
\label{prop:clarkegradlimit}
Suppose Assumptions~\ref{ass:confunc},~\ref{ass:phibasic},~\ref{ass:phiadvanced}, and~\ref{ass:phiandcdfass} hold. Then
\[
\uset{k \to \infty}{\limsup_{x \to \bar{x}}}\: \partial \ph_k(x) + N_{\xv}(x) \subset \{\nabla p(\bar{x})\} + N_{\xv}(\bar{x}).
\]
\end{proposition}
\begin{proof}
See Appendix~\ref{subsec:propclarkegradlimit}.
 \halmos
\end{proof}

The following key result is an immediate consequence of the above proposition (cf.~\citet[Theorem~4.2]{shan2014smoothing}), and establishes convergence of Clarke stationary solutions of the sequence of approximating problems~\eqref{eqn:app} to those of the true problem~\eqref{eqn:sp}.

\begin{theorem}
\label{thm:convofstatpts}
Suppose Assumptions~\ref{ass:xcompact},~\ref{ass:confunc},~\ref{ass:phibasic},~\ref{ass:phiadvanced}, and~\ref{ass:phiandcdfass} hold. Let $\{x_k\}$ be a sequence of stationary solutions to the sequence of approximating problems $\uset{x \in \xv}{\min}\: \ph_k(x)$. Then every accumulation point of $\{x_k\}$ is an stationary solution to $\uset{x \in \xv}{\min}\: p(x)$.
\end{theorem}
\begin{proof}
From the corollary to Proposition~2.4.3 of~\citet{clarke1990optimization}, we have $0 \in \partial \ph_k(x_k) + N_{\xv}(x_k)$ by virtue of the stationarity of $x_k$ for $\uset{x \in \xv}{\min}\: \ph_k(x)$. The stated result follows from Proposition~\ref{prop:clarkegradlimit} since $0 \in \{\nabla p(\bar{x})\} + N_{\xv}(\bar{x})$ for any accumulation point $\bar{x}$ of $\{x_k\}$ (see Theorem~5.37 of~\citet{rockafellar2009variational}). 
 \halmos
\end{proof}

\subsection{Examples of smooth approximations}
\label{subsec:exmsmoothapprox}

We present a few examples of sequences of smooth approximations of the step function that satisfy Assumptions~\ref{ass:phibasic},~\ref{ass:phiadvanced}, and~\ref{ass:phiandcdfass}.
Throughout this section, we let $\{\tau_{k,j}\}$, $j \in \{1,\cdots,m\}$, denote sequences of \textit{positive} reals satisfying~\eqref{eqn:smoothingapproxfactor}.

\begin{example}
This example is from Example~5.1 of~\citet{shan2014smoothing}. The sequence of approximations
\[
\phi_{k,j}(y) = \begin{cases} 
0 &\mbox{if } y < -\tau_{k,j}, \\ 
1 - 2\left(\frac{y}{\tau_{k,j}}\right)^3 - 3\left(\frac{y}{\tau_{k,j}}\right)^2 &\mbox{if } -\tau_{k,j} \leq y \leq 0, \\
1 & \mbox{if } y > 0 
\end{cases}
\]
of the step function satisfy Assumptions~\ref{ass:phibasic} and~\ref{ass:phiadvanced}.
Assumption~\ref{ass:phiandcdfass}A is easily verified, and Assumption~\ref{ass:phiandcdfass}C trivially holds with $\varepsilon_k = \tau_c^{k-1}$.
\end{example}

\begin{example}
This example is based on the above example and the outer-approximation framework of~\citet{geletu2017inner}. The sequence of approximations
\[
\phi_{k,j}(y) = \begin{cases} 
0 &\mbox{if } y < 0, \\ 
1 - 2\left(\frac{y-\tau_{k,j}}{\tau_{k,j}}\right)^3 - 3\left(\frac{y-\tau_{k,j}}{\tau_{k,j}}\right)^2 &\mbox{if } 0 \leq y \leq \tau_{k,j}, \\
1 & \mbox{if } y > \tau_{k,j} 
\end{cases}
\]
of the step function satisfy the strong form of Assumption~\ref{ass:phibasic} and Assumption~\ref{ass:phiadvanced}.
Assumption~\ref{ass:phiandcdfass}A is easily verified, and Assumption~\ref{ass:phiandcdfass}C trivially holds with $\varepsilon_k = \tau_c^{k-1}$.
\end{example}

The next example is the sequence of smooth approximations $\{\phi_k\}$ that we adopt in this work.

\begin{example}
\label{exm:oursmoothapprox}
The sequence of approximations
\[
\phi_{k,j}(y) = \dfrac{1}{1 + \exp\left(-\frac{y}{\tau_{k,j}}\right)}
\]
of the step function satisfy Assumptions~\ref{ass:phibasic} and~\ref{ass:phiadvanced} (see Proposition~\ref{prop:oursmoothapproxconst}).
Assumption~\ref{ass:phiandcdfass}A is easily verified, and Assumption~\ref{ass:phiandcdfass}C could be satisfied with $\varepsilon_k = \tau_c^{\beta(k-1)}$ for some constant $\beta \in (0,1)$ (depending on the distribution function $F$, cf.~Remark~3.14 of~\citet{ermoliev1995minimization}).
\end{example}

Figure~\ref{fig:compsmoothfns} illustrates the above smoothing functions with the common smoothing parameter $\tau = 0.1$.
We refer the reader to the works of~\citet{shan2014smoothing},~\citet{geletu2017inner}, and~\citet{cao2018sigmoidal} for other examples of `smooth' approximations that can be accommodated within our framework.
The next result estimates some important constants related to our smooth approximation in Example~\ref{exm:oursmoothapprox}.

\begin{proposition}
\label{prop:oursmoothapproxconst}
The sequence of approximations $\{\phi_k\}$ in Example~\ref{exm:oursmoothapprox} satisfies Assumption~\ref{ass:phiadvanced} with constants $M^{'}_{\phi,k,j} \leq 0.25 \tau^{-1}_{k,j}$ and $L^{'}_{\phi,k,j} \leq 0.1\tau^{-2}_{k,j}$, $\forall k \in \N$, $j \in \{1,\cdots,m\}$.
\end{proposition}
\begin{proof}
The bounds follow by noting that for each $y \in \R$:
\begin{align*}
d\phi_k(y) &= \dfrac{\tau^{-1}_{k,j}}{\left[\exp(0.5\tau^{-1}_{k,j} y) + \exp(-0.5\tau^{-1}_{k,j} y)\right]^2} \leq 0.25 \tau^{-1}_{k,j}, \:\: \text{and} \\
\abs*{d^2\phi_{k,j}(y)} &= \tau^{-2}_{k,j} \abs*{\dfrac{\exp\left( \frac{y}{2\tau_{k,j}} \right) - \exp\left( -\frac{y}{2\tau_{k,j}} \right)}{\left[ \exp\left( \frac{y}{2\tau_{k,j}} \right) + \exp\left( -\frac{y}{2\tau_{k,j}} \right) \right]^3}} \leq 0.1 \tau^{-2}_{k,j},
\end{align*}
where the final inequality is obtained by maximizing the function $x \mapsto \abs*{\dfrac{x - \frac{1}{x}}{\left( x +
\frac{1}{x} \right)^3}}$ on $\R$.  \halmos
\end{proof}

\begin{figure}[t]
\begin{center}
\caption{Illustration of the examples of smoothing functions in Section~\ref{subsec:exmsmoothapprox}.} \label{fig:compsmoothfns}
\includegraphics[height=2.8in]{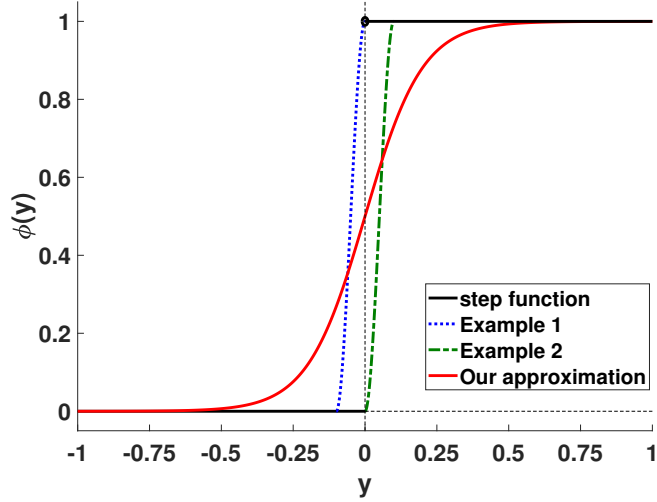}
\end{center}
\end{figure}

\section{Proposed algorithm}
\label{sec:propalgo}

Our proposal for approximating the efficient frontier of~\eqref{eqn:ccp} involves solving a sequence of approximations~\eqref{eqn:app} constructed using the sequence of smooth approximations of the step function introduced in Example~\ref{exm:oursmoothapprox}. 
We use the projected stochastic subgradient algorithm of~\citet{davis2018stochastic} to obtain an approximately stationary solution to~\eqref{eqn:app} for each value of the objective bound $\nu$ and each element of the sequence of smoothing parameters~$\{\tau_{k}\}$ (note that we suppress the constraint index $j$ in our notation from here on unless necessary).
Section~\ref{subsec:algooutline} outlines a conceptual algorithm for approximating the efficient frontier of~\eqref{eqn:ccp} using the projected stochastic subgradient algorithm, 
%of~\citet{davis2018stochastic},
and Sections~\ref{subsec:algoparams} and~\ref{subsec:practcon} outline our proposal for estimating some parameters of the algorithm in Section~\ref{subsec:algooutline} that are important for a good practical implementation.
%We close this section with a discussion of our proposed approach in Section~\ref{subsec:limitapp}.
%(this section includes an extension of Algorithm~\ref{alg:propalgoef} to the case when Problem~\eqref{eqn:ccp} contains multiple sets of joint chance constraints).

\subsection{Outline of the algorithm}
\label{subsec:algooutline}

Algorithm~\ref{alg:protoalgo} presents a prototype of our approach for constructing an approximation of the efficient frontier of~\eqref{eqn:ccp}.
It takes as input an initial guess $\hat{x}^0$, a sequence of objective bounds $\{\nu^i\}_{i=1}^{M}$, the maximum number of `runs', $R$, of the projected stochastic subgradient method, and the maximum number of iterations of the method in each run, $N_{max}$.
For each objective bound $\nu^i$, we first determine a sequence of smoothing parameters and step lengths $\{(\tau_k,\gamma_k)\}_{k=1}^{K}$ for the sequence of approximations~\eqref{eqn:app}.
The middle loop of Algorithm~\ref{alg:protoalgo} then iteratively solves the sequence of approximations~\eqref{eqn:app} with $\nu = \nu^i$ and smoothing parameters $\{\tau_k\}$ by using the solution from the previous approximation as its initial guess. 
The innermost loop of Algorithm~\ref{alg:protoalgo} monitors the progress of the stochastic subgradient method over several `runs'. Each run generates a candidate solution $\bar{x}^{i,k,r}$ with risk level $\bar{\alpha}^{i,k,r}$ (we omit the superscript $r$ in Algorithm~\ref{alg:protoalgo} for simplicity), and we move on to the next approximation in the sequence either if insufficient progress is made over the last few runs, or if the number of runs exceeds its limit.
Since the choice of the smoothing parameters $\{\tau_k\}$, step lengths $\{\gamma_k\}$, and the termination criteria
employed within Algorithm~\ref{alg:protoalgo} are important for its practical performance, we present our proposal for estimating/setting these algorithmic parameters from the problem data in Section~\ref{subsec:algoparams}.
First, we split Algorithm~\ref{alg:protoalgo} into several component algorithms so that it is easier to parse.

Algorithm~\ref{alg:propalgoef} outlines our proposal for approximating the efficient frontier of~\eqref{eqn:ccp}. This algorithm takes as its inputs an initial guess $\hat{x}^0 \in X$, an initial objective bound $\bar{\nu}^0$ of interest, an objective spacing $\tilde{\nu}$ for discretizing the efficient frontier, and a lower bound $\alpha_{low}$ on risk levels of interest that is used as a termination criterion. Section~\ref{subsec:algoparams} prescribes ways to determine a good initial guess $\hat{x}^0$ and initial bound $\bar{\nu}^0$, whereas Section~\ref{subsec:impldet} lists our setting for $\tilde{\nu}$, $\alpha_{low}$, and other algorithmic parameters.
The preprocessing step uses Algorithm~\ref{alg:scaling} to determine a suitable initial sequence of smoothing parameters $\{\tau_k\}$ based on the problem data, and uses this choice of $\{\tau_k\}$ to determine an initial sequence of step lengths for the stochastic subgradient algorithm (see Section~\ref{subsec:algoparams}). 
Good choices of both parameters are critical for our proposal to work well.
The optimization phase of Algorithm~\ref{alg:propalgoef} uses Algorithm~\ref{alg:propalgo} to solve the sequence of approximations~\eqref{eqn:app} for different choices of the objective bound $\nu = \bar{\nu}^i$ (note that Algorithm~\ref{alg:propalgo} rescales its input smoothing parameters $\{\bar{\tau}_k\}$ at the initial guess $\bar{x}^i$ for each bound $\bar{\nu}^i$).
Finally, Algorithm~\ref{alg:propalgoccp} in the appendix adapts Algorithm~\ref{alg:propalgoef} to solve~\eqref{eqn:ccp} approximately for a given risk level $\hat{\alpha}$.

\begin{algorithm}
\caption{Prototype of algorithm for approximating the efficient frontier of~\eqref{eqn:ccp}}
\label{alg:protoalgo}
{
%\fontsize{11}{14}\selectfont
\begin{algorithmic}[1]

\State \textbf{Input}: initial point $\hat{x}^0 \in X$, sequence of objective bounds $\{\nu^i\}_{i=1}^{M}$, maximum number of projected stochastic subgradient iterations, $N_{max}$, and number of `runs' of the stochastic subgradient method, $R$.
%, sequence of smoothing parameters $\{\tau_k\}_{k=1}^{K}$, and sequence of step lengths $\{\gamma_k\}_{k=1}^{K}$ for the sequence of approximations~\eqref{eqn:app} with smoothing parameters $\{\tau_k\}_{k=1}^{K}$.

\State \textbf{Output}: pairs $\{(\nu^i,\hat{\alpha}^i)\}_{i=1}^{M}$ of objective values and risk levels on the approximation of the efficient frontier along with associated solutions $\{\hat{x}^i\}_{i=1}^{M}$.

%\State \textbf{Set algorithmic parameters}: maximum number of projected stochastic subgradient iterations $N_{max}$, number of `runs' of the stochastic subgradient method $R$, sequence of smoothing parameters $\{\tau_k\}_{k=1}^{K}$, and sequence of step lengths $\{\gamma_k\}_{k=1}^{K}$ for the sequence of approximations~\eqref{eqn:app} with smoothing parameters $\{\tau_k\}_{k=1}^{K}$.

\For{iteration $i = 1$ to $M$} \Comment{Loop through the objective bounds $\{\nu^i\}$}

\State Set initial guess $\bar{x}^{i,0} := \hat{x}^{i-1}$.

\State Determine sequence of smoothing parameters and step lengths $\{(\tau_k,\gamma_k)\}_{k=1}^{K}$.

\For{approximation $k = 1$ to $K$} \Comment{Loop through the sequence of smooth approximations}

\State Set initial guess $\tilde{x}^{k,0} := \bar{x}^{i,k-1}$ for approximation $k$.

\For{run $r = 1$ to $R$} \Comment{Runs of the stochastic subgradient method}

\State Initialize the first iterate $x^{1} := \tilde{x}^{k,0}$.

\State Choose number of iterations $N$ for this run uniformly at random from $\{1,\cdots,N_{max}\}$.

\State Execute $N$ iterations of the projected stochastic subgradient algorithm of~\citep{davis2018stochastic} for solving~\eqref{eqn:app} with

\Statex \hspace*{0.55in} $\nu = \nu^i$ and smoothing parameter $\tau_k$ using step length $\gamma_k$. 

\State Set $\bar{x}^{i,k}$ to be the last iterate of the stochastic subgradient method, and estimate its risk level $\bar{\alpha}^{i,k}$.

\State Update incumbents $(\hat{\alpha}^i,\hat{x}^i)$ with $(\bar{\alpha}^{i,k},\bar{x}^{i,k})$ if needed. Update $\tilde{x}^{k,0} = \bar{x}^{i,k}$.
%\Statex \hspace*{0.42in} $\hat{\alpha} > \bar{\alpha}^i$.

\State If there is insufficient improvement in the incumbent $\hat{\alpha}^i$ over the last few runs, exit the `runs' loop.

\EndFor

\State Update initial guess for the next approximation in the sequence to the incumbent solution, i.e., set $\bar{x}^{i,k} = \hat{x}^i$.

\EndFor

\EndFor

\end{algorithmic}
}
\end{algorithm}

\begin{algorithm}
\caption{Approximating the efficient frontier of~\eqref{eqn:ccp}}
\label{alg:propalgoef}
{
%\fontsize{11}{14}\selectfont
\begin{algorithmic}[1]

\State \textbf{Input}: initial point $\hat{x}^0 \in X$, initial objective bound $\bar{\nu}^0$, objective spacing $\tilde{\nu}$, and lower bound on risk levels of interest $\alpha_{low} \in (0,1)$.

\State \textbf{Set algorithmic parameters}: mini-batch size $M$, maximum number of iterations $N_{max}$, minimum and maximum number of `runs', $R_{min}$ and $R_{max}$, parameters for determining and updating step length, termination criteria, fixed sample $\{\bar{\xi}^{l}\}_{l=1}^{N_{MC}}$ from $\P$ for estimating risk levels of candidate solutions, and sequence of smoothing parameters $\{\bar{\tau}_k\}_{k=1}^{K}$ for some $K \in \N$.

\State \textbf{Output}: set of pairs $\{(\bar{\nu}^i,\bar{\alpha}^i)\}$ of objective values and risk levels that can be used to approximate the efficient frontier and associated solutions $\{\hat{x}^i\}$.

\State \textbf{Preprocessing}: determine smoothing parameters $\{\tau_k\}_{k=1}^{K}$ scaled at $\proj{\hat{x}^0}{X_{\bar{\nu}^0}}$ using Algorithm~\ref{alg:scaling}, and an initial sequence of step lengths $\{\bar{\gamma}_k\}_{k=1}^{K}$ for this sequence of approximating problems~\eqref{eqn:app} with $\nu = \bar{\nu}^0$ using Algorithm~\ref{alg:steplength}.

\State \textbf{Optimization Phase}:

\State Initialize iteration count $i = 0$.

\Repeat

\State Update iteration count $i \leftarrow i+1$ and set objective bound $\bar{\nu}^i = \bar{\nu}^0 - (i-1)\tilde{\nu}$.

\State Obtain $(\bar{\alpha}^i,\hat{x}^i)$ by solving the sequence of approximations~\eqref{eqn:app} using Algorithm~\ref{alg:propalgo} with objective bound $\nu = \bar{\nu}^i$,

\Statex \hspace*{0.2in} scaling parameters $\{\bar{\tau}_k\}_{k=1}^{K}$, the above algorithmic parameter settings, and $\bar{x}^i := \proj{\hat{x}^{i-1}}{X_{\bar{\nu}^i}}$ as the

\Statex \hspace*{0.2in} initial guess.

\Until{$\bar{\alpha}^i \leq \alpha_{low}$}

\end{algorithmic}
}
\end{algorithm}

Algorithm~\ref{alg:propalgo} solves a sequence of approximations~\eqref{eqn:app} for a given objective bound $\nu$ using the projected stochastic subgradient algorithm of~\citet{davis2018stochastic} and an adaptation of the two-phase randomized stochastic projected gradient method of~\citet{ghadimi2016mini}.
This algorithm begins by rescaling the smoothing parameters $\{\bar{\tau}_k\}$ at the initial guess $\bar{x}^0$ using Algorithm~\ref{alg:scaling} to ensure that the sequence of approximations~\eqref{eqn:app} are `well-scaled'\footnote{We make this notion more precise in Section~\ref{subsec:algoparams}.} at $\bar{x}^0$.
The optimization phase then solves each element of the sequence of approximations~\eqref{eqn:app} using two loops: the inner loop employs a mini-batch version\footnote{Although~\citet{davis2018stochastic} establish that mini-batching is not necessary for the convergence of the projected stochastic subgradient algorithm, a small mini-batch can greatly enhance the performance of the algorithm in practice.} of the algorithm of~\citet{davis2018stochastic} to solve~\eqref{eqn:app} for a given smoothing parameter $\tau_k$, and the outer loop assesses the progress of the inner loop across multiple `runs/replicates' (the initialization strategy for each run is based on \texttt{Algorithm 2-RSPG-V} of~\citet{ghadimi2016mini}).
Finally, the initial guess for the next approximating problem~\eqref{eqn:app} is set to be a solution corresponding to the smallest estimate of the risk level determined thus far. This initialization step is important for the next approximating problem in the sequence to be well-scaled at its initial guess - this is why we do not solve a `single tight approximating problem'~\eqref{eqn:app} that can be hard to initialize (cf.~Figure~\ref{fig:compsmoothfns}), but instead solve a sequence of approximations to approximate the solution of~\eqref{eqn:sp}.
Note that while the number of iterations $N$ of the inner projected stochastic subgradient loop is chosen at random as required by the theory in~\citet{davis2018stochastic}, a practical alternative is to choose a deterministic number of iterations $N$ and set $\bar{x}^i$ to be the solution at the final iterate.
Line~17 of Algorithm~\ref{alg:propalgo} heuristically updates the step length to ensure good practical performance.
Additionally, in contrast with the proposal of \texttt{Algorithm 2-RSPG-V} of~\citet{ghadimi2016mini} that chooses the solution over the multiple runs that is closest to stationarity, motivated by practical application, we return the solution corresponding to the smallest found risk level.
Because of these heuristic steps, the algorithm would require slight modification (removal of the heuristic step-size
update method and different choice of the solution from multiple runs) to be
assured to have the convergence guarantees of \citet{davis2018stochastic}. However, as our emphasis in this paper is on
deriving a practical algorithm and testing it empirically, we keep the algorithm description matching what we have
implemented. 
%We choose not to formally state a nuanced result outlining conditions under which a modification of our proposal converges to the efficient frontier of~\eqref{eqn:ccp} because we believe that Algorithm~\ref{alg:propalgoef} is more useful from a practical viewpoint.

\begin{algorithm}
\caption{Generating a point approximately on the efficient frontier}
\label{alg:propalgo}
{
%\fontsize{11}{14}\selectfont
\begin{algorithmic}[1]

\State \textbf{Input}: objective bound $\nu$, initial guess $\bar{x}^{0} \in \xv$, mini-batch size $M$, maximum number of iterations $N_{max}$, minimum and maximum number of `runs', $R_{min}$ and $R_{max} \geq R_{min}$, initial step lengths $\{\bar{\gamma}_{k}\}_{k=1}^{K}$, parameters for updating step length, run termination criteria, sample $\{\bar{\xi}^{l}\}_{l=1}^{N_{MC}}$ from $\P$ for estimating risk levels, and sequence of smoothing parameters $\{\bar{\tau}_k\}_{k=1}^{K}$.

\State \textbf{Output}: Smallest estimate of the risk level $\hat{\alpha}$ and corresponding solution $\hat{x} \in \xv$.

\State \textbf{Preprocessing}: determine smoothing parameters $\{\tau_k\}_{k=1}^{K}$ scaled at $\bar{x}^0$ using Algorithm~\ref{alg:scaling}.

\State \textbf{Initialize}: Best found solution $\hat{x} = \bar{x}^0$ and estimate of its risk level $\hat{\alpha}$ using the sample $\{\bar{\xi}^{l}\}_{l=1}^{N_{MC}}$.

\State \textbf{Optimization Phase}:

\For{approximation $k = 1$ to $K$}

\State Initialize run count $i = 0$ and step length $\gamma_k = \bar{\gamma}_k$.

\Repeat \Comment{Begin outer loop of `runs/replicates'}

\State Update $i \leftarrow i+1$, and initialize the first iterate $x^{1} := \bar{x}^{i-1}$.

\State Choose number of iterations $N$ uniformly at random from $\{1,\cdots,N_{max}\}$.

\For{iteration $l = 1$ to $N-1$} \Comment{Inner projected stochastic subgradient loop}

\State Draw an i.i.d. sample $\{\xi^{l,q}\}_{q=1}^{M}$ of $\xi$ from $\P$.

\State Estimate an element $G(x^l) \in \R^n$ of the subdifferential of the objective of~\eqref{eqn:app} with smoothing

\Statex \hspace*{0.65in} parameters $\tau_k$ at $x^l$ using the mini-batch $\{\xi^{l,q}\}_{q=1}^{M}$.

\State Update $x^{l+1} := \proj{x^{l} - \gamma_k G(x^l)}{\xv}$.

\EndFor

\State Set $\bar{x}^i = x^{N}$, and estimate its risk level $\bar{\alpha}^i$ using the sample $\{\bar{\xi}^{l}\}_{l=1}^{N_{MC}}$.

\State Update step length $\gamma_k$ using Algorithm~\ref{alg:updatesteplength}, and the incumbents $(\hat{\alpha},\hat{x})$ with $(\bar{\alpha}^i,\bar{x}^i)$ if $\hat{\alpha} > \bar{\alpha}^i$.
%\Statex \hspace*{0.42in} $\hat{\alpha} > \bar{\alpha}^i$.

\Until{run termination criteria are satisfied (see Algorithm~\ref{alg:updatesteplength})}

\State Update initial guess $\bar{x}^0 = \hat{x}$ for the next approximation.

\EndFor

\end{algorithmic}
}
\end{algorithm}

In the remainder of this section, we verify that the assumptions of~\citet{davis2018stochastic} hold to justify the use of their projected stochastic subgradient method for solving the sequence of approximations~\eqref{eqn:app}.
First, we verify that the objective function $\ph_k$ of~\eqref{eqn:app} with smoothing parameter $\tau_k$ is a weakly convex function on $\xv$.
Next, we verify that Assumption~(A3) of~\citet{davis2018stochastic} holds.
Note that a function $h:\R^d \to \overline{\R}$ is said to be $\rho_h$-weakly convex (for some $\rho_h \geq 0$) if $h + 0.5\rho_h\norm{\cdot}^2$ is convex on $\R^d$ (see Definition~4.1 of~\citet{drusvyatskiy2016efficiency} and the surrounding discussion).

\begin{proposition}
\label{prop:weakconvparam}
Suppose Assumptions~\ref{ass:confunc},~\ref{ass:phibasic} and~\ref{ass:phiadvanced} hold. Then $\ph_k(\cdot)$ is
$\bar{L}_k$-weakly convex on $\xv$, where $\bar{L}_k$ is the Lipschitz constant of the Jacobian $\bexpect{\nabla\phi_k\left(g(\cdot,\xi)\right)}$ on $\xv$.
%, where $\bar{L}_k$ is the Lipschitz constant of the Jacobian map $\expect{\nabla \phi_k\left(g(\cdot,\xi)\right)}$ on $\xv$.
\end{proposition}
\begin{proof}
Follows from Lemma~4.2 of~\citet{drusvyatskiy2016efficiency}.  \halmos
\end{proof}
The following result is useful for bounding $\ph_k$'s weak convexity parameter in terms of known constants.

\begin{proposition}
\label{prop:lipconstcomp}
Suppose Assumptions~\ref{ass:confunc},~\ref{ass:phibasic} and~\ref{ass:phiadvanced} hold. Then
\begin{enumerate}

\item For any $\xi \in \Xi$, the Lipschitz constant $L_{k,j}(\xi)$ of $\nabla \phi_{k,j}\left(g_j(\cdot,\xi)\right)$ on $\xv$ satisfies 
\[
L_{k,j}(\xi) \leq M^{'}_{\phi,k,j} L^{'}_{g,j}(\xi) + L^{'}_{\phi,k,j} L^2_{g,j}(\xi).
\]

\item The Lipschitz constant $\bar{L}_k$ of the Jacobian $\nabla\bexpect{\phi_k\left(g(\cdot,\xi)\right)}$ on $\xv$ satisfies 
\[
\bar{L}_k \leq \Bexpect{\bigl( \sum_{j=1}^{m} L^2_{k,j}(\xi) \bigr)^{\frac{1}{2}}}.
\]
\end{enumerate}
\end{proposition}
\begin{proof}
\begin{enumerate}
\item Follows by using the chain rule $\nabla \phi_{k,j}\left(g_j(\cdot,\xi)\right) =
d\phi_{k,j}\left(g_j(\cdot,\xi)\right) \nabla g_j(\cdot,\xi)$ and bounding the Lipschitz constant of the above product
on $\xv$ using Assumptions~\ref{ass:confunc}A,~\ref{ass:confunc}B, and~\ref{ass:phiadvanced} and standard arguments.

\item From part one of Lemma~\ref{lem:problipconst} and Theorem~7.44 of~\citet{shapiro2009lectures}, we have
$\nabla\bexpect{\phi_k\left(g(x,\xi)\right)} = \bexpect{\nabla_x \phi_k\left(g(x,\xi)\right)}$, $\forall x \in \xv$.
The stated result then follows from the first part of this proposition and the fact that the Frobenius norm of a matrix
provides an upper bound on its spectral norm, where the existence of the quantity $\bexpect{\left( \sum_{j=1}^{m}
L^2_{k,j}(\xi) \right)^{\frac{1}{2}}}$ follows from the fact that $\sqrt{\sum_j L^2_{k,j}(\xi)} \leq \sum_j
L_{k,j}(\xi)$, Assumptions~\ref{ass:confunc}A and~\ref{ass:confunc}B, and Lebesgue's dominated convergence theorem.
 \halmos
\end{enumerate}
\end{proof}

We derive alternative results regarding the weak convexity parameter of the approximation $\ph_k$ when~\eqref{eqn:ccp} is used to model recourse formulations in Appendix~\ref{app:recourse}.
We now establish that Assumption~(A3) of~\citet{davis2018stochastic} holds.

\begin{proposition}
\label{prop:stochsubgradvar}
Suppose Assumptions~\ref{ass:confunc},~\ref{ass:phibasic}, and~\ref{ass:phiadvanced} hold.
Let $G(x,\xi)$ be a stochastic $B$-subdifferential element of $\ph_k(\cdot)$ at $x \in \xv$, i.e., $\bexpect{G(x,\xi)}
\in \partial_B \bexpect{\max\left[\phi_k\left( g(x,\xi)\right)\right]}$. Then
\[
\bexpect{\norm{G(x,\xi)}^2} \leq \uset{j \in \{1,\cdots,m\}}{\max} \left(M^{'}_{\phi,k,j}\right)^2 \sigma^2_{g,j}.
\]
\end{proposition}
\begin{proof}
Follows from Assumptions~\ref{ass:confunc}C and~\ref{ass:phiadvanced} and the fact that $G(x,\xi) =
d\phi_{k,j}(g_j(x,\xi)) \nabla_x g_j(x,\xi)$ for some active constraint index $j \in \{1,\cdots,m\}$ such that
$\max\left[\phi_k\left( g(x,\xi)\right)\right] = \phi_{k,j}\left( g_j(x,\xi)\right)$.  \halmos
\end{proof}

\subsection{Estimating the parameters of Algorithm~\ref{alg:propalgoef}}
\label{subsec:algoparams}

This section outlines our proposal for estimating some key parameters of Algorithm~\ref{alg:propalgoef}.
We present pseudocode for determining suitable smoothing parameters and step lengths, and for deciding when to terminate Algorithm~\ref{alg:propalgo} with a point on our approximation of the efficient frontier.

Algorithm~\ref{alg:scaling} rescales the sequence of smoothing parameters $\{\bar{\tau}_k\}$ based on the values assumed by the constraints $g$ at a reference point $\bar{x}$ and a Monte Carlo sample of the random vector $\xi$. 
The purpose of this rescaling step is to ensure that the initial approximation~\eqref{eqn:app} with smoothing parameter
$\tau_1$ is `well-scaled' at the initial guess $\bar{x}$. 
By `well-scaled', we mean that the smoothing parameter $\tau_1$ is chosen (neither too large, nor too small) such that the realizations of the scaled constraint functions $\bar{g}_j(\bar{x},\xi) := \frac{g_j(\bar{x},\xi)}{\tau_{1,j}}$, $j \in \{1,\cdots,m\}$, are mostly supported on the interval $[-1,1]$ but not concentrated at zero.
This rescaling ensures that stochastic subgradients at $\bar{x}$ are not too small in magnitude.
If $\bar{x}$ is `near' a stationary solution, then we can hope that the approximation remains well-scaled over a region of interest and that our `first-order algorithm' makes good progress in a reasonable number of iterations.
Note that the form of the scaling factor $\beta_j$ in Algorithm~\ref{alg:scaling} is influenced by our choice of the smoothing functions in Example~\ref{exm:oursmoothapprox}.
Algorithm~\ref{alg:steplength} uses the step length rule in~\citet[Page~6]{davis2018stochastic} with the trivial bound `$R = 1$' and sample estimates of the weak convexity parameter of $\ph_k$ and the parameter $\sigma^2_k$ related to its stochastic subgradient.
Since stochastic approximation algorithms are infamous for their sensitivity to the choice of step lengths (see Section~2.1 of~\citet{nemirovski2009robust}, for instance), Algorithm~\ref{alg:updatesteplength} prescribes heuristics for updating the step length based on the progress of Algorithm~\ref{alg:propalgo} over multiple runs. These rules increase the step length if insufficient progress has been made, 
%over the last few runs, 
and decrease the step length if the risk level has increased too significantly. 
Algorithm~\ref{alg:updatesteplength} also suggests the following heuristic for terminating the `runs loop' within Algorithm~\ref{alg:propalgo}: terminate either if the upper limit on the number of runs is hit, or if insufficient progress has been made over the last few runs.

\begin{algorithm}
\caption{Scaling the smoothing parameters}
\label{alg:scaling}
{
%\fontsize{11}{14}\selectfont
\begin{algorithmic}[1]

\State \textbf{Input}: reference point $\bar{x} \in X$, and sequence of smoothing parameters $\{\bar{\tau}_k\}_{k=1}^{K}$.

\State \textbf{Set algorithmic parameters}: number of samples $N_{scale}$, scaling tolerance $s_{tol} > 0$, and scaling factor $\omega > 0$.

\State \textbf{Output}: smoothing parameters $\{\tau_k\}_{k=1}^{K}$ scaled at $\bar{x}$.

\State Draw a sample $\{\xi^{i}\}_{i=1}^{N_{scale}}$ of $\xi$ from $\P$.

\State For each $j \in \{1,\cdots,m\}$ and $k$, set $\tau_{k,j} = \beta_j \bar{\tau}_{k}$, where $\beta_j := \omega \max\left\lbrace \median{\abs{g_j(\bar{x},\xi^i)}}, s_{tol} \right\rbrace$.

\end{algorithmic}
}
\end{algorithm}

\begin{algorithm}
\caption{Determining initial step length}
\label{alg:steplength}
{
%\fontsize{11}{14}\selectfont
\begin{algorithmic}[1]

\State \textbf{Input}: objective bound $\nu$, reference point $\bar{x} \in \xv$, mini-batch size $M$, maximum number of iterations $N_{max}$, minimum number of runs $R_{min}$, and smoothing parameters $\{\tau_k\}_{k=1}^{K}$.

\State \textbf{Set algorithmic parameters}: number of samples for estimating weak convexity parameter $N_{wc}$, sampling radius $r > 0$, number of samples for estimating `variability' of stochastic gradients $N_{var}$, and batch size for computing expectations $N_{batch}$.

\State \textbf{Output}: initial step lengths $\{\bar{\gamma}_{k}\}_{k=1}^{K}$.

\State Estimate the weak convexity parameter $\rho_k$ of the objective of~\eqref{eqn:app} with smoothing parameter $\tau_k$ using i.~$N_{wc}$ samples of pairs of points in $\xv \cap B_{r}(\bar{x})$, and ii.~mini-batches of $\xi$ for estimating stochastic gradients (see Section~\ref{subsec:practcon}).

\State Estimate $\sigma^2_k := \max_{x \in \xv} \bexpect{\norm{G(x,\xi)}^2}$ using i.~$N_{var}$ sample points $x \in \xv \cap B_{r}(\bar{x})$, and ii.~stochastic $B$-subdifferential elements $G(x,\xi)$ of the subdifferential of the objective of~\eqref{eqn:app} with smoothing parameter $\tau_k$ and mini-batches of $\xi$ (see Section~\ref{subsec:practcon}).

\State Set $\bar{\gamma}_{k} := \dfrac{1}{\sqrt{\rho_k \sigma^2_k (N_{max} + 1) R_{min}}}$.

\end{algorithmic}
}
\end{algorithm}

\begin{algorithm}[ht!]
\caption{Updating step length and checking termination}
\label{alg:updatesteplength}
{
%\fontsize{11}{14}\selectfont
\begin{algorithmic}[1]

\State \textbf{Input}: current run number $r$, minimum and maximum number of `runs', $R_{min}$ and $R_{max}$, step length $\gamma_k$, estimated risk levels $\{\bar{\alpha}^i\}_{i=1}^{r}$, and the risk level at the initial guess $\hat{\alpha}^0$.
% (where $\hat{\alpha}^0$ is the risk level at the initial guess).

\State \textbf{Set algorithmic parameters}: step length checking frequency $N_{check}$, run termination checking parameter $N_{term}$, relative increase in risk level $\delta_1 > 0$ beyond which we terminate, relative increase in risk level $\delta_2 > \delta_1$ beyond which we decrease the step length by factor $\gamma_{decr} > 1$, 
and step length increase factor $\gamma_{incr} > 1$.
%relative decrease in risk level $\delta_3 > 0$ beyond which we increase the step length by factor $\gamma_{incr} > 1$.

\State Construct the sequence $\{\hat{\alpha}^i\}_{i=0}^{r}$, with $\hat{\alpha}^i := \min\{\hat{\alpha}^{i-1},\bar{\alpha}^i\}$, of smallest risk levels at the end of each run thus far.

\State \textbf{Updating step length}:

\If{run number is divisible by $N_{check}$}

\State Determine maximum relative `decrease' in the estimated risk level
\[
\hat{\delta} := \uset{i \in \{1,\cdots,N_{check}\}}{\max} \frac{\hat{\alpha}^{r- \scaleto{N_{check}}{4pt}} - \bar{\alpha}^{r+1-i}}{\hat{\alpha}^{r-\scaleto{N_{check}}{4pt}}}
\]
\Statex \hspace*{0.2in} over the past $N_{check}$ runs, relative to the smallest known estimate of the risk level $N_{check}$ runs ago.

%\If{$\hat{\delta} \in (-\delta_1,\delta_3)$}
\If{$\hat{\delta} \geq -\delta_1$}
\State Increase step length by factor $\gamma_{incr}$. \Comment{Insufficient decrease in estimated risk level}
\ElsIf{$\hat{\delta} \leq -\delta_2$}
\State Decrease step length by factor $\gamma_{decr}$. \Comment{Unacceptable increase in estimated risk level}
\EndIf

\EndIf

\State \textbf{Termination check}:

\If{number of runs equals $R_{max}$}
\State Terminate.
\ElsIf{number of runs is at least $R_{min}$}
\State Determine maximum relative `decrease' in the estimated risk level over the past $N_{term}$ runs
\[
\hat{\delta} := \uset{i \in \{1,\cdots,N_{term}\}}{\max} \frac{\hat{\alpha}^{r- \scaleto{N_{term}}{4pt}} - \bar{\alpha}^{r+1-i}}{\hat{\alpha}^{r-\scaleto{N_{term}}{4pt}}}.
\]
\Statex \hspace*{0.2in} If $\hat{\delta} < -\delta_1$, terminate due to insufficient decrease in risk level.
\EndIf

%\Until{termination}

\end{algorithmic}
}
\end{algorithm}

\subsection{Other practical considerations}
\label{subsec:practcon}

We list some other practical considerations for implementing Algorithm~\ref{alg:propalgoef} below.

\paragraph{Setting an initial guess:} The choice of the initial point $\hat{x}^0$ for Algorithm~\ref{alg:propalgoef} is important especially because Algorithm~\ref{alg:propalgo} does not guarantee finding global solutions to~\eqref{eqn:app}. Additionally, a poor choice of the initial bound $\bar{\nu}^0$ can lead to excessive effort expended on uninteresting regions of the efficient frontier. We propose to initialize $\hat{x}^0$ and $\bar{\nu}^0$ by solving a set of tuned scenario approximation problems. For instance, Algorithm~\ref{alg:scenapprox} in the appendix can be solved with $M = 1$, $R = 1$, and a tuned sample size $N_1$ (tuned such that the output $(\bar{\nu}^{1,1},\bar{\alpha}^{1,1},\bar{x}^{1,1})$ of this algorithm corresponds to an interesting region of the efficient frontier) to yield the initial guess $\hat{x}^0 = \bar{x}^{1,1}$ and $\bar{\nu}^0 = \bar{\nu}^{1,1}$. Note that it may be worthwhile to solve the scenario approximation problems to global optimality to obtain a good initialization $(\hat{x}^0,\bar{\nu}^0)$, e.g., when the number of scenarios $N_1$ is not too large (in which case they may be solved using off-the-shelf solvers/tailored decomposition techniques in acceptable computation times).
Case study~\ref{case:resource} in Section~\ref{sec:computexp} provides an example where globally optimizing the scenario approximation problems to obtain an initial guess yields a significantly better approximation of the efficient frontier.

\paragraph{Computing projections:} Algorithm~\ref{alg:propalgoef} requires projecting onto the sets $\xv$ for some $\nu
\in \R$  multiple times while executing Algorithm~\ref{alg:propalgo}. Thus, it may be beneficial to implement tailored projection subroutines~\citep{censor2012effectiveness}, especially for sets $\xv$ with special structures~\citep{condat2016fast} (also see Case study~\ref{case:portfolio_var} in Section~\ref{sec:computexp}). If the set $\xv$ is a polyhedron, then projection onto $\xv$ may be carried out by solving a quadratic program.

\paragraph{Computing stochastic subgradients:} Given a point $\bar{x} \in \xv$ and a mini-batch $\{\xi^l\}_{l=1}^{M}$ of the random vector $\xi$, a stochastic element $G(\bar{x},\xi)$ of the $B$-subdifferential of the objective of~\eqref{eqn:app} with smoothing parameter $\tau_k$ at $\bar{x}$ can be obtained as
\[
G(\bar{x},\xi) = \dfrac{1}{M} \sum_{l=1}^{M} d\phi_{k,j(l)}\left(g_{j(l)}(\bar{x},\xi^l)\right) \nabla_x g_{j(l)}(\bar{x},\xi^l),
\]
where for each $l \in \{1,\cdots,M\}$, $j(l) \in \{1,\cdots,m\}$ denotes a constraint index satisfying $\max\left[\phi_k\left(g(\bar{x},\xi^l)\right)\right] = \phi_{k,j(l)}\left(g_{j(l)}(\bar{x},\xi^l)\right)$. Note that sparse computation of stochastic subgradients can speed Algorithm~\ref{alg:propalgoef} up significantly. 

\paragraph{Estimating risk levels:} Given a candidate solution $\bar{x} \in \xv$ and a sample $\{\bar{\xi}^l\}_{l=1}^{N_{MC}}$ of the random vector $\xi$, a stochastic upper bound on the risk level $p(\bar{x})$ can be obtained as
\[
\bar{\alpha} = \uset{\alpha \in [0,1]}{\max} \Set{\alpha}{\sum_{i=0}^{N_{viol}} \binom{N_{MC}}{i} \alpha^i (1-\alpha)^{N_{MC} - i} = \delta},
\]
where $N_{viol}$ is the cardinality of the set $\Set{l}{g(\bar{x},\bar{\xi}^l) \not\leq 0}$ and $1-\delta$ is the required confidence level, see~\citet[Section~4]{nemirovski2006convex}. 
%Note that the above one-dimensional root finding problem can be solved using the bisection method with initial lower and upper bounds
%\[
%\alpha^L = \max\left\lbrace \dfrac{N_{viol} - \displaystyle\sqrt{-0.5N_{MC}\log(\delta)}}{N_{MC}}, 0\right\rbrace, \quad \alpha^U = \min\left\lbrace\dfrac{N_{viol} + \displaystyle\sqrt{-0.5N_{MC}\log(\delta)}}{N_{MC}}, 1\right\rbrace
%\] 
%that are obtained from tail bounds on the binomial distribution. 
Since checking the satisfaction of $N_{MC}$ constraints at the end of each run in Algorithm~\ref{alg:propalgo} may be time consuming (especially for problems with recourse structure), we use a smaller sample size (determined based on the risk lower bound $\alpha_{low}$) to estimate risk levels during the course of Algorithm~\ref{alg:propalgo} and only use all $N_{MC}$ samples of $\xi$ to estimate the risk level $\{\bar{\alpha}^i\}$ of the final solutions $\{\hat{x}^i\}$ in Algorithm~\ref{alg:propalgoef}.
If our discretization of the efficient frontier obtained from Algorithm~\ref{alg:propalgoef} consists of $N_{EF}$ points each of whose risk levels is estimated using a confidence level of $1-\delta$, then we may conclude that our approximation of the efficient frontier is `achievable' with a confidence level of at least $1 - \delta N_{EF}$ using Bonferroni's inequality.

\paragraph{Estimating weak convexity parameters:} Proposition~\ref{prop:weakconvparam} shows that the Lipschitz constant
$\bar{L}_k$ of the Jacobian $\bexpect{\nabla\phi_k\left(g(\cdot,\xi)\right)}$ on $\xv$ provides a conservative estimate
of the weak convexity parameter $\rho_k$ of $\ph_k$ on $\xv$. Therefore, we use an estimate of the Lipschitz constant
$\bar{L}_k$ as an estimate of $\rho_k$. To avoid overly conservative estimates, we restrict the estimation of
$\bar{L}_k$ to a neighborhood of the reference point $\bar{x}$ in Algorithm~\ref{alg:steplength} to get at the local
Lipschitz constant of the Jacobian $\bexpect{\nabla\phi_k\left(g(\cdot,\xi)\right)}$.

\paragraph{Estimating $\sigma^2_k$:} For each point $x$ sampled from $\xv$, we compute multiple realizations of the
mini-batch stochastic subdifferential element $G(x,\xi)$ outlined above to estimate $\bexpect{\norm{G(x,\xi)}^2}$.
Once again, we restrict the estimation of $\sigma^2_k$ to a neighborhood of the reference point $\bar{x}$ in Algorithm~\ref{alg:steplength} to get at the local `variability' of the stochastic subgradients.

\paragraph{Estimating initial step lengths:} Algorithm~\ref{alg:steplength} estimates initial step lengths for the
sequence of approximations~\eqref{eqn:app} by estimating the parameters $\{\rho_k\}$ and $\{\sigma^2_k\}$, which in turn
involve estimating subgradients and the Lipschitz constants of $\{\ph_k\}$. Since the numerical conditioning of the
approximating problems~\eqref{eqn:app} deteriorates rapidly as the smoothing parameters $\tau_k$ approach zero (see
Figure~\ref{fig:compsmoothfns}), obtaining good estimates of these constants via sampling becomes challenging when $k$
increases. To circumvent this difficulty, we only estimate the initial step length $\bar{\gamma}_1$ for the initial approximation~\eqref{eqn:app} with smoothing parameter $\tau_1$ by sampling, and propose the conservative initialization $\bar{\gamma}_k := \left(\frac{\tau_k}{\tau_1}\right)^2\bar{\gamma}_1$ for $k > 1$ (see Propositions~\ref{prop:oursmoothapproxconst},~\ref{prop:weakconvparam},~\ref{prop:lipconstcomp}, and~\ref{prop:stochsubgradvar} for a justification).

\section{Computational study}
\label{sec:computexp}

We present implementation details and results of our computational study in this section. We use the abbreviation `EF' for the efficient frontier throughout this section.

\subsection{Implementation details}
\label{subsec:impldet}

The following parameter settings are used for testing our stochastic approximation method:
\begin{itemize}[itemsep=0em]
\item Algorithm~\ref{alg:propalgoef}: Obtain initial guess $\hat{x}^0$ and initial objective bound $\bar{\nu}^0$ using Algorithm~\ref{alg:scenapprox} in Appendix~\ref{subsec:scenapproximpl} with $M = 1$, $R = 1$, and $N_1 = 10$. Set $\tilde{\nu} = 0.005\abs{\bar{\nu}^0}$ (we assume that $\bar{\nu}^0 \neq 0$ for simplicity) and $\alpha_{low} = 10^{-4}$ unless otherwise specified. In addition, set $M = 20$, $N_{max} = 1000$, $R_{min} = 10$, $R_{max} = 50$, $K = 3$, and $\{\bar{\tau}_k\} = \{(0.1)^{k-1}\}$ (i.e., $\tau_c = 0.1$ in Equation~\eqref{eqn:smoothingapproxfactor}).

\item Algorithm~\ref{alg:scaling}: $N_{scale} = 10^4$, $s_{tol} = 10^{-6}$, and $\omega = 1$.

\item Algorithm~\ref{alg:steplength}: $N_{wc} = N_{var} = 200$, $N_{batch} = 20$, and $r = 0.1\norm{\bar{x}}$ (assuming that $\bar{x} \neq 0$).

\item Algorithm~\ref{alg:updatesteplength}: $N_{check} = 3$, $N_{term} = 5$, $\delta_1 = 10^{-4}$, $\delta_2 = 10^{-2}$, and $\gamma_{incr} = \gamma_{decr} = 10$.

\item Projecting onto $\xv$: Case study~\ref{case:portfolio}: using the algorithm of~\citet{condat2016fast}; Case study~\ref{case:portfolio_var}: numerical solution of the KKT conditions (see Appendix~\ref{subsec:tailproj_portvar} for details); Case studies~\ref{case:normopt} and~\ref{case:resource} and Case study~\ref{case:normopt2} in the appendix: by solving a quadratic program.

\item Estimating risk level $p(x)$: Case studies~\ref{case:portfolio} and~\ref{case:portfolio_var}: analytical solution; Case studies~\ref{case:normopt} and~\ref{case:resource}: Monte Carlo estimate using $N_{MC} = 10^5$ and reliability level $\delta = 10^{-6}$; Case study~\ref{case:normopt2} in the appendix: numerical estimation of the exact risk level based on~\citet{ruben1962probability}.

\item Case study~\ref{case:portfolio_var}: $\tilde{\nu} = 0.02\abs{\bar{\nu}^0}$

\item Case studies~\ref{case:normopt} and~\ref{case:resource}: $\alpha_{low} = 5\times 10^{-4}$.% (since it is memory intensive to store a large Monte Carlo sample for Case study~\ref{case:normopt}, and it is time intensive to estimate the risk level for a large Monte Carlo sample for Case study~\ref{case:resource}).
\end{itemize}
%
%As discussed in Section~\ref{sec:introduction}, the bounds afforded by the papers~\citet{calafiore2005uncertain,campi2011sampling}, and~\citet{luedtke2008sample} on the number of samples required to `ensure' feasibility for~\eqref{eqn:ccp} are typically quite conservative. 
We compare the results of our proposed approach with a tuned application of scenario approximation that enforces a
predetermined number of random constraints, solves the scenario problem to local/global optimality, and determines an
a~posteriori estimate of the risk level at the solution of the scenario problem using an independent Monte Carlo sample
(see Section~\ref{subsec:practcon}) to estimate a point on the EF of~\eqref{eqn:ccp}. Algorithm~\ref{alg:scenapprox} in
Appendix~\ref{app:algos} presents pseudocode for approximating the EF of~\eqref{eqn:ccp} by solving many scenario
approximation problems using an iterative (cutting-plane) approach. 

Our codes are written in Julia~0.6.2~\citep{bezanson2017julia}, use Gurobi~7.5.2~\citep{gurobi} to solve linear, quadratic, and second-order cone programs, use IPOPT~3.12.8~\citep{wachter2006implementation} to solve nonlinear scenario approximation problems (with MUMPS~\citep{amestoy2000mumps} as the linear solver), and use SCIP~6.0.0~\citep{gleixner2018} to solve nonlinear scenario approximation problems to global optimality (if necessary). The above solvers were accessed through the JuMP~0.18.2 modeling interface~\citep{dunning2017jump}.
All computational tests\footnote{The scenario approximation problems solved to global optimality using SCIP in Case study~\ref{case:resource} were run on a different laptop running Ubuntu 16.04 with a 2.6 GHz four core Intel i7 CPU, 8 GB of RAM due to interfacing issues on Windows.} were conducted on a Surface Book 2 laptop running Windows 10 Pro with a $1.90$~GHz four core Intel~i7 CPU, $16$~GB of RAM.

\subsection{Numerical experiments}
\label{subsec:numexpts}

We tested our approach on four test cases from the literature. We present basic details of the test instances below.
The \texttt{Julia} code and data for the test instances are available at \url{https://github.com/rohitkannan/SA-for-CCP}.
For each test case, we compared the EFs generated by the stochastic approximation method against the solutions obtained
using the tuned scenario approximation method and, if available, the analytical EF. Because the output of our proposed
approach is random, we present enclosures of the EF generated by our proposal over ten different replicates in
Appendix~\ref{app:computexp} for each case study. These results indicate that the output of the proposed method does not vary significantly across the replicates.

\begin{casestudy}
\label{case:portfolio}
This portfolio optimization instance is based on Example~2.3.6 in~\citet{ben2009robust}, and includes a single individual linear chance constraint:
\begin{alignat*}{2}
&\uset{t, \: x \in \Delta_N}{\max} \:\: && t \\
&\quad \text{s.t.} && \prob{\tr{\xi} x \geq t} \geq 1 - \alpha,
\end{alignat*}
where $\Delta_N := \Set{y \in \R^N_{+}}{\sum_i y_i = 1}$ is the standard simplex, $x_i$ denotes the fraction of investment in stock $i \in \{1,\cdots,N\}$, $\xi \sim \P := \mathcal{N}(\mu,\Sigma)$ is a random vector of returns with joint normal probability distribution, $\mu_i = 1.05 + 0.3\frac{N-i}{N-1}$ and $\sigma_i = \frac{1}{3}\left(0.05 + 0.6\frac{N-i}{N-1}\right)$, $i \in \{1,\cdots,N\}$, and $\Sigma = \text{diag}(\sigma^2_1,\cdots,\sigma^2_N)$. We consider the instance with number of stocks $N = 1000$.
We assume the returns $\xi$ are normally distributed so that we can benchmark our approach against an analytical solution for the true EF when the risk level $\alpha \leq 0.5$ (e.g., see~\citet[Theorem~10.4.1]{prekopa2013stochastic}).

Figure~\ref{fig:portfolio} compares a typical EF obtained using our approach against the analytical EF and the solutions generated by the tuned scenario approximation algorithm.
Our proposal is able to find a very good approximation of the true EF, whereas the scenario approximation method finds
solutions that can be improved either in objective or risk. Our proposed approach took $388$ seconds on average (and a maximum of $410$ seconds) to approximate the EF using $26$ points, whereas the tuned scenario approximation method took a total of $6900$ seconds to generate its $1000$ points in Figure~\ref{fig:portfolio}.
Note that even though the above instance of~\eqref{eqn:ccp} does not satisfy Assumption~\ref{ass:xcompact} as written (because $X_{\nu} := \Set{(x,t)}{x \in \Delta_N, t \geq \nu}$ is not compact), we can essentially incorporate the upper bound $t \leq \nu$ within the definition of the set $X_{\nu}$ because the solution to Problem~\eqref{eqn:app} always satisfies $t^* = \nu$.

Since existing smoothing-based approaches provide the most relevant comparison, we compare our results with one such approach from the literature. We choose the smoothing-based approach of~\citet{cao2018sigmoidal} for comparison because: i.~they report encouraging computational results in their work relative to other such approaches, and ii.~they have made their implementation available.
Table~\ref{tab:sigvar} summarizes typical results of the sigmoidal smoothing approach of~\citet{cao2018sigmoidal} when applied to the above instance with a specified risk level of $\alpha = 0.01$, with a varying number of scenarios, and with different settings for the scaling factor $\gamma$ of the sigmoidal approximation method (see Algorithm~\texttt{SigVar-Alg} of~\citet{cao2018sigmoidal}). 
Appendix~\ref{subsec:sigapproximpl} lists details of our implementation of this method.
The second, third, and fourth columns of Table~\ref{tab:sigvar} present the overall solution time in seconds (or a failure status returned by IPOPT), the best objective value and the true risk level of the corresponding solution returned by the method over two replicates (we note that there was significant variability in solution times over the replicates; we report the results corresponding to the smaller solution times).
Figure~\ref{fig:portfolio} plots the solutions returned by this method (using the values listed in Table~\ref{tab:sigvar}).
The relatively poor performance of Algorithm~\texttt{SigVar-Alg} on this example is not surprising; even the instance of the sigmoidal approximation problem with only a hundred scenarios (which is small for $\alpha = 0.01$) has more than a thousand variables and a hundred thousand nonzero entries in the (dense) Jacobian. Therefore, tailored approaches have to be developed for solving these problems efficiently.
Since our implementation of the proposal of~\citet{cao2018sigmoidal} failed to perform well on this instance even for a single risk level, 
%with a reasonable amount of effort, 
we do not compare against their approach for the rest of the case studies.
\end{casestudy}

\begin{figure}[ht!]
\begin{center}
\caption{Comparison of the efficient frontiers for Case study~\ref{case:portfolio}.} \label{fig:portfolio}
\includegraphics[width=0.85\linewidth]{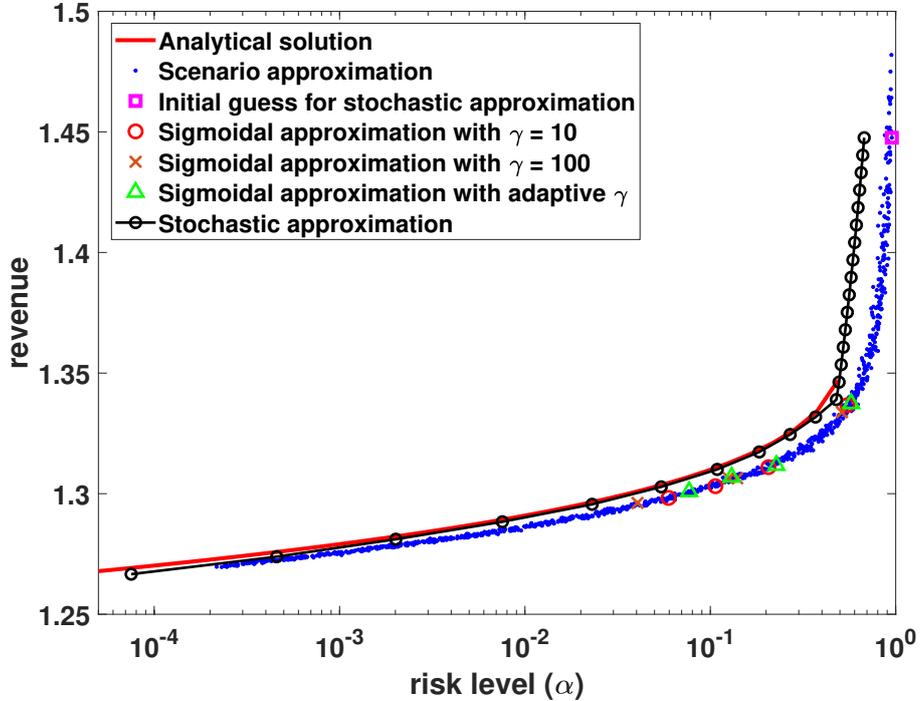}
\end{center}
\end{figure}

\begin{table}[ht!]
\caption{Performance of the sigmoidal smoothing approach of~\protect\citet{cao2018sigmoidal} on Case study~\ref{case:portfolio}. The entry for each column of $\gamma$ indicates the time in seconds (or IPOPT status)/best objective value/best risk level returned by the sigmoidal approximation method over two replicates. The entries $\protect\infeasible$ and $\protect\tle$ denote infeasible and time limit exceeded statuses returned by IPOPT.}
\label{tab:sigvar}
\begin{center}
    \begin{tabular}{| c | c c c | c c c | c c c |}
    \hline
    Num.~scenarios & & $\gamma = 10$ & & & $\gamma = 100$ & & & $\gamma$ adaptive & \\ \hline
     & Time  & Obj  & Risk & Time  & Obj  & Risk & Time & Obj & Risk \\ \hline
    100 & 34 & 1.337 & 0.548 & 23 & 1.334 & 0.513 & 29 & 1.337 & 0.575 \\ \hline
    500 & 508 & 1.311 & 0.206 & 106 & 1.306 & 0.139 & 59 & 1.312 & 0.227 \\ \hline
    1000 & 606 & 1.303 & 0.106 & 1203 & 1.306 & 0.125 & 1011 & 1.307 & 0.130 \\ \hline
    2000 & 1946 & 1.298 & 0.059 & 296 & 1.296 & 0.040 & 2787 & 1.301 & 0.077 \\ \hline
    5000 & 4050 & 1.226 & $2.4 \times 10^{-8}$ & 6289 & 1.292 & 0.025 & & $\infeasible$ & \\ \hline
    10000 & 2708 & 1.199 & $5.1 \times 10^{-11}$ & & $\tle$ & & & $\tle$ & \\ \hline
    \end{tabular}
\end{center}
\end{table}

\begin{casestudy}
\label{case:portfolio_var}
We consider the following variant of Case study~\ref{case:portfolio} where the variance of the portfolio is minimized subject to the constraint that the return is larger than a fixed threshold at least with a given probability:
%on the $100\alpha$ percentile of the return:
\begin{alignat*}{2}
&\uset{x \in \Delta_N}{\min} \:\: && \tr{x} \Sigma x \\
&\quad \text{s.t.} && \prob{\tr{\xi} x \geq \bar{t}} \geq 1 - \alpha,
\end{alignat*}
where we require that the $100\alpha$ percentile of the return is at least $\bar{t} = 1.2$.
Once again, we consider the instance with number of stocks $N = 1000$, and assume the returns $\xi$ are normally distributed as in Case study~\ref{case:portfolio} so that we can benchmark our approach against the analytical solution for the true EF.
% when the risk level $\alpha \leq 0.5$.

We found that using a general-purpose NLP solver to solve the quadratically-constrained quadratic
program to compute projections onto the set $X_{\nu} := \Set{x \in \Delta_N}{\tr{x} \Sigma x \leq \nu}$ during the
course of Algorithm~\ref{alg:propalgoef} was a bottleneck of the algorithm in this case. Thus, we implemented a tailored projection routine that exploits the structure of this projection problem.
Appendix~\ref{subsec:tailproj_portvar} presents details of the projection routine.
Figure~\ref{fig:portfolio_var} compares a typical EF obtained using our approach against the analytical EF and the solutions generated by the tuned scenario approximation algorithm.
Our proposal is able to generate solutions that lie on or very close to the true EF, whereas the scenario approximation method finds
solutions that can be improved either in objective or risk. Our proposed approach took $1556$ seconds on average (and a maximum of $1654$ seconds) to approximate the EF using $20$ points, whereas the tuned scenario approximation method took a total of $36259$ seconds to generate its $1000$ points in Figure~\ref{fig:portfolio_var}.
\end{casestudy}

\begin{figure}[ht!]
\begin{center}
\caption{Comparison of the efficient frontiers for Case study~\ref{case:portfolio_var}.} \label{fig:portfolio_var}
\includegraphics[width=0.85\linewidth]{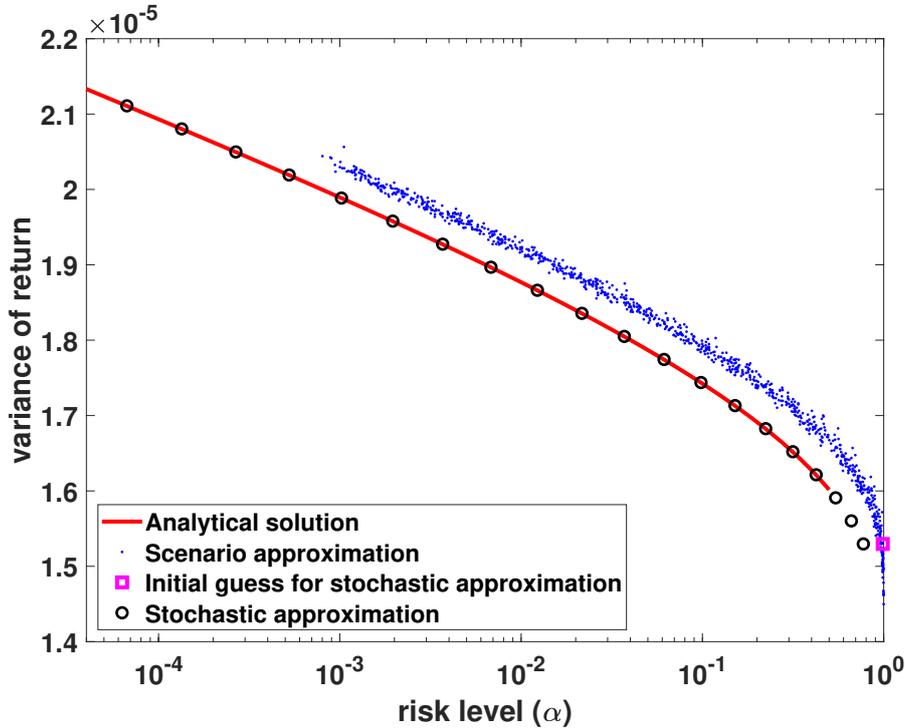}
\end{center}
\end{figure}

\begin{casestudy}
\label{case:normopt}
This norm optimization instance is based on Section~5.1.2 of~\citet{hong2011sequential}, and includes a joint convex
nonlinear chance constraint:
\begin{alignat*}{2}
&\uset{x \in \R^n_+}{\min} \:\: && -\sum_i x_i \\
&\:\: \text{s.t.} && \prob{\sum_i \xi^2_{ij} x^2_i \leq U^2, \:\: j = 1,\cdots,m} \geq 1 - \alpha, 
\end{alignat*}
where $\xi_{ij}$ are dependent normal random variables with mean $\tfrac{j}{d}$ and variance $1$, and cov$(\xi_{ij},\xi_{i^{'}j}) = 0.5$ if $i \neq i^{'}$, cov$(\xi_{ij},\xi_{i^{'}j^{'}}) = 0$ if $j \neq j^{'}$. We consider the instance with number of variables $n = 100$, number of constraints $m = 100$, and bound $U = 100$.
Figure~\ref{fig:normopt} compares a typical EF obtained using our approach against the solutions generated by the tuned scenario approximation algorithm.
Once again, our proposed approach is able to find a significantly better approximation of the EF than the scenario approximation method, although it is unclear how our proposal fares compared to the true EF since it is unknown in this case. Our proposal took $6742$ seconds on average (and a maximum of $7091$ seconds) to approximate the EF using $31$ points, whereas it took tuned scenario approximation a total of $25536$ seconds to generate its $1000$ points in Figure~\ref{fig:normopt}.
We note that more than $70\%$ of the reported times for our method is spent in generating random numbers because the random variable $\xi$ is high-dimensional and the covariance matrix of the random vector $\xi_{{\cdot}j}$ is full rank. A practical instance might have a covariance matrix rank that is at least a factor of ten smaller, which would reduce our overall computation times roughly by a factor of three (cf.~the smaller computation times reported for the similar Case study~\ref{case:normopt2} in the appendix).
Appendix~\ref{app:computexp} benchmarks our approach against the true EF when the random variables $\xi_{ij}$ are assumed to be i.i.d.~in which case the analytical solution is known (see Section~5.1.1 of~\citet{hong2011sequential}).
Once again, note that Assumption~\ref{ass:xcompact} does not hold because the set $X_{\nu} = \Set{x \in \R^n_+}{\sum_i x_i \geq -\nu}$ is not compact; however, we can easily deduce an upper bound on each $x_i$, $i \in \{1,\cdots,n\}$, that is required for feasibility of the chance constraint for the largest (initial) risk level of interest and incorporate this upper bound within the definition of $X_{\nu}$ without altering the EF.
%In practice, we find that the iterates explored by the algorithm always lie within a compact subset of $X_{\nu}$ for this instance.
\end{casestudy}

\begin{figure}
\begin{center}
\caption{Comparison of the efficient frontiers for Case study~\ref{case:normopt}.} \label{fig:normopt}
\includegraphics[width=0.85\linewidth]{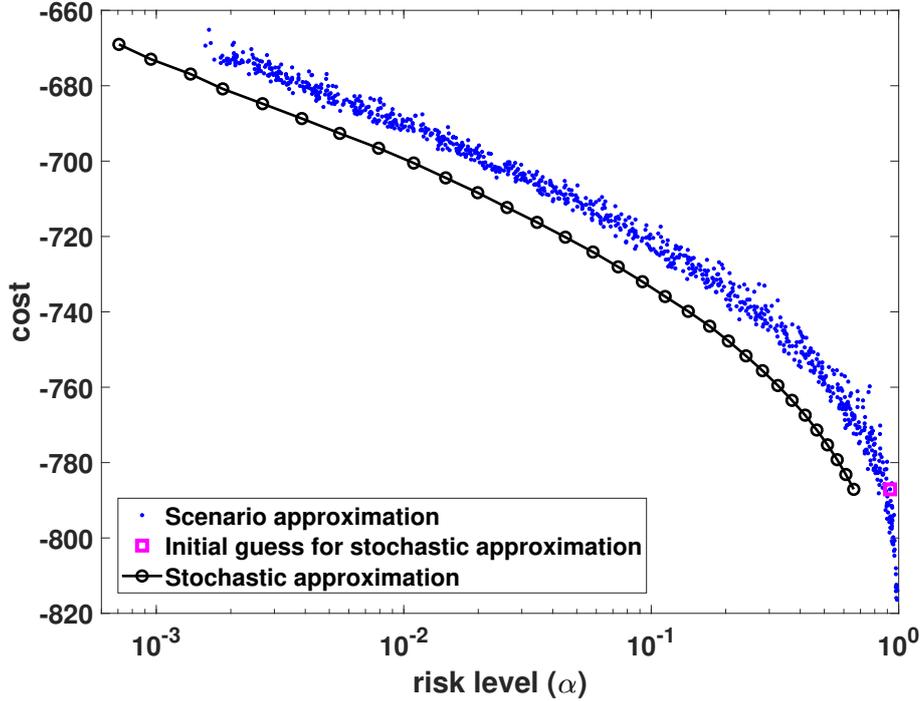}
\end{center}
\end{figure}

\begin{casestudy}
\label{case:resource}
This probabilistic resource planning instance is based on Section~3 of~\citet{luedtke2014branch}, and is modified to include a nonconvex recourse constraint:
\begin{alignat*}{2}
&\uset{x \in \R^n_+}{\min} \:\: && \tr{c} x \\
&\:\: \text{s.t.} && \prob{x \in R(\lambda,\rho)} \geq 1 - \alpha, 
\end{alignat*}
where $x_i$ denotes the quantity of resource $i$, $c_i$ denotes the unit cost of resource $i$,
\[
R(\lambda,\rho) = \Set{x \in \R^n_+}{\exists y \in \R^{n n_c}_+ \text{ s.t. } \sum_{j=1}^{n_c} y_{ij} \leq \rho_i x^2_i, \: \forall i \in \{1,\cdots,n\}, \:\: \sum_{i=1}^{n} \mu_{ij} y_{ij} \geq \lambda_j, \: \forall j \in \{1,\cdots,n_c\}},
\]
$n_c$ denotes the number of customer types, $y_{ij}$ denote the amount of resource $i$ allocated to customer type $j$,
$\rho_i \in (0,1]$ is a random variable that denotes the yield of resource $i$, $\lambda_j \geq 0$ is a random variable
that denotes the demand of customer type $j$, and $\mu_{ij} \geq 0$ is a deterministic scalar that denotes the service
rate of resource $i$ for customer type $j$. Note that the nonlinear term $\rho_i x^2_i$ in the definition of $R(\lambda,\rho)$ is a modification of the corresponding linear term $\rho_i x_i$ in~\citet{luedtke2014branch}. This change could be interpreted as a reformulation of the instance in~\citet{luedtke2014branch} with concave objective costs (due to economies of scale).
We consider the instance with number of resources $n = 20$ and number of customer types $n_c = 30$. Details of how the parameters of the model are set (including details of the random variables) can be found in the electronic companion to~\citet{luedtke2014branch}.

Figure~\ref{fig:resource} compares a typical EF obtained using our approach against the solutions generated by the tuned scenario approximation algorithm.
The blue dots in the top part of Figure~\ref{fig:resource} correspond to the $1000$ points obtained using the scenario
approximation method when IPOPT is used to solve the scenario approximation problems. We mention that the vertical axis
has been truncated for readability; the scenario approximation yields some solutions that are further away from the EF (with objective values up to $110$ for the risk levels of interest).
The $1000$ red dots in the bottom part of Figure~\ref{fig:resource} correspond to the points obtained by solving the scenario approximation problems to global optimality using SCIP.
The top black curve (with circles) corresponds to the EF obtained using the stochastic approximation method when it is initialized using the IPOPT solution of a scenario approximation problem. When a better initial point obtained by solving a scenario approximation problem using SCIP is used, the stochastic approximation method generates the bottom green curve (with squares) as its approximation of the EF.

Several remarks are in order. The scenario approximation solutions generated using the local solver IPOPT are very scattered possibly because IPOPT gets stuck at suboptimal local minima. However, the best solutions generated using IPOPT provide a comparable approximation of the EF as the stochastic approximation method that is denoted by the black curve with circles (which appears to avoid the poor local minima encountered by IPOPT). 
We note that IPOPT finds a good local solution at one of the $1000$ scenario approximation runs (indicated by the blue circle). 
Solving the scenario approximation problems to global optimality using SCIP yields a much better approximation of the EF
than the local solver IPOPT. Additionally, when the proposed approach is initialized using the global solution obtained
from SCIP for a single
scenario approximation problem (which took less than $60$ seconds to compute), it generates a significantly better
approximation of the EF that performs comparably to the EF generated using the global solver SCIP.
We also tried initializing the solution of IPOPT with the one global solution from SCIP, and  this did not yield a
better approximation of the EF. 
Our approach took $9148$ seconds on average (and a maximum of $9439$ seconds) to generate the green curve (with squares) approximation of the EF using $26$ points,
whereas it took the blue scenario approximations (solved using IPOPT) and the red scenario approximations (solved using SCIP) a total of $110291$ seconds and $145567$ seconds, respectively, to generate their $1000$ points in Figure~\ref{fig:resource}.
\end{casestudy}

\begin{figure}
\begin{center}
\caption{Comparison of the efficient frontiers for Case study~\ref{case:resource}.} \label{fig:resource}
\includegraphics[width=0.85\linewidth]{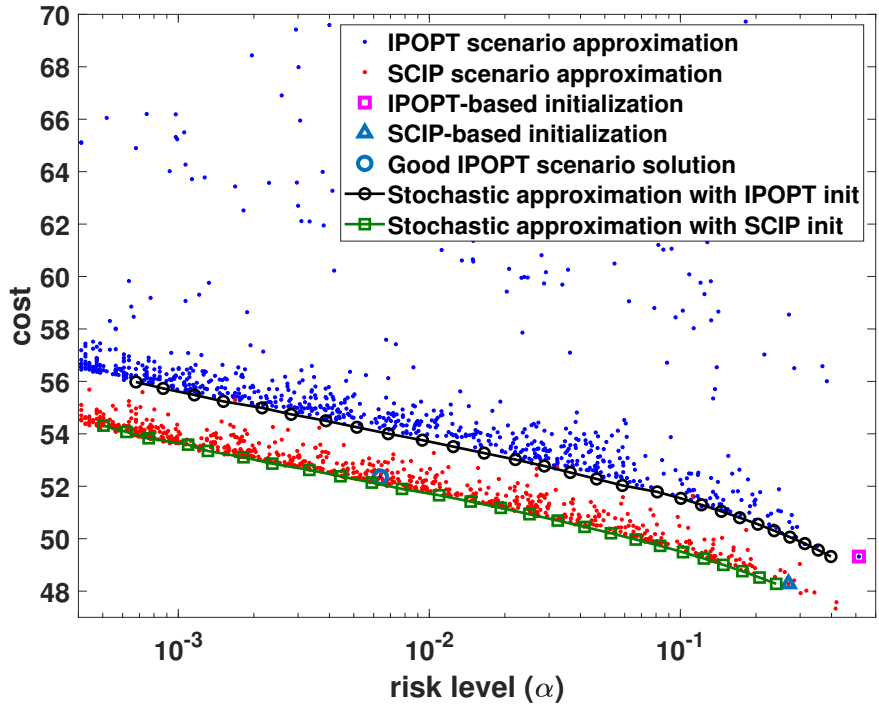}
\end{center}
\end{figure}

\section{Conclusion and future work}
\label{sec:conclandfw}

We proposed a stochastic approximation algorithm for estimating the efficient frontier of chance-con\-strained NLPs.
Our proposal involves solving a sequence of partially smoothened stochastic optimization problems to local optimality using a projected stochastic subgradient algorithm.
We established that every limit point of the sequence of stationary/global solutions of the above sequence of approximations yields a stationary/global solution of the original chance-constrained program with an appropriate risk level. 
A potential advantage of our proposal is that it can find truly stationary solutions of the chance-constrained NLP unlike scenario-based approaches that may get stuck at spurious local optima generated by sampling.
Our computational experiments demonstrated that our proposed approach is consistently able to determine good approximations of the efficient frontier in reasonable computation times.

Extensions of our proposal that can handle multiple sets of joint chance constraints merit further investigation.
One option is to minimize the maximum risk of constraints violation over the various sets of joint chance constraints, which can be formulated as a minimax stochastic program (cf.~Section~3.2 of~\citet{nemirovski2009robust}). This can in turn be reformulated as a weakly convex-concave saddle point problem that can be solved to stationarity using existing techniques~\citep{rafique2018non}. 
Since our proposal relies on computing projections efficiently, approaches for reducing the computational effort spent on projections, such as random constraint projection techniques, could be explored.
Additionally, extensions that incorporate deterministic nonconvex constraints in a more natural fashion provide an avenue for future work.
The projected stochastic subgradient method of~\citet{davis2018stochastic} has recently been extended to the non-Euclidean case~\citep{zhang2018convergence}, which could accelerate convergence of our proposal in practice.
Finally, because stochastic approximation algorithms are an active area of research, several auxiliary techniques, such as adaptive step sizes, parallelization, acceleration, etc.,~may be determined to be applicable (and practically useful) to our setting.

\section*{Acknowledgements}
The authors thank the anonymous reviewers for suggestions that improved the paper.
R.K.\ also thanks Rui Chen, Eli Towle, and Cl{\'e}ment Royer for helpful discussions.
%\end{acknowledgements}

%The authors declare that they have no conflict of interest.

{
\small
\bibliographystyle{plainnat}
\bibliography{main}
}

\appendix

\section{Algorithm outlines}
\label{app:algos}

\subsection{Implementation of the scenario approximation algorithm}
\label{subsec:scenapproximpl}

Algorithm~\ref{alg:scenapprox} details our implementation of the tuned scenario approximation algorithm. 
Instead of solving each scenario approximation problem instance using an off-the-shelf solver by imposing all of the sampled constraints at once, we adopt a cutting-plane approach that iteratively adds some of the violated scenario constraints at a candidate solution to the list of enforced constraints in a bid to reduce the overall computational effort (cf.\ Algorithm~2.6 in~\citet{bienstock2014chance}).
We use $M = 50$ iterations and $R = 20$ replicates for our computational experiments. Settings for sample sizes $N_i$, constraints added per iteration $N_c$, and number of samples $N_{MC}$ are problem dependent and provided below.
Line 19 in Algorithm~\ref{alg:scenapprox}, which estimates the risk level of a candidate solution, may be replaced, if possible, by one that computes the true analytical risk level or a numerical estimate of it~\citep{van2014gradient}.
\begin{itemize}
\item $N_i = \ceil{10^{a_i}}$, $i = 1,\cdots,50$, where: for Case study~\ref{case:portfolio}: $a_i = 1 + \frac{5}{49}(i-1)$; for Case study~\ref{case:portfolio_var}: $a_i = 1 + \frac{4}{49}(i-1)$; for Case study~\ref{case:normopt}: $a_i = 1 + \frac{\log_{10}(50000)-1}{49}(i-1)$; for Case study~\ref{case:resource}: $a_i = 1 + \frac{4}{49}(i-1)$; and for Case study~\ref{case:normopt2}: $a_i = 1 + \frac{4}{49}(i-1)$ (the upper bounds on $N_i$ were determined based on memory requirements)
\item $N_c$: Case study~\ref{case:portfolio}: $1000$; Case study~\ref{case:portfolio_var}: 100000; Case study~\ref{case:normopt}: $10$; Case study~\ref{case:resource}: $5$; and Case study~\ref{case:normopt2}: $10$ (these values were tuned for good performance)
\item $N_{MC}$: See Section~\ref{subsec:impldet} of the paper (we use the same Monte Carlo samples that were used by our proposed method to estimate risk levels).
\end{itemize}

\begin{algorithm}
\caption{Scenario approximation for approximating the efficient frontier of~\eqref{eqn:ccp}}
\label{alg:scenapprox}
\begin{algorithmic}[1]
\State \textbf{Input}: Number of major iterations $M$, distinct sample sizes $N_i$, $i = 1,\cdots,M$, number of replicates per major iteration $R$, maximum number of scenario constraints added per iteration $N_c$, number of samples to estimate risk levels $N_{MC}$, and initial guess $\hat{x} \in X$.

\State \textbf{Output}: Pairs $(\bar{\nu}^{i,r},\bar{\alpha}^{i,r})$, $i \in \{1,\cdots,M\}$ and $r \in \{1,\cdots,R\}$, of objective values and risk levels that can be used to approximate the efficient frontier, and corresponding sequence of solutions $\{\bar{x}^{i,r}\}$.

\State \textbf{Preparation}: Draw a (fixed) sample $\{\bar{\xi}^{l}\}_{l=1}^{N_{MC}}$ from $\P$ for estimating risk levels of candidate solutions.

\For{\textup{major iteration} $i = 1$ to $M$}
\For{\textup{replicate} $r = 1$ to $R$}

\State Sample: Draw i.i.d. sample $\{\xi^{l,r}\}_{l=1}^{N_i}$ from $\P$.

\State Initialize: Starting point $x^{0} := \hat{x}$, list of scenarios enforced for each chance constraint $\I^0_k = \emptyset$, $k \in \{1,\cdots,m\}$,

\Statex \hspace*{0.5in} and counter $q = 0$.

\Repeat

\State Update $q \leftarrow q+1$, number of scenario constraints violated $V \leftarrow 0$, $\I^{q}_k = \I^{q-1}_k$, $\forall k$.

\State Solve the following scenario approximation problem (locally) using the initial guess $x^{q-1}$ to obtain

\Statex \hspace*{0.60in} a solution $x^q$ and corresponding optimal objective $\nu^q$:
\begin{align*}
\uset{x \in X}{\min} \:\: & f(x) \\
\text{s.t.} \:\: & g_k(x,\xi^{l,r}) \leq 0, \quad \forall l \in \I^{q-1}_k.
\end{align*}

\For{$k = 1$ to $m$}

\State Evaluate $k^{\text{th}}$ random constraint $g_k$ at $x^q$ for each of the scenarios $l = 1,\cdots,N_i$.

\State Sort the values $\{g_k(x^q,\xi^{l,r})\}_{l=1}^{N_i}$ in decreasing order.

\State Increment $V$ by the number of scenario constraints satisfying $g_k(x^q,\xi^{l,r}) > 0$.
%, $l = 1,\cdots,N_i$. 

%\Statex \hspace*{1in} $l = 1,\cdots,N_i$.

\State Add at most $N_c$ of the most violated scenario indices $l$ to $\I^q_k$.

\EndFor

\Until{$V = 0$}

\State Set $\hat{x} = x^{i,r} = x^q$ and $\bar{\nu}^{i,r} = \nu^q$ to their converged values.

\State Estimate risk level $\bar{\alpha}^{i,r}$ of candidate solution $\hat{x}$ using the sample $\{\bar{\xi}^{l}\}_{l=1}^{N_{MC}}$.
%(see Section~\ref{subsec:algoparams}).

\EndFor
\EndFor
\end{algorithmic}
\end{algorithm}

\subsection{Solving~\eqref{eqn:ccp} for a fixed risk level}
\label{subsec:algofixedrisk}

Algorithm~\ref{alg:propalgoccp} adapts Algorithm~\ref{alg:propalgoef} to solve~\eqref{eqn:ccp} for a given risk level $\hat{\alpha} \in (0,1)$.
An initial point $\bar{x}^0$ and an upper bound on the optimal objective value $\nu_{up}$ can be obtained in a manner similar to Algorithm~\ref{alg:propalgoef}, whereas an initial lower bound on the optimal objective value $\nu_{low}$ can be obtained either using lower bounding techniques (see~\citep{nemirovski2006convex,luedtke2008sample}), or by trial and error.
Note that the sequence of approximations in line 9 of Algorithm~\ref{alg:propalgoccp} need not be solved until termination if Algorithm~\ref{alg:propalgo} determines that $\bar{\alpha}^i < \hat{\alpha}$ before its termination criteria have been satisfied.

\begin{algorithm}
\caption{Solving~\eqref{eqn:ccp} for a fixed risk level}
\label{alg:propalgoccp}
\begin{algorithmic}[1]

\State \textbf{Input}: target risk level $\hat{\alpha} \in (0,1)$, `guaranteed' lower and upper bounds on the optimal objective value $\nu_{low}$ and $\nu_{up}$ with $\nu_{low} \leq \nu_{up}$, and initial point $\bar{x}^0 \in X$.

\State \textbf{Set algorithmic parameters}: in addition to line 2 of Algorithm~\ref{alg:propalgoef}, let $\nu_{tol} > 0$ denote an optimality tolerance.

\State \textbf{Output}: approximate optimal objective value $\hat{\nu}$ of~\eqref{eqn:ccp} for a risk level of $\hat{\alpha}$.

\State \textbf{Preprocessing}: let $\bar{\nu}^0 = \frac{1}{2}(\nu_{low} + \nu_{up})$, and determine smoothing parameters $\{\tau_{k,j}\}_{k=1}^{K}$ scaled at $\proj{\bar{x}^0}{X_{\bar{\nu}^0}}$ using Algorithm~\ref{alg:scaling} and an initial sequence of step lengths $\{\bar{\gamma}_k\}_{k=1}^{K}$ for the corresponding sequence of approximating problems~\eqref{eqn:app} with $\nu = \bar{\nu}^0$ using Algorithm~\ref{alg:steplength}.

\State \textbf{Optimization Phase}:

\State Initialize index $i = 0$.

\Repeat

\State Update iteration count $i \leftarrow i+1$ and set objective bound $\bar{\nu}^i = \frac{1}{2}(\nu_{low} + \nu_{up})$.

\State Obtain $(\bar{\alpha}^i,\bar{x}^i)$ by solving sequence of approximations~\eqref{eqn:app} using Algorithm~\ref{alg:propalgo} with the above algorithmic

\Statex \hspace*{0.2in} parameter settings and $\proj{\bar{x}^{i-1}}{X_{\bar{\nu}^i}}$ as the initial guess.

%\If{$\bar{\alpha}^i \geq \hat{\alpha}$}
%\State Set $\nu_{low} \leftarrow \bar{\nu}^i$.
%\Else
%\State Set $\nu_{up} \leftarrow \bar{\nu}^i$.
%\EndIf

\IfThenElse {$\bar{\alpha}^i \geq \hat{\alpha}$}% If ...
      {set $\nu_{low} \leftarrow \bar{\nu}^i$}% ...then...
      {set $\nu_{up} \leftarrow \bar{\nu}^i$}% ...else...

\Until{$\nu_{low} \geq \nu_{up} - \nu_{tol}$}

\State Set $\hat{\nu} = \nu_{up}$.

\end{algorithmic}
\end{algorithm}

\subsection{Implementation of the sigmoidal approximation algorithm}
\label{subsec:sigapproximpl}

We used the following `tuned' settings to solve each iteration of the sigmoidal approximation problem (see Section~4 of~\citet{cao2018sigmoidal}) using IPOPT: \texttt{tol} $= 10^{-4}$, \texttt{max\_iter = 10000}, \texttt{hessian\_approximation = limited\_memory}, \texttt{jac\_c\_constant=yes}, and \texttt{max\_cpu\_time=3600} seconds. 
We terminated the loop of Algorithm~\texttt{SigVar-Alg} of~\citet{cao2018sigmoidal} when the objective improved by less than $0.01\%$ relative to the previous iteration. 
In what follows, we use the notation of Algorithm~\texttt{SigVar-Alg} of~\citet{cao2018sigmoidal}.
For our `adaptive $\gamma$' setting, we use the proposal of~\citet{cao2018sigmoidal} to specify $\gamma$ when the solution of the CVaR problem corresponds to $t_c(\alpha) < 0$. When $t_c(\alpha) = 0$ at the CVaR solution returned by Gurobi, we try to estimate a good value of $t_c(\alpha)$ by looking at the true distribution of $g(x_c(\alpha),\xi)$.

\subsection{Tailored projection step for Case study~\ref{case:portfolio_var}}
\label{subsec:tailproj_portvar}

Because Algorithm~\ref{alg:propalgoef} requires projecting onto the sets $X_{\nu} := \Set{x \in \Delta_N}{\tr{x} \Sigma x \leq \nu}$ many times for Case study~\ref{case:portfolio_var} with different values of $\nu$, we develop a tailored projection routine that is computationally more efficient than solving quadratically-constrained quadratic programs to compute these projections.
To project a point $y \in \R^N$ onto $X_{\nu}$ for some $\nu > 0$, note that the KKT conditions for the projection problem $\uset{x \in X_{\nu}}{\min} 0.5\norm{x-y}^2$ yield:
\begin{align*}
x^*_i &= \frac{y_i-\mu+\pi_i}{1+2\sigma^2_i\lambda}, \quad x^*_i \geq 0, \quad \pi_i \geq 0, \quad x^*_i \pi_i = 0, \quad \forall i \in \{1,\cdots,N\}, \\
\lambda &\geq 0, \quad \sum_{i=1}^{N} x^*_i = 1, \quad \sum_{i=1}^{N} \alpha^2_i (x^*_i)^2 \leq \nu, \quad \lambda \left( \sum_{i=1}^{N} \alpha^2_i (x^*_i)^2 - \nu \right) = 0,
\end{align*}
where $x^* \in \R^N$ denotes the projection and $\mu$, $\lambda$, and $(\pi_1,\cdots,\pi_N)$ denote KKT multipliers. Note that the inequalities in the first row can be simplified as $x^*_i = \max\left\{0,\frac{y_i-\mu}{1+2\sigma^2_i\lambda}\right\}$, $\forall i \in \{1,\cdots,N\}$.
We first check whether $\lambda = 0$ satisfies the above system of equations (i.e., if the quadratic constraint is inactive at the solution) by setting $x^*$ to be the projection of $y$ onto the unit simplex.
If $\lambda = 0$ is infeasible, we solve the following system of equations by using binary search over $\lambda$:
\begin{align*}
x^*_i &= \max\left\{0,\frac{y_i-\mu}{1+2\sigma^2_i\lambda}\right\}, \:\: \forall i \in \{1,\cdots,N\}, \quad \sum_{i=1}^{N} x_i = 1, \quad \sum_{i=1}^{N} \alpha^2_i (x^*_i)^2 = \nu,
\end{align*}
where for each fixed $\lambda > 0$, the first two sets of equations in $x^*$ and $\mu$ are solved by adapting the algorithm of~\citet{condat2016fast} for projection onto the unit simplex.

\section{Proof of Proposition~\ref{prop:clarkegradlimit}}
\label{subsec:propclarkegradlimit}

We first sketch an outline of the proof (note that our proof will follow a different order for reasons that will become evident). We will establish that (see Chapters~4 and~5 of~\citet{rockafellar2009variational} for definitions of technical terms)
\begin{align}
\label{eqn:seplimsup}
\uset{k \to \infty}{\limsup_{x \to \bar{x}}}\: \partial \ph_k(x) + N_{\xv}(x) &= \uset{k \to \infty}{\limsup_{x \to \bar{x}}}\: \partial \ph_k(x) + \uset{k \to \infty}{\limsup_{x \to \bar{x}}}\: N_{\xv}(x).
\end{align}
Then, because of the outer semicontinuity of the normal cone mapping, it suffices to prove that the outer limit of $\partial \ph_k(x)$ is a subset of $\{\nabla p(\bar{x})\}$.
To demonstrate this, we will first show that $\partial \ph_k(x) = - \partial \int_{-\infty}^{\infty} F(x,\eta) d\phi\left(\eta;\tau_c^{k-1}\right) d\eta$, where $F$ is the cumulative distribution function of $\max\left[\bar{g}(x,\xi)\right]$. Then, we will split this integral into two parts - one accounting for tail contributions (which we will show vanishes when we take the outer limit), and the other accounting for the contributions of $F$ near $(x,\eta) = (\bar{x},0)$ that we will show satisfies the desired outer semicontinuity property.

Since $\ph_k(x)$ can be rewritten as $\ph_k(x) = \bexpect{\phi\left(\max\left[\bar{g}(x,\xi)\right];\tau_c^{k-1}\right)}$
by Assumptions~\ref{ass:phibasic}B and~\ref{ass:phiandcdfass}, we have
\begin{align*}
\ph_k(x) = \Bexpect{\phi\left(\max\left[\bar{g}(x,\xi)\right];\tau_c^{k-1}\right)} &= \int_{-\infty}^{+\infty} \phi\left(\max\left[\bar{g}(x,\xi)\right];\tau_c^{k-1}\right) d\Pb \\
&= \int_{-\infty}^{+\infty} \phi\left(\eta;\tau_c^{k-1}\right) dF(x,\eta) \\
&= \lim_{\eta \to +\infty} \phi\left(\eta;\tau_c^{k-1}\right) - \int_{-\infty}^{+\infty} F(x,\eta) d\phi\left(\eta;\tau_c^{k-1}\right) d\eta,
\end{align*} 
where the second line is to be interpreted as a Lebesgue-Stieljes integral, and the final step follows by integrating by parts and Assumption~\ref{ass:phibasic}. This yields (see Proposition~\ref{prop:clarkegradapprox} and Assumption~\ref{ass:phiandcdfass} for the existence of these quantities)
\begin{align}
\label{eqn:clarkegradapprox}
\partial \ph_k(x) &\subset - \partial \int_{\abs{\eta} \geq \varepsilon_k} F(x,\eta) d\phi\left(\eta;\tau_c^{k-1}\right) d\eta - \partial \int_{-\varepsilon_k}^{\varepsilon_k} F(x,\eta) d\phi\left(\eta;\tau_c^{k-1}\right) d\eta
\end{align}
by the properties of the Clarke generalized gradient, where $\{\varepsilon_k\}$ is defined in Assumption~\ref{ass:phiandcdfass}.

Let $\{x_k\}$ be any sequence in $\xv$ converging to $\bar{x} \in \xv$. Suppose $v_k \in - \partial \int_{\abs{\eta} \geq \varepsilon_k} F(x_k,\eta) d\phi\left(\eta;\tau_c^{k-1}\right) d\eta$ with $v_k \to v \in \R^n$. We would like to show that $v = 0$.
Note that by an abuse of notation (where we actually take norms of the integrable selections that define $v_k$)
\begin{align}
\label{eqn:tailcontrib}
\lim_{k \to \infty} \norm{v_k} = \lim_{k \to \infty} \norm*{- \partial \int_{\abs{\eta} \geq \varepsilon_k} F(x_k,\eta) d\phi\left(\eta;\tau_c^{k-1}\right) d\eta} &= \lim_{k \to \infty} \norm*{\int_{\abs{\eta} \geq \varepsilon_k} \partial_x F(x_k,\eta) d\phi\left(\eta;\tau_c^{k-1}\right) d\eta} \nonumber\\
&\leq \lim_{k \to \infty} \int_{\abs{\eta} \geq \varepsilon_k} \norm*{\partial_x F(x_k,\eta)} d\phi\left(\eta;\tau_c^{k-1}\right) d\eta \nonumber\\
&\leq \lim_{k \to \infty} \int_{\abs{\eta} \geq \varepsilon_k} L_F(\eta) d\phi\left(\eta;\tau_c^{k-1}\right) d\eta = 0,
\end{align}
where the first step follows from Theorem~2.7.2 of~\citet{clarke1990optimization} (whose assumptions are satisfied by
virtue of Assumption~\ref{ass:phiandcdfass}B), the second inequality follows from Proposition~2.1.2
of~\citet{clarke1990optimization}, and the final equality follows from Assumption~\ref{ass:phiandcdfass}C.

The above arguments establish that $\limsup\limits_{\substack{x \to \bar{x} \\ k \to \infty}}\: - \partial \int_{\abs{\eta} \geq \varepsilon_k} F(x,\eta) d\phi\left(\eta;\tau_c^{k-1}\right) d\eta = \{0\}$. Consequently, we have $\limsup\limits_{\substack{x \to \bar{x} \\ k \to \infty}}\: \partial \ph_k(x) \subset \limsup\limits_{\substack{x \to \bar{x} \\ k \to \infty}}\: - \partial \int_{-\varepsilon_k}^{\varepsilon_k} F(x,\eta) d\phi\left(\eta;\tau_c^{k-1}\right) d\eta$ from Equation~\eqref{eqn:clarkegradapprox}. We now consider the outer limit $\limsup\limits_{\substack{x \to \bar{x} \\ k \to \infty}}\: - \partial \int_{-\varepsilon_k}^{\varepsilon_k} F(x,\eta) d\phi\left(\eta;\tau_c^{k-1}\right) d\eta$.
Suppose $w_k \in - \partial \int_{-\varepsilon_k}^{\varepsilon_k} F(x_k,\eta) d\phi\left(\eta;\tau_c^{k-1}\right) d\eta$ with $w_k \to w \in \R^n$. We wish to show that $w = \nabla p(\bar{x})$.
Invoking Theorem~2.7.2 of~\citet{clarke1990optimization} once again, we have that
\begin{align}
\label{eqn:maincontrib}
\uset{k \to \infty}{\limsup}\: - \partial \int_{-\varepsilon_k}^{\varepsilon_k} F(x_k,\eta) d\phi\left(\eta;\tau_c^{k-1}\right) d\eta &\subset \uset{k \to \infty}{\limsup}\: - \int_{-\varepsilon_k}^{\varepsilon_k} \nabla_x F(x_k,\eta) d\phi\left(\eta;\tau_c^{k-1}\right) d\eta  \nonumber\\
&= \uset{k \to \infty}{\limsup}\: - \left(\int_{-\varepsilon_k}^{\varepsilon_k} \nabla_x F(x_k,\eta) d\hat{\phi}\left(\eta\right) d\eta\right) \left(\int_{-\varepsilon_k}^{\varepsilon_k} d\phi\left(z;\tau_c^{k-1}\right) dz\right) \nonumber\\
&= \left(\uset{k \to \infty}{\limsup}\: - \int_{-\varepsilon_k}^{\varepsilon_k} \nabla_x F(x_k,\eta) d\hat{\phi}\left(\eta\right) d\eta\right) \:\: \left(\lim_{k \to \infty} \int_{-\varepsilon_k}^{\varepsilon_k} d\phi\left(z;\tau_c^{k-1}\right) dz\right) \nonumber\\
&= \uset{k \to \infty}{\limsup}\: - \int_{-\varepsilon_k}^{\varepsilon_k} \nabla_x F(x_k,\eta) d\hat{\phi}_k\left(\eta\right) d\eta, 
\end{align}
where for each $k \in \N$ large enough, $\hat{\phi}_k:\R \to \R$ is defined as $\hat{\phi}_k\left(y\right) =
\dfrac{\phi\left(y;\tau_c^{k-1}\right)}{\int_{-\varepsilon_k}^{\varepsilon_k} d\phi\left(z;\tau_c^{k-1}\right) dz}$, the
first step follows (by an abuse of notation) from the fact that $\varepsilon_k < \theta$ for $k$ large enough (see
Assumption~\ref{ass:phiandcdfass}B), and the third and fourth steps follow from Assumption~\ref{ass:phiandcdfass}C. 
Noting that\footnote{Lemma~4.2 of~\citet{shan2014smoothing} reaches a stronger conclusion from Equation~\eqref{eqn:maincontrib} than our result (they only require that the distribution function $F$ is locally Lipschitz continuous). We do not adapt their arguments since we are unable to follow a key step (the first step in~\citep[pg.~191]{shan2014smoothing}) in their proof.}
\[
\int_{-\varepsilon_k}^{\varepsilon_k} \nabla_x F(x_k,\eta) d\hat{\phi}_k(\eta) d\eta = 
\tr{\begin{bmatrix}
\dfrac{\partial F}{\partial x_1}(x_k,\omega_{1,k}) \cdots \dfrac{\partial F}{\partial x_n}(x_k,\omega_{n,k})
\end{bmatrix}}
\int_{-\varepsilon_k}^{\varepsilon_k} d\hat{\phi}_k(\eta) d\eta = 
\begin{bmatrix}
\dfrac{\partial F}{\partial x_1}(x_k,\omega_{1,k}) \\
\vdots \\
\dfrac{\partial F}{\partial x_n}(x_k,\omega_{n,k})
\end{bmatrix}
\]
for some constants $\omega_{i,k} \in (-\varepsilon_k,\varepsilon_k)$, $i = 1,\cdots,n$, by virtue of
Assumption~\ref{ass:phiandcdfass}B and the first mean value theorem for definite integrals, we have 
\begin{align*}
\uset{k \to \infty}{\limsup}\: - \partial \int_{-\varepsilon_k}^{\varepsilon_k} F(x_k,\eta) d\phi\left(\eta;\tau_c^{k-1}\right) d\eta &\subset \uset{k \to \infty}{\limsup}\: \left\lbrace - \tr{\begin{bmatrix}
\dfrac{\partial F}{\partial x_1}(x_k,\omega_{1,k}) \cdots \dfrac{\partial F}{\partial x_n}(x_k,\omega_{n,k})
\end{bmatrix}}\right\rbrace \\
&= \{-\nabla_x F(\bar{x},0)\} = \{\nabla p(\bar{x})\},
\end{align*}
where the first equality above follows from Assumption~\ref{ass:phiandcdfass}B, and the second equality follows from the fact that $p(x) = 1 - F(x,0)$ for each $x \in \xv$. Reconciling our progress with Equation~\eqref{eqn:clarkegradapprox}, we obtain
\[
\uset{k \to \infty}{\limsup_{x \to \bar{x}}}\: \partial \ph_k(x) \subset \{\nabla p(\bar{x})\}.
\]
To establish the desirable equality in Equation~\eqref{eqn:seplimsup}, it suffices to show that $\partial \ph_k(x_k) \subset C$ for $k$ large enough, where $C \subset \R^n$ is independent of $k$.
From Equation~\eqref{eqn:tailcontrib}, we have that the first term in the right-hand side of Equation~\eqref{eqn:clarkegradapprox} is contained in a bounded set that is independent of $k \in \N$.
From Equation~\eqref{eqn:maincontrib}, we have that any element of the second term in the right-hand side of Equation~\eqref{eqn:clarkegradapprox} is bounded above in norm by $\max\limits_{\substack{x \in \cl{B_{\delta}(\bar{x})} \\ \eta \in [-0.5\theta,0.5\theta]}} \norm*{\nabla_x F(x,\eta)}$ for $k$ large enough for any $\delta > 0$. The above arguments in conjunction with Equation~\eqref{eqn:clarkegradapprox} establish Equation~\eqref{eqn:seplimsup}. The desired result then follows from the outer semicontinuity of the normal cone mapping, see Proposition~6.6 of~\citet{rockafellar2009variational}.

%\subsection{Proof of Proposition~\ref{prop:weakconvparamrec}}
%\label{subsec:propweakconvparamrec}

\section{Recourse formulation}
\label{app:recourse}

As noted in Section~\ref{sec:introduction} of the paper, Problem~\eqref{eqn:ccp} can also be used to model chance-constrained programs with static recourse decisions, e.g., by defining $g$ through the solution of the following auxiliary optimization problem for each $(x,\xi) \in X \times \Xi$:
\begin{alignat*}{2}
g(x,\xi) := \: &\uset{y, \eta}{\min} \:\: && \eta \\
&\:\: \text{s.t.} && T(x,\xi) + W(\xi) y + \eta e \geq 0, \\
& && y \in \R^{n_y}_+, \:\: \eta \geq -1,
\end{alignat*}
where $T: \R^n \times \R^d \to \R^{m_r}$ is continuously differentiable, and $W:\R^d \to \R^{m_r \times n_y}$.
Note that $g$ can be recast in the form of a joint chance constraint by appealing to linear programming duality, viz.,
\[
g(x,\xi) = \uset{(v,w) \in \mathrm{EXT}(V(\xi))}{\max} -\tr{v} T(x,\xi) - w,
\]
where $\mathrm{EXT}(V(\xi))$ denotes the set of extreme points of the polytope 
\[
V(\xi) = \Set{(v,w) \in \R^{m_r}_+ \times \R_+}{\tr{v}W(\xi) \leq 0, \: \tr{v} e + w = 1}.
\]
Explicitly reformulating the recourse constraints into the above explicit form is not generally practical 
%(cf.~Algorithms~\ref{alg:scaling} and~\ref{alg:steplength}) 
since Algorithms~\ref{alg:scaling} and~\ref{alg:steplength} rely on stepping through each of the constraint functions,
and the cardinality of the set of extreme points of $V(\xi)$ may be huge.
Therefore, we consider the case when the approximations $\ph_k$ are constructed using a single smoothing function, i.e.,
$\ph_k(x) = \bexpect{\phi_k(g(x,\xi))}$, where $\{\phi_k\}$ is a sequence of smooth scalar functions that approximate the step function.

Throughout this section, we assume that $g$ that can be reformulated as $g(x,\xi) := \uset{j \in J}{\max} \:\: h_j(x,\xi)$, where $J$ is a (potentially large) finite index set, and $h_j:\R^n \times R^d \to \R$ are continuously differentiable functions that are Lipschitz continuous and have Lipschitz continuous gradients in the sense of Assumption~\ref{ass:confunc}.
We let $L^{'}_h(\xi)$ denote the Lipschitz constant of $\nabla h(\cdot,\xi)$ on $\xv$ with $\bexpect{\left(L^{'}_h(\xi)\right)^2} < +\infty$.
While the theory in this section is also applicable to the ordinary joint chance constrained setting, it is usually advantageous to consider individual smoothing functions for each chance constraint function in that case whenever feasible.
%As an aside, we mention that imposing the lower bound constraint $\eta \geq -1$ in the definition of $g$ is not necessary (and it may in fact be advantageous to omit it) when the set 
%\[
%\Set{z \in \R}{z = \uset{i = 1,\cdots,m_r}{\min} T_i(x,\xi) + \tr{\left(w_i(\xi)\right)} y \:\: \text{for some } (x,y,\xi) \in X \times \R^{n_y}_+ \times \Xi}
%\]
%is bounded from above, where $w_i(\xi)$ denotes the vector corresponding to the $i^{\textup{th}}$ row of $W(\xi)$.

We first characterize the Clarke generalized gradient of the approximation $\ph_k$ under this setup.

\begin{proposition}
\label{prop:subdiffrec}
Suppose Assumptions~\ref{ass:phibasic} and~\ref{ass:phiadvanced} and the above assumptions on $g(x,\xi) = \max\left[h(x,\xi)\right]$ and $\phi_k$ hold.
Then $\partial \ph_k(x) = \bexpect{d\phi_k\left(\max\left[h(x,\xi)\right]\right) \times \partial_x \max\left[h(x,\xi)\right]}$.
\end{proposition}
\begin{proof}
Note that for any $x \in X$ and $\xi \in \Xi$, we have $\partial_x \phi_k(g(x,\xi)) = d\phi_k(g(x,\xi)) \: \partial_x \max\left[h(x,\xi)\right]$, see Theorem~2.3.9 of~\citet{clarke1990optimization}.
A justification for swapping the derivative and expectation operators then follows from
Proposition~\ref{prop:clarkegradapprox} of the paper and the fact that $\ph_k(x) = \bexpect{\phi_k(g(x,\xi))} = \bexpect{\uset{j}{\max}\left[\phi_k\left(g_j(x,\xi)\right)\right]}$ since $\phi_k$ is monotonically nondecreasing on $\R$.
 \halmos
\end{proof}

The following result establishes that the approximation $\ph_k$ continues to enjoy the weak convexity property for the above setup under mild assumptions.

\begin{proposition}
\label{prop:weakconvparamrec}
Suppose Assumptions~\ref{ass:xcompact},~\ref{ass:phibasic}, and~\ref{ass:phiadvanced} hold.
Additionally, suppose $g(x,\xi) = \max\left[h(x,\xi)\right]$ satisfies the above conditions and $\phi_k$ is a scalar smoothing function.
Then $\ph_k(\cdot)$ is $\bar{L}_k$-weakly convex for some constant $\bar{L}_k$ that depends on $L^{'}_h(\xi)$, $L^{'}_{\phi,k}$, $M^{'}_{\phi,k}$, and the diameter of $\xv$.
\end{proposition}
\begin{proof}
First, note that $g(\cdot,\xi) := \max\left[h(\cdot,\xi)\right]$ is $L^{'}_h(\xi)$-weakly convex on $\xv$ for each $\xi \in \Xi$, see Lemma~4.2 of~\citet{drusvyatskiy2016efficiency}. 
In particular, this implies that for any $y, z \in \xv$,
\begin{align*}
g(z,\xi) = \max\left[h(z,\xi)\right] &\geq g(y,\xi) + \tr{s}_y(\xi) (z - y) - \dfrac{L^{'}_h(\xi)}{2} \norm{z - y}^2
\end{align*}
for any $s_y(\xi) \in \partial \max\left[h(y,\xi)\right]$. The Lipschitz continuity of $d\phi_k(\cdot)$ implies that for any scalars $v$ and $w$:
\[
\phi_k(w) \geq \phi_k(v) + d\phi_k(v) (w - v) - \dfrac{L^{'}_{\phi,k}}{2} (w - v)^2.
\]
From the monotonicity Assumption~\ref{ass:phibasic}B and by recursively applying the above result, we have for any $\xi \in \Xi$ and $y, z \in \xv$:
{
\small
\begin{align*}
\phi_k\left(g(z,\xi)\right) &\geq \phi_k\left(g(y,\xi) + \tr{s}_y(\xi) (z - y) - \dfrac{L^{'}_h(\xi)}{2} \norm{z - y}^2 \right) \\
&\geq \phi_k\left( g(y,\xi) + \tr{s}_y(\xi) (z - y) \right) -\dfrac{L^{'}_h(\xi)}{2} d\phi_k\left( g(y,\xi) + \tr{s}_y (z - y) \right) \norm{z - y}^2 - \dfrac{L^{'}_{\phi,k}\left(L^{'}_h(\xi)\right)^2}{8} \norm{z - y}^4  \\
&\geq \phi_k\left( g(y,\xi) \right) + \tr{\left(d\phi_k\left( g(y,\xi) \right) s_y(\xi) \right)} (z - y) - \dfrac{L^{'}_{\phi,k}}{2}\left(\tr{s}_y(\xi) (z - y) \right)^2 - \dfrac{L^{'}_h(\xi) M^{'}_{\phi,k}}{2} \norm{z - y}^2 - \\
&\quad\quad\quad\quad \dfrac{L^{'}_{\phi,k}\left(L^{'}_h(\xi)\right)^2}{8} \norm{z - y}^4 \\
&\geq \phi_k\left( g(y,\xi) \right) + \tr{\left(d\phi_k\left( g(y,\xi) \right) s_y(\xi) \right)} (z - y) - \dfrac{\norm{z - y}^2}{2} \Bigg[ L^{'}_{\phi,k}\norm{s_y(\xi)}^2 + L^{'}_h(\xi) M^{'}_{\phi,k} + \\
&\hspace*{3.5in} \dfrac{L^{'}_{\phi,k} \left(L^{'}_h(\xi)\right)^2 \diam{\xv}^2}{4} \Bigg].
\end{align*}
Taking expectation on both sides and noting that $\norm{s_y(\xi)} \leq L^{'}_h(\xi)$, we get
\begin{align*}
\bexpect{\phi_k\left(g(z,\xi)\right)} &\geq \bexpect{\phi_k\left( g(y,\xi) \right)} + \bexpect{\tr{\left(d\phi_k\left( g(y,\xi) \right) s_y \right)} (z - y)} - \\
&\quad\quad\quad\quad \dfrac{\norm{z - y}^2}{2} \left[ L^{'}_{\phi,k} \bexpect{\left( L^{'}_h(\xi) \right)^2} +
\bexpect{L^{'}_h(\xi)} M^{'}_{\phi,k} + \dfrac{L^{'}_{\phi,k} \bexpect{\left(L^{'}_h(\xi)\right)^2} \diam{\xv}^2}{4} \right].
\end{align*}
}
Therefore, $\ph_k$ is $\bar{L}_k$-weakly convex on $\xv$ with
\[
\bar{L}_k := L^{'}_{\phi,k} \bexpect{\left( L^{'}_h(\xi) \right)^2} + \bexpect{L^{'}_h(\xi)} M^{'}_{\phi,k} +
\dfrac{1}{4} L^{'}_{\phi,k} \bexpect{\left(L^{'}_h(\xi)\right)^2} \diam{\xv}^2. \hspace*{1.95in} \halmos
\]
\end{proof}

We now outline our proposal for estimating the weak convexity parameter $\bar{L}_k$ of $\ph_k$ for the above setting. First, we note that estimate of the constants $L^{'}_{\phi,k}$ and $M^{'}_{\phi,k}$ can be obtained from Proposition~\ref{prop:oursmoothapproxconst}. Next, we propose to replace the unknown constant $\diam{\xv}$ with the diameter $2r$ of the sampling ball in Algorithm~\ref{alg:steplength}. Finally, we note that estimating the Lipschitz constant $L^{'}_h(\xi)$ for any $\xi \in \Xi$ is tricky since it would involve looking at all of the $\abs{J}$ constraints in general, just the thing we wanted to avoid! To circumvent this, we propose to replace $L^{'}_h(\xi)$ in $\bar{L}_k$ with an estimate of the local Lipschitz constant of our choice of the $B$-subdifferential element of $g(\cdot,\xi)$ through sampling (similar to the proposal in Algorithm~\ref{alg:steplength}). When $g$ is defined through the solution of the recourse formulation considered, we can estimate a $B$-subdifferential element by appealing to linear programming duality and use this as a proxy for the `gradient' of $g$. Note that a crude estimate of the step length does not affect the theoretical convergence guarantee of the stochastic subgradient method of~\citet{davis2018stochastic}.

\section{Additional computational results}
\label{app:computexp}

We present results of our replication experiments for the four case studies in the paper, and also consider a variant of Case study~\ref{case:normopt} that has a known analytical solution to benchmark our proposed approach.

Figure~\ref{fig:portfolio_trajec} presents the enclosure of the trajectories of the EF generated by ten replicates of
the proposed approach when applied to Case study~\ref{case:portfolio}. For each value of the objective bound $\nu$, we
plot the smallest and largest risk level determined by Algorithm~\ref{alg:propalgo} at that bound over the different
replicates. We find that the risk levels returned by the proposed algorithm do not vary significantly across the different replicates, with the maximum difference in the risk levels across the $26$ points on the EF being a factor of $1.48$.
Figure~\ref{fig:portfolio_var_trajec} presents the corresponding plot for Case study~\ref{case:portfolio_var}. We again find that the risk levels returned by the proposed algorithm do not vary significantly across the different replicates, with the maximum difference in the risk levels across the $20$ points on the EF being a factor of $1.006$.
Figure~\ref{fig:normopt_trajec} presents the corresponding plot for Case study~\ref{case:normopt}. 
The maximum difference in the risk levels at the $31$ points on the EF for this case is a factor of $1.22$ across the ten replicates.
%We again find that the risk levels returned by the proposed algorithm do not vary significantly across the different replicates, with the maximum difference in the risk levels across the $31$ points on the EF being a factor of $1.21$.
Figure~\ref{fig:resource_trajec} presents the corresponding plot for Case study~\ref{case:resource}. 
The maximum difference in the risk levels at the $26$ points on the EF for this case is a factor of $1.25$ across the ten replicates.

\begin{figure}[ht!]
\begin{center}
\caption{Enclosure of the trajectories of the efficient frontier for Case study~\ref{case:portfolio} generated by ten replicates of the proposed approach.} \label{fig:portfolio_trajec}
\includegraphics[width=0.8\linewidth]{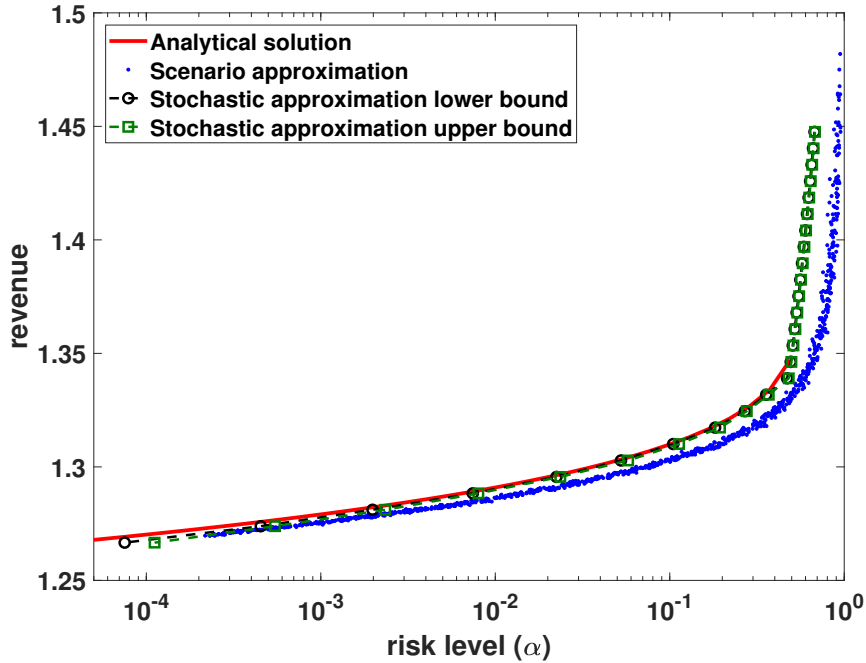}
\end{center}
\end{figure}

\begin{figure}[ht!]
\begin{center}
\caption{Enclosure of the trajectories of the efficient frontier for Case study~\ref{case:portfolio_var} generated by ten replicates of the proposed approach.} \label{fig:portfolio_var_trajec}
\includegraphics[width=0.8\linewidth]{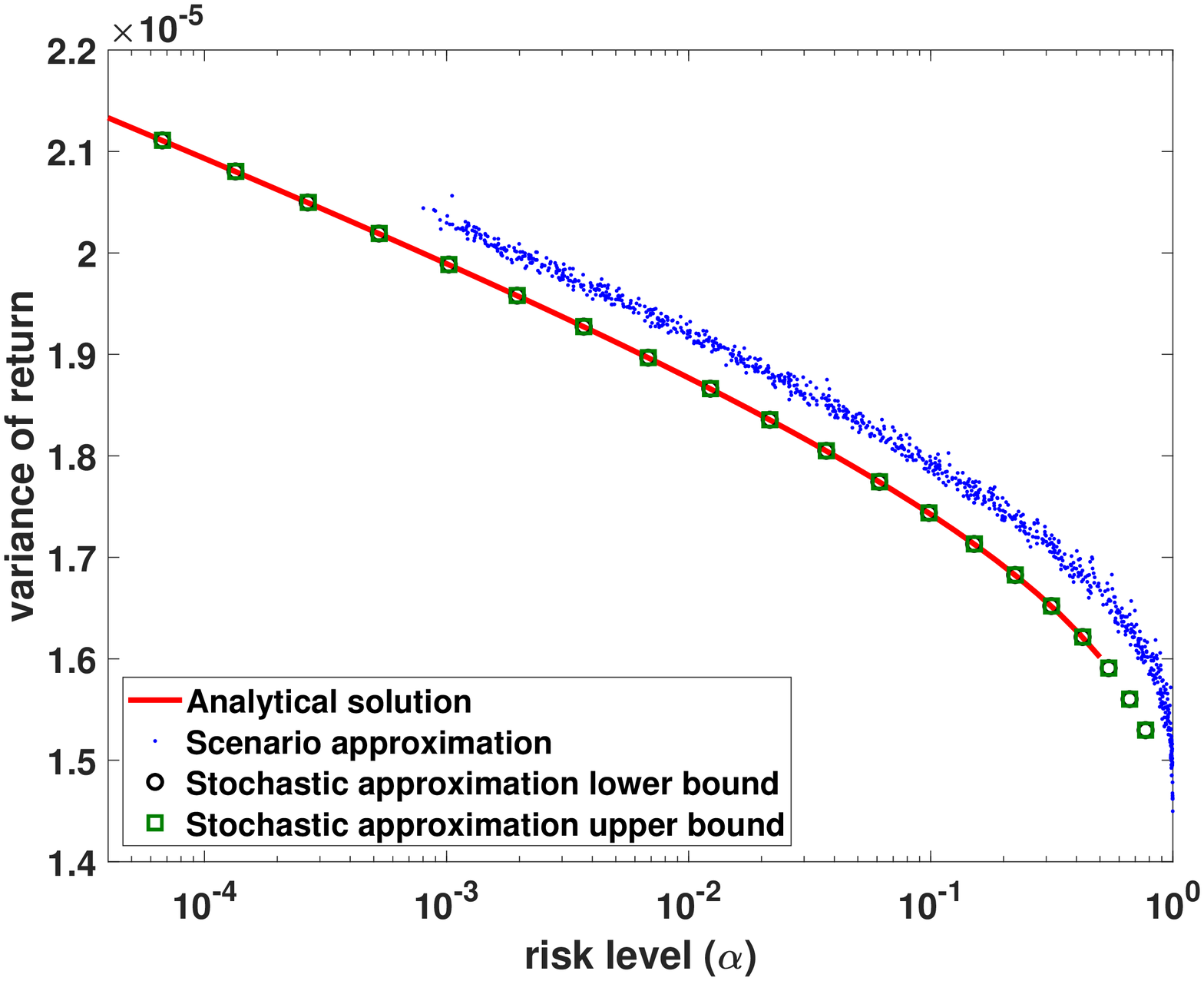}
\end{center}
\end{figure}

\begin{figure}[ht!]
\begin{center}
\caption{Enclosure of the trajectories of the efficient frontier for Case study~\ref{case:normopt} generated by ten replicates of the proposed approach.} \label{fig:normopt_trajec}
\includegraphics[width=0.8\linewidth]{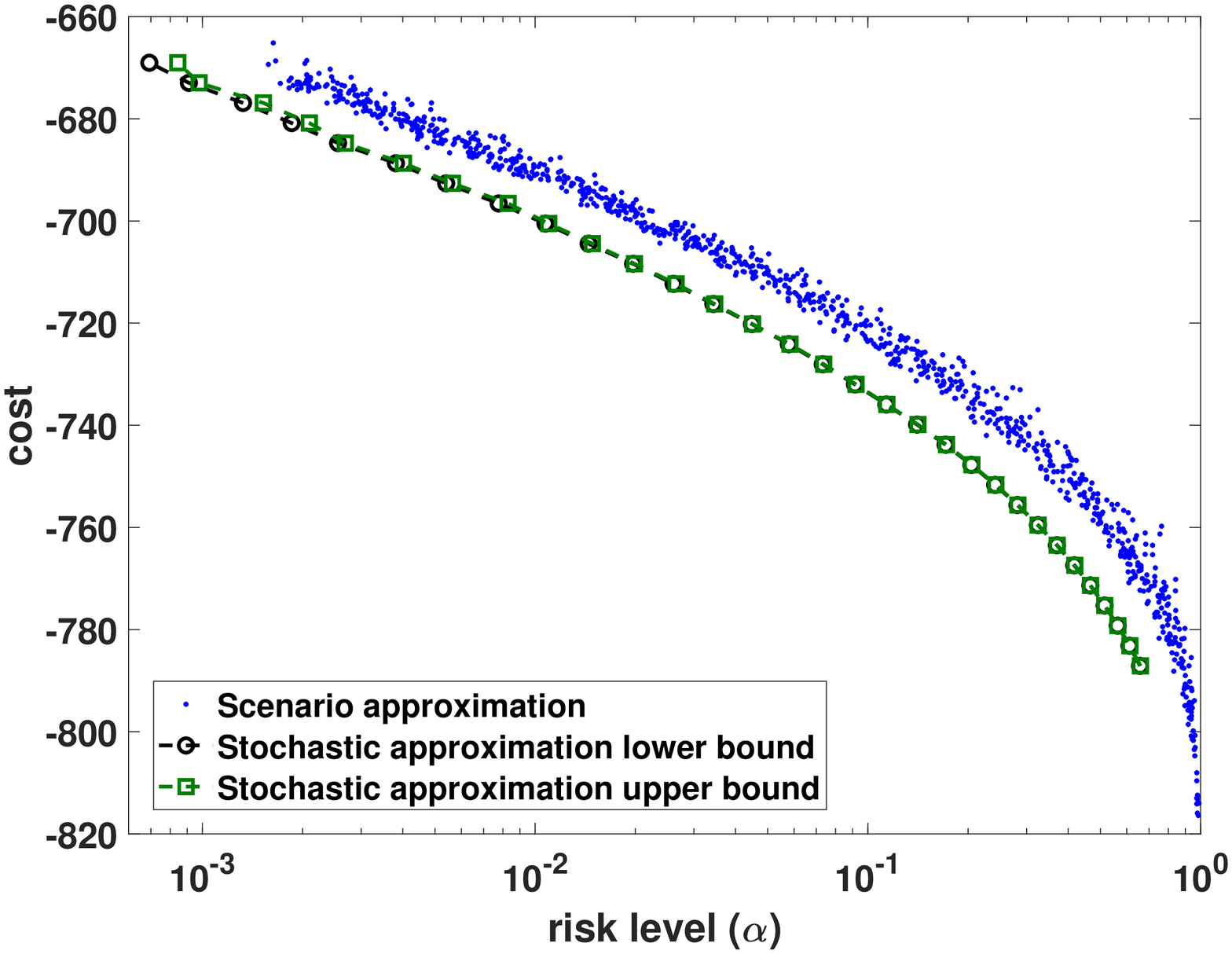}
\end{center}
\end{figure}

\begin{figure}[ht!]
\begin{center}
\caption{Enclosure of the trajectories of the efficient frontier for Case study~\ref{case:resource} generated by ten replicates of the proposed approach.} \label{fig:resource_trajec}
\includegraphics[width=0.8\linewidth]{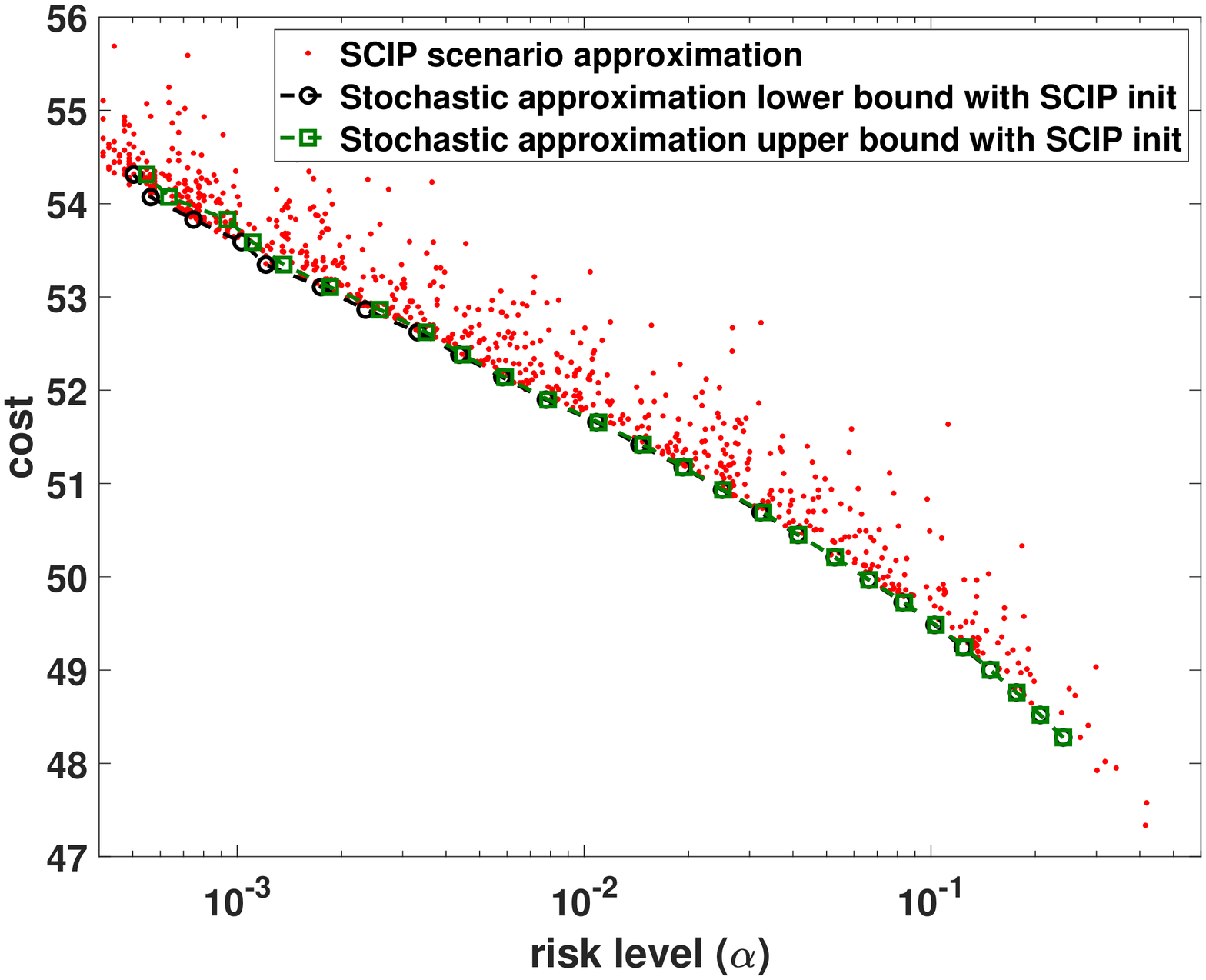}
\end{center}
\end{figure}

\begin{casestudy}
\label{case:normopt2}
We consider Case study~\ref{case:normopt} when the random variables $\xi_{ij} \sim \P := \mathcal{N}(0,1)$ are i.i.d., the number of variables $n = 100$, the number of constraints $m = 100$, and bound $U = 100$.
The EF can be computed analytically in this case, see Section~5.1.1 of~\citet{hong2011sequential}.
Figure~\ref{fig:normopt2} compares a typical EF obtained using our approach against the analytical EF and the solutions generated by the tuned scenario approximation algorithm.
Our proposed approach is able to converge to the analytical EF, whereas the scenario approximation method is only able to determine a suboptimal EF. Our proposal takes $2332$ seconds on average (and a maximum of $2379$ seconds) to approximate the EF using $32$ points, whereas it took the scenario approximation a total of $16007$ seconds to generate its $1000$ points in Figure~\ref{fig:normopt2}.
We note that about $60\%$ of the reported times for our method is spent in generating random numbers because the random variable $\xi$ is high-dimensional.

Figure~\ref{fig:normopt2_trajec} presents an enclosure of the trajectories of the EF generated by the proposed approach
over ten replicates for this case. We find that the risk levels returned by the proposed algorithm do not vary significantly across the different replicates, with the maximum difference in the risk levels across the $32$ points on the EF being a factor of $1.001$.
\end{casestudy}

\begin{figure}[ht!]
\begin{center}
\caption{Comparison of the efficient frontiers for Case study~\ref{case:normopt2}.} \label{fig:normopt2}
\includegraphics[width=0.8\linewidth]{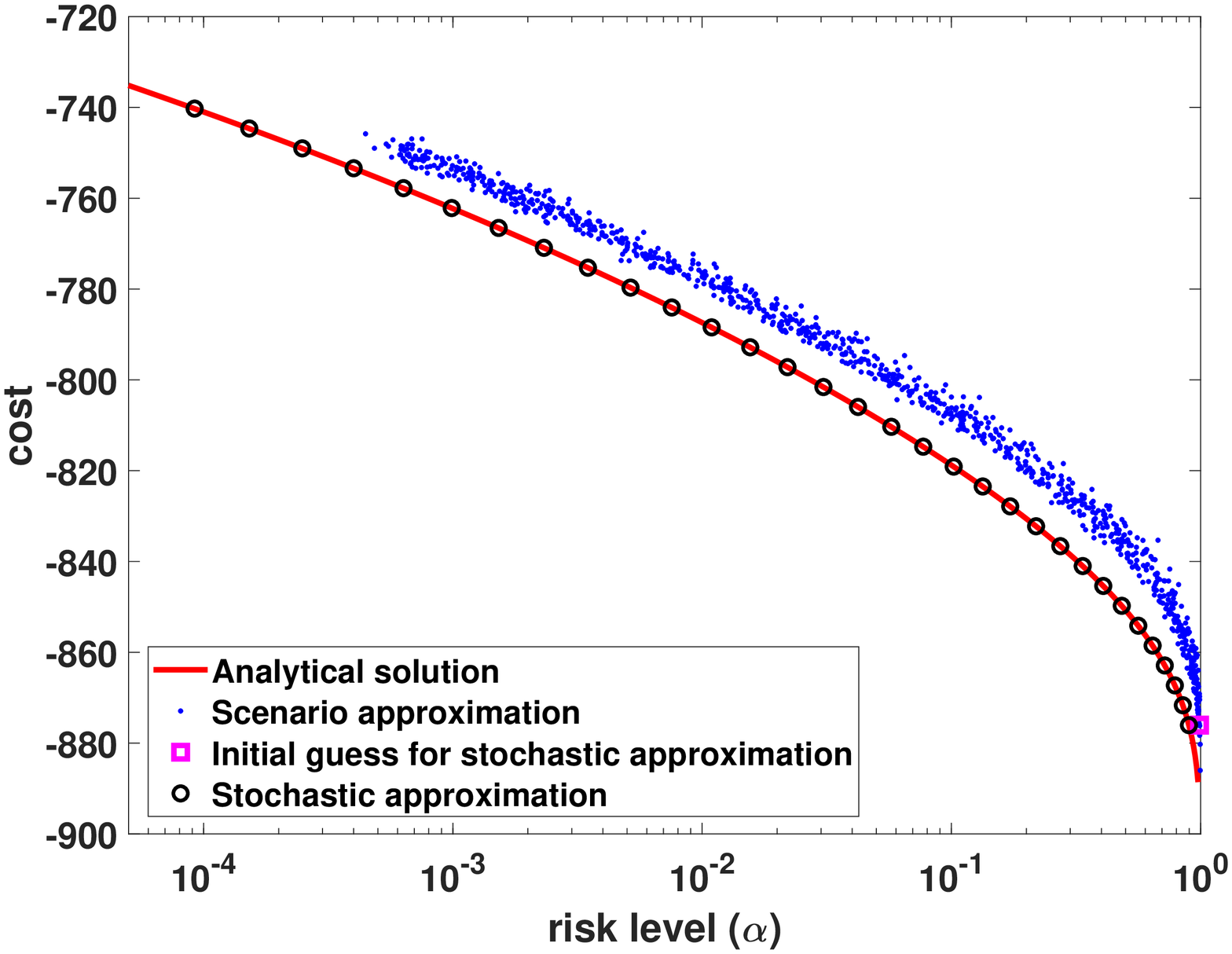}
\end{center}
\end{figure}

\begin{figure}[ht!]
\begin{center}
\caption{Enclosure of the trajectories of the efficient frontier for Case study~\ref{case:normopt2} generated by ten replicates of the proposed approach.} \label{fig:normopt2_trajec}
\includegraphics[width=0.8\linewidth]{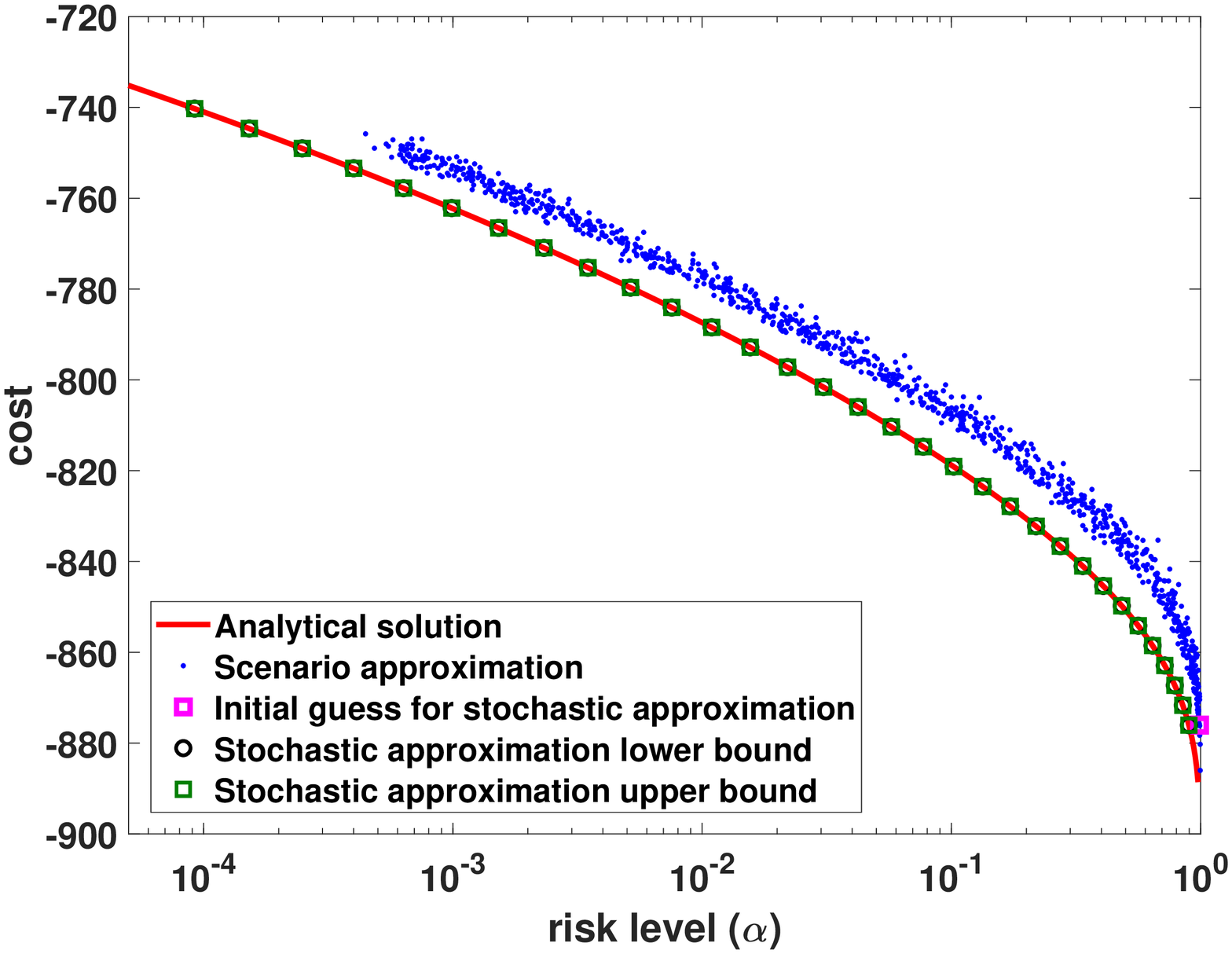}
\end{center}
\end{figure}

%%%%%%%%%%%%%%%%%
\end{document}